\documentclass[11pt, leqno]{article}

\usepackage{amsfonts,amsmath,oldgerm,amssymb,amscd,amsthm,mathrsfs}
\usepackage{bm}
\usepackage[active]{srcltx}
\usepackage[all]{xy}
\usepackage{color}
\usepackage[normalem]{ulem}
\usepackage[margin=1in]{geometry}                
\usepackage{xspace}                              
\usepackage{enumerate}

\newcommand{\longby}[1]{\stackrel{#1}{\longrightarrow}}
\newcommand{\surj}{\twoheadrightarrow}
\newcommand{\inj}{\hookrightarrow}
\newcommand{\lra}{\longrightarrow}
\newcommand{\ol}{\overline}
\newcommand{\wt}{\widetilde}
\newcommand{\wh}{\widehat}

\newcommand{\bosym}{\boldsymbol}

\newtheorem{theorem}{Theorem}[section]

\newtheorem{propn}[theorem]{Proposition}
\newtheorem{cor}[theorem]{Corollary}

\newtheorem{lemma}[theorem]{Lemma}
\theoremstyle{definition}
\newtheorem{defn}[theorem]{Definition}

\newtheorem{conj}[theorem]{Conjecture}
\newtheorem{remark}[theorem]{Remark}               

\newtheorem{thmalph}{Theorem}

\newcommand{\N}{\mathbb N}
\newcommand{\Q}{{\mathbb Q}}
\newcommand{\R}{\mathbb R}                         
\newcommand{\Z}{{\mathbb Z}}
\newcommand{\C}{\mathbb C}
\renewcommand{\AA}{{\mathbb{A}}} 
\newcommand{\FF}{{\mathbb{F}}} 
\newcommand{\II}{\mathbb I}
\newcommand{\JJ}{\mathbb J}
\newcommand{\KK}{\mathbb K}

\newcommand{\TT}{\mathbb T}
\newcommand{\WW}{\mathbb W}

\newcommand{\q}{{\mathfrak q}}
\newcommand{\p}{{\mathfrak p}}
\newcommand{\m}{{\mathfrak m}}

\newcommand{\fr}{\mathfrak{r}}
\newcommand{\fc}{\mathfrak{c}}
\newcommand{\fN}{\mathfrak{N}}

\newcommand{\cA}{{\mathcal A}}                      
\newcommand{\cC}{{\mathcal C}}
\newcommand{\cF}{{\mathcal F}}
\newcommand{\cG}{{\mathcal G}}
\newcommand{\ocG}{\overline{\mathcal G}}
\newcommand{\cH}{{\mathcal H}}
\newcommand{\cL}{{\mathcal L}}

\newcommand{\cO}{{\mathcal O}} 
\newcommand{\cU}{{\mathcal U}}

\newcommand{\sR}{{\mathscr R}}
\newcommand{\sS}{\mathscr S}                        
\newcommand{\sT}{{\mathscr T}}

\newcommand{\gq}{G_{\mathbb Q}}                     
\newcommand{\GQ}{G_{\mathbb Q}}
\newcommand{\gal}{\mathrm{Gal}}

\newcommand{\ad}[1]{\mathrm{Ad}^0{(#1)}}
\newcommand{\sg}[2]{\mathrm{Sel}_{#1}{(#2)}}        
\newcommand{\sgd}[2]{\mathrm{Sel}^\ast_{#1}{(#2)}}  
\newcommand{\sgart}[2]{\mathrm{S}_{#1}{(#2)}}
\newcommand{\Om}[2]{\Omega_{#1/#2}}
\newcommand{\gl}[1]{GL_2{(#1)}}                     
 
\newcommand{\F}{\cF}
\newcommand{\cR}{\mathcal R}
\newcommand{\cRin}{\mathcal R_\infty}
\newcommand{\Ein}{E_\infty}
\newcommand{\veps}{\varepsilon}

\newcommand{\IIG}{\mathbb I[[\cG]]}
\newcommand{\cnl}{\mbox{$CNL_{\mathcal O}$}\xspace} 

\newcommand{\bfi}{{\mathbf{i}}}
\newcommand{\bh}{\mathbf{h}}                        
\newcommand{\bxi}{{\bar{\xi}}}
\newcommand{\inft}[1]{{{#1}_\infty}}
\newcommand{\cy}[1]{{#1}_{cyc}}                     
\newcommand{\cyn}[2]{{#1}_{{#2},cyc}}               
\newcommand{\fH}{\mathfrak{H}}

\newcommand{\he}[1]{\mathbb{H}_{(#1)}}

\newcommand{\Hom}[3]{\mathrm{Hom}_{#1}(#2,#3)}
\newcommand{\homs}[2]{\mathrm{Hom}(#1,#2)}
\newcommand{\homcat}[3]{\mathrm{Hom}_{\mathrm{#1}}(#2,#3)}
\newcommand{\Hone}[2]{\mathrm{H}^1(#1,#2)}
\newcommand{\kone}[1]{K_1\left(#1\right)}
\newcommand{\konep}[1]{K'_1(#1)}
\newcommand{\Log}{\mathrm{Log}}
\newcommand{\Ind}[3]{\mathrm{Ind}^{#1}_{#2}{#3}}
\newcommand{\Indm}[1]{\mathrm{Ind}_M^{\Q}(#1)}
\newcommand{\maptheta}[2]{\Theta^{#1}_{#2}}

\newcommand{\Res}[2]{\mathrm{Res}^{#1}_{#2}}
\newcommand{\Conj}[1]{\mathrm{Conj}(#1)}
\newcommand{\Wh}{\mathrm{Wh}}

\usepackage[OT2,T1]{fontenc}
\DeclareSymbolFont{cyrletters}{OT2}{wncyr}{m}{n}
\DeclareMathSymbol{\Sha}{\mathalpha}{cyrletters}{"58}

\entrymodifiers={+!!<0pt,\fontdimen22\textfont2>}

\makeatletter
\def\imod#1{\allowbreak\mkern10mu({\operator@font mod}\,\,#1)}
\makeatother

\title{Noncommutative Iwasawa theory arising from Hecke algebras}

\author{Chandrakant Aribam}

\date{}

\begin{document}

\maketitle
\abstract{Let $p$ be an odd prime and $f$ be a nearly ordinary Hilbert modular Hecke eigenform defined over a totally real field $F$. Let $\II$ be an irreducible component
of the universal nearly ordinary or locally cyclotomic deformation of the representation of $\gal_F$ that is associated to $f$. We study the deformation rings over a $p$-adic 
Lie extension $\inft{F}$ that contains the cyclotomic $\Z_p$-extension of $F$. More precisely, we prove a control theorem about these rings. 
We introduce a category $\mathfrak{M}_\cH^\II(\cG)$, where $\cG=\gal(\inft{F}/F)$ and $\cH=\gal(\inft{F}/\cy{F})$, which is the category of modules which are torsion 
with respect to a certain Ore set, which generalizes the Ore set introduced by Venjakob. 
For Selmer groups which are in this category, we formulate a 
Main conjecture in the spirit of Noncommutative Iwasawa theory. We then set up a strategy to prove the conjecture by generalizing work of Burns, Kato, Kakde, and Ritter and Weiss.
This requires appropriate generalizations of results due to Oliver and Taylor, and Oliver on Logarithms of certain $K$-groups, which we have presented here.
} 
\tableofcontents
\section{Introduction}\label{intro}
Let $F$ be a totally real number field and $F_\infty$ be a Galois extension such that $\gal(F_\infty/F)$ is a pro-$p$, $p$-adic Lie group $\cG$ for an odd prime $p$. 
We assume that the cyclotomic $\Z_p$-extension $\cy{F}$ is contained in $\inft{F}$.
For every $n\in\N$, let $F_n$ be a finite Galois 
extension of $F$ contained in $F_\infty$ such that $F_n\subset F_{n+1}$ and $F_\infty=\cup_n F_n$. 

Over the totally real field $F$, consider a Hecke eigenform 
$f_0\in S_{\kappa}^{n.ord}(\mathfrak{N},\veps_0;W)$ of weight $\kappa=(0,I)$ (see \S\ref{adelic-hmf} for a precise definition of these weights). 
Let $\rho_0$ be the representation of $\gal(\ol F/F)$ that is associated to $f_0$. As $\gal(F_\infty/F)$
is a pro-$p$, $p$-adic Lie group it is solvable and we can consider the base change $f_n$ of $f_0$ to the totally real field $F_n$. Then the representation 
$\rho_n:=\rho_{f_n}$ is isomorphic to the restriction $\rho_n:=\rho_0\mid_{\gal_{F_n}}$. For each of these representations, we consider the deformations of the residual 
representations $\ol\rho_n$. Under some conditions including the absolute irreducibility of the residual representation $\ol\rho_0$, 
a universal deformation ring $\cR_{F_n}$ exists for $\ol\rho_n$ for all $n$ ( see \S\ref{deformation}). 
In this article, we study the deformation rings over the $p$-adic Lie extension $F_\infty$. This allows us to study the Selmer groups of the adjoint representations $\ad{\rho_0}$
defined over the $p$-adic Lie extension $F_\infty$. In fact, if $\boldsymbol\rho_0$ denotes the deformation of $\rho_0$, then we consider the Selmer groups of the adjoint 
$\ad{\boldsymbol\rho_0}$ along an irreducible component $\II$ of $\cR_0$. Let $\cG:=\gal(\inft{F}/F)$ and $\cH:=\gal(\inft{F}/\cy{F})$. Then there is a natural action of the 
group $\cG$ on these Selmer groups, making these Selmer groups modules over $\II[[\cG]]$.

We formulate a Main conjecture for Selmer groups of $\ad{\bosym{\rho_0}}$ over $\IIG$. A Main conjecture for Galois representations attached to the 
ordinary Hida 
family was also formulated by Barth in his thesis \cite{barth}. Our formulation is slightly different from his formulation. 
All of these conjectures are generalizations of the Main Conjecture of Iwasawa theory in \cite{cfksv} for Galois representations arising from motives which are ordinary 
at a prime $p$. A requirement for this formulation is that the Pontryagin dual of Selmer groups defined over $\inft{F}$ be in the category $\mathfrak{M}_\cH(\cG)$. 
This category $\mathfrak{M}_\cH(\cG)$ consists of finitely generated modules over $\Z_p[[\cG]]$ and torsion 
with respect to a 
certain Ore Set (see Section \ref{non-commmutative}). 
This Ore set is defined to be the set of all the elements $x$ of $\Z_p[[\cG]]$ such that $\Z_p[[\cG]]/x$ is a finitely generated module over $\Z_p[[\cH]]$.
It is conjectured in \cite{cfksv} that the Selmer groups defined over $\inft{F}$ attached to $p$-ordinary Galois representations are in the category $\mathfrak{M}_\cH(\cG)$.

As a generalization, we consider the set $\sS$ which consists of elements $x\in\IIG$ such that $\IIG/x$ is a finitely generated module over $\II[[\cH]]$. Then we consider the
category $\mathfrak{M}_\cH^\IIG$ which consists of finitely generated modules over $\IIG$ which are $\sS$-torsion (see section \ref{non-commmutative} for more details). 
To formulate a noncommutative Main conjecture, we also require that the Selmer group of $\ad{\bosym\rho_0}$ over $\IIG$ is in the 
category $\mathfrak{M}_\cH^\II(\cG)$. We state this as a conjecture and in fact, this is a natural generalization of the conjecture for the Selmer group of 
$\ad{\rho_0}$ (\cite{cfksv}). Very little is known about this conjecture,  
even in the first non-trivial and crucial case, namely the case when $\cG$ is a 2 dimensional $p$-adic Lie group. 
It is interesting to note that, by \cite[Cor 5.5]{cfksv}, the Selmer group of a CM elliptic curve $E$ with ordinary reduction at $p$ is in $\mathfrak{M}_\cH(\cG)$
if its $\mu$-invariant is zero over $\cy{\Q(E[p])}$ .
This turns out to be related to a conjecture of Iwasawa regarding the vanishing of $\mu$-invariant of the field  $\cy{\Q(E[p])}$.
It can also be related to the vanishing of the $\mu$-invariants for CM modular forms of higher weight ( see section \ref{mu-cm}).
\begin{thmalph}(= Theorem \ref{elliptic-CM})
 Let $E$ be a CM elliptic curve defined over $\Q$. Let $p$ be a prime of good ordinary reduction for $E$, and $\rho_E$ be the representation 
 of $\gal_\Q$ attached to the Tate module of $E$. Let $K=\Q(E[p])$ be the field of $p$-torsion points of $E$, $\inft{K}=\Q(E[p^\infty])$, $\cG=\gal(\inft{K}/K)$, and $\cH=\gal(\inft{K}/\cy{K})$. 
 Then, assuming Iwasawa's conjecture 
 on the vanishing of the $\mu$-invariant of $\gal(\cy{L}/\cy{K})$, the dual Selmer group $\sgd{\cy{K}}{{\rho_E}}$ has $\mu$-invariant equal to zero. 
 In other words, Iwasawa's conjecture 
 on the vanishing of the $\mu$-invariant of $\gal(\cy{L}/\cy{K})$ implies that the dual Selmer group $\sgd{\cy{K}}{{\rho_E}}$
  is in the category $\mathfrak{M}_\cH(\cG)$.
\end{thmalph}
This can perhaps be considered as an evidence supporting the $\mathfrak{M}_\cH(\cG)$-conjecture for CM elliptic curves, since Iwasawa's conjecture on the vanishing of $\mu$-invariants is believed to be true. A proof of this theorem rests on the existence of Artin representations of weight one in the
Hida family of the Galois representation of $E$.
In general, we show that the conjecture on 
$\mathfrak{M}_\cH(\cG)$ is true if $\cRin$ is noetherian. This result also ties up with the numerical examples for the prime $p=3$, in \cite{chan-fine}, that Iwasawa's $\mu$-invariant over $\cy{\Q(E[p])}$ is zero, for all CM elliptic curves defined over $\Q$ with good ordinary reduction at the prime $3$, as well as those in \cite{chan-selmer}, that the $\mu$-invariant for
these Selmer groups vanish.


The main conjecture for Selmer group of $\ad{\bosym\rho_0}$ along $\II$ when $F_\infty$ is the cyclotomic $\Z_p$-extension of $F$ was studied by Hida in many papers.
However, for many of the results here we refer to the book \cite{hida-hmf}.

Our aim here is to explore the noncommutative Iwasawa theory for a $p$-adic family of modular forms.
The results of Hida in \cite{hs2} dealt with certain twists of the adjoint representation $\ad{\bosym\rho_0}$, and the Main conjecture formulated here takes
care of all the Artin twists arising from the $p$-adic Lie extension $\gal(\inft{F}/F)$ at once (see the Conjecture \ref{main-conj} below). After formulating the Main conjecture 
for Selmer groups over $\II[[\cG]]$, we also show that a suitable generalization of the strategy of Burns, Kato, Kakde and Hara may be build up to prove the Main conjecture over 
$\II[[\cG]]$. Their strategy has been successfully used to prove the main conjecture over totally real fields and for the trivial Galois representation. 
We also generalize the torsion congruences which played a crucial role in the proof of the noncommutative Main conjecture over totally real fields.
It would have been clear that the methods of Burns and Kato could be generalized to a $p$-adic family of Galois representations. However, there remained many technical details 
to verify in order to accomplish this. These details are in trying to generalize those results of Olivier which have been used by Burns and Kato. We carefully verify and extend only those results, which were used in the methods of these authors. Their results and techniques provide a basis for our proofs based on induction. The strategy can be used in general to show the existence of noncommutative $p$-adic L-functions for a $p$-adic family of Galois representations.

We first show a relation between the conjectures regarding the categories $\mathfrak{M}_\cH(\cG)$ and $\mathfrak{M}_\cH^\II(\cG)$.
\begin{thmalph}[=Theorem \ref{big-torsion-small-torsion}]
Consider the representation $\bosym{\rho}_\II:\gal_F\lra\gl{\II}$ which arises from the irreducible component $\II$ and let $\phi_k:\II\lra\cO$ be a morphism of local 
algebras which give rise to a locally cyclotomic
point $P$ of weight $k$. 
 The dual Selmer group $\sgd{\Ein}{\ad{{\rho}_\II}}$ is $\sS$-torsion if and only if $\sgd{\Ein}{\ad{\rho_P}}$ is $S$-torsion.
\end{thmalph}
\begin{thmalph}[=Theorem \ref{noetherian-mhg}] Let $\cG$ be a pro-$p$, $p$-adic analytic Lie group which is admissible. 
 If $\cRin$ is noetherian, then the dual Selmer group $\sgd{F_\infty}{\ad{\rho_P}\otimes_{W}W^\ast}$ is in the category $\mathfrak M_{\cH}(\cG)$.
\end{thmalph}
Over the cyclotomic $\Z_p$-extension, results of Hida show that the noetherian property of $\cy\cR$ is related to the vanishing of $\mu$-invariant of the dual Selmer group
of $\ad{\rho}$. 
In the spirit of this, we also mention few observations about the adjoint Selmer groups.

Let $\II\cong\cO[[X_1,\cdots,X_r]]$, for some $r$, with $\cO$ unramified over $\Z_p$.
and $\cG$ a $p$-adic Lie group of dimension 1. Let $\Sigma(\cG)$ be any set of rank 1 subquotients of $\cG$ of the form $U^{ab}$ with $U$ an open subgroup of 
$\cG$ that has the following property:
\begin{description}
 \item[($\ast$)] For each Artin representation $\rho$ of $\cG$, there is a finite subset $\{U^{ab}_i:i\in I\}$ of $\Sigma(\cG)$ and for each index 
$i$ an integer $m_i$ and a degree one representation $\rho_i$ of $U^{ab}$ such that there is an isomorphism of virtual representations 
$\rho\cong\sum_{i\in I}m_i.\Ind{\cG}{U_i}{\Ind{U_i}{U_i^{ab}}{\rho_i}}$. 
\end{description}
Let $U^{ab}$ be a subquotient satisfying the above property $(\ast)$, and for any group $G$, let $\II(G):=\II[[G]]$.
Note that we have the following natural homomorphism,
\begin{equation*}
 \konep{\IIG_\sS}\lra \konep{\II(U)_\sS}\lra \konep{\II(U^{ab})_\sS}\lra\II(U^{ab})_\sS^\times\subset Q_\II(U^{ab})^\times.
\end{equation*}
Taking all the $U^{ab}$ in $\Sigma(\cG)$ we get the following homomorphism
\begin{equation*}
 \Theta_{\Sigma(\cG)}:\konep{\IIG}\lra\prod_{U^{ab}\in\Sigma(\cG)}Q_\II(U^{ab})^\times.
\end{equation*}
For any subgroup $P$ of $\ol\cG$, we write $\maptheta{\ol\cG,ab}{P}$ for the following natural composite homomorphism
\begin{equation*}
 \konep{\IIG}\stackrel{\maptheta{\ol\cG}{P}}{\lra}\kone{\II(U_P)}\lra\kone{\II(U_P^{ab})}\cong\II(U_P^{ab})^\times,
\end{equation*}
where the isomorphism is induced by taking determinants over $\II(U_P^{ab})$.
\begin{thmalph}[=Theorem \ref{cong}]
 Let $\Xi\in\konep{\IIG}$ and for all subgroups $P$ of $\ol\cG$, put $\Xi_P:=\maptheta{\ol\cG,ab}{P}(\Xi)\in\II(U_P^{ab})^\times$.
\begin{enumerate}
 \item For all subgroups $P, P'$ of $\ol\cG$ with $[P',P']\leq P\leq P'$, we have
\begin{equation*}
 \mathrm{Nr}_P^{P'}(\Xi_{U_{P'}^{ab}})=\Pi_P^{P'}(\Xi_{U_{P'}^{ab}}).
\end{equation*}
 \item For all subgroups $P$ of $\ol\cG$ and all $g$ in $\ol\cG$ we have $\Xi_{gU_{P}^{ab}g^{-1}}=g\Xi_{U_{P}^{ab}}g^{-1}$.
 \item For every $P\in\ocG$ and $P\neq (1)$, we have
 \begin{equation*}
  \mathrm{ver}_P^{P'}(\Xi_{U_{P'}^{ab}})\equiv \Xi_{U_P^{ab}} \pmod{\sT_{P,P'}} (\mbox{ resp. } \sT_{P,P',\sS} \mbox{ and } \widehat\sT_{P,P'}).
 \end{equation*}
 \item For all $P\in C(\ol\cG)$  we have $\alpha_P(\Xi_{U_{P}^{ab}})\equiv\prod_{P'\in C_P(\ol\cG)}\alpha_{P'}(\Xi_{U_{P'}^{ab}})\pmod{p\sT_P}$.
\end{enumerate}
Conversely, if $\Xi_{U_P^{ab}}\in\II(U_P^{ab})^\times$ for all subgroups $P$ of $\ol\cG$, such that the above congruences hold then there exists an element 
$\Xi\in\konep{\II(\cG)}$ such that $\maptheta{\ol\cG,ab}{P}(\Xi)=\Xi_{U_P^{ab}}\in\II(U_P^{ab})^\times$.
\end{thmalph}
Crucial in the proof is the existence of the following logarithmic map
\begin{equation*}
\konep{\IIG}\stackrel{\mathfrak{L}}{\lra}\II(Z)[\Conj{\ocG}]^\tau. 
\end{equation*}
Further, we  show that the integral logarithm map fits in the following commutative diagram: 
  \begin{equation*}
  \xymatrix{
  1\ar[r] &\mu(\cO)\times\WW\times\cG^{ab}\ar[r]\ar[d]_{=} &\konep{\IIG}\ar[r]^{\!\!\!\!\mathfrak L}\ar[d]^{\Theta^{\ocG}}
                                                                      &\II(Z)[\Conj{\ocG}]^\tau\ar[r]\ar[d]^{\beta^{\ocG}}_\cong &\WW\times\cG^{ab}\ar[r]\ar[d]_{=} &1\\
  1\ar[r] &\mu(\cO)\times\WW\times\cG^{ab}\ar[r]       &\Phi^{\ocG}\ar[r]_{\!\!\!\!\mathcal L}        &\Psi^{\ocG}\ar[r]                    &\WW\times\cG^{ab}\ar[r]       &1,
  }
 \end{equation*}
 where $\mathbb{W}:=(1+p\Z_p)^r$, and $\mathfrak{L}$ and $\mathcal{L}$ are the integral logarithm maps. In Theorem \ref{Theta-iso}, we show that the map $\Theta^{\ocG}$ is an 
 isomorphism and the congruences in the theorem above are derived from this isomorphism. 
\begin{thmalph}[=Theorem \ref{Theta-iso}] The map $\Theta^{\ocG}$ is an isomorphism.
\end{thmalph}
To show this, among many other algebraic results, we need 
a generalization of a classical result of Higman in \cite{higman}, regarding the torsion subgroup of units of group rings. More precisely, we have the following generalization
of Higman's theorem: 
\begin{thmalph}[=Theorem \ref{k-tors}] For any finite $p$-group $G$, we have
\begin{equation*}
 (\kone{\II[G]})_{\mathrm{tors}}\cong\mu_K\times G^{ab}\times S\kone{\II[G]}.
\end{equation*} 
\end{thmalph}
In Section \ref{nearly-ordinary}, we recall Hilbert modular forms, the action of the Hecke algebra on the space of nearly ordinary Hilbert modular forms and the nearly ordinary
Hecke algebra. In section \ref{deformation}, we recall the deformation of two-dimensional representations of the Galois group
$\gal_F$, where $F$ is a totally real field. We then recall the results that relate the nearly ordinary deformation to the nearly ordinary Hilbert modular forms
both in Sections \ref{nearly-ordinary} and \ref{deformation}, in some detail, keeping in mind future applications and also for the convenience of the reader.
In Section \ref{adjoint}, we introduce the Selmer groups of the adjoint representation and then prove a control theorem. We recall that Selmer groups can be viewed as 
K\"ahler differentials. We then study the deformation rings over a $p$-adic Lie extension in section \ref{admissible}. In Section \ref{non-commmutative}, we give a sufficient
condition for Selmer groups to be in the category $\mathfrak{M}_\cH^\II(\cG)$ in terms of the deformation rings. This section also has some results which may be of independent interest. We then present a
noncommutative Main Conjecture for these Selmer groups. 
We continue looking at this category in Section \ref{mu-mhg}, and we first consider those Selmer groups attached to Artin representations of weight one. Then we relate them to Selmer groups of CM elliptic curves.

We also give some results regarding the structure of the deformation rings over the $p$-adic Lie extension $F_\infty$. In Section \ref{k-one}, we extend the 
strategy of Burns, Kato, Kakde, and of Ritter and Weiss to prove the Main Conjecture. This section has some results on $K$-groups and logarithm maps that may be of independent interest. 
We extend some of the results of Oliver, and this required us to generalize results on logarithm maps and computation of certain $K$-groups. 
In addition, we consider a suitable generalization of the $SK_1$-groups in Definition \ref{sk-def2}.
The results
in this section may be used to establish the Main Conjecture of a $p$-adic family of Galois representations arising from motives. In the next section, we show that the 
$p$-adic L-function over $\IIG$ specializes to the $p$-adic L-function for each of the members in the family. 

Section \ref{nearly-ordinary}, where we have recalled the main results regarding Hecke algebras provides a background for our work. It is evident how our results 
here are a generalization of results of Hida. One can begin with a representation of the Galois group over the ring $\II$. Along with this, 
Section \ref{k-one} is the main section where we do the computation of the $K$-groups and it is clear how many of our results are generalizations of those of Burns, 
Kato, Kakde, Ritter and Weiss, and Oliver. This section owes a lot to their works. In addition, the paper \cite{cfksv} of Coates, Fukaya, Kato, Sujatha and Venjakob
has been a strong influence on our work.
\section{Nearly ordinary Hilbert modular Hecke algebra}\label{nearly-ordinary}
\subsection{Adelic Hilbert Modular forms}\label{adelic-hmf}
In this section, we briefly recall the space of Hilbert modular forms, their \emph{Hecke algebra}, its \emph{nearly ordinary part} and the associated \emph{Galois representations} and the \emph{specialization} of a family of nearly ordinary forms to get a nearly ordinary form of a 
given weight. This is only to set the notations, and recall the definition of these terms, and we refer to Hida's book \cite{hida-hmf}, especially Chapter 2, for all the details that we recall below.

Let $F$ be a totally real number field, and $\cO$ denote the ring of integers of $F$. Let $\fN$ denote an integral ideal of $F$. 
Consider the algebraic group $G=\mathrm{Res}_{\cO/\Z}GL(2)$ over $\Z$. Then for each commutative ring $A$, we have $G(A)=GL_2(A\otimes_{\Z}\cO)$.
Let $T_0=\mathbb{G}_{m/\cO}^2$ be the diagonal torus of $GL(2)_{/\cO}$. Then consider 
$T=\mathrm{Res}_{\cO/\Z}\mathbb{G}_{m/\cO}$ and $T_G=\mathrm{Res}_{\cO/\Z}T_0$. Then $T_G$
contains the center $Z$ of $G$. 

Writing $I=\homcat{field}{F}{\ol\Q}$, the group of algebraic characters $X(T_G)=\homcat{alg\, gp}{T_{G/\ol\Q}}{\mathbb{G}_{m/\ol\Q}}$ can be identified 
with $\Z[I]^2$ so that $\kappa=(\kappa_1,\kappa_2)\in\Z[I]^2$ induces the following character on $T_G(\Q)=F^\times\times F^\times$:
\begin{equation*}
 T_G(\Q)\lra\ol\Q^\times: (\xi_1,\xi_2)\mapsto \kappa(\xi_1,\xi_2)=\xi_1^{\kappa_1}\xi_2^{\kappa_2},
\end{equation*}
where $\xi_j^{\kappa_j}=\prod_{\sigma\in I}\sigma(\xi_j)^{\kappa_{j,\sigma}}\in\ol\Q^\times$. Then consider the ``Neben'' characters defined as the triple
\begin{equation*}
 \varepsilon=(\varepsilon_1,\varepsilon_2:T(\widehat\Z)\lra\C^\times,\varepsilon_+:Z(\mathbb{A})/Z(\Q)\lra\C^\times).
\end{equation*}
This is the way in which the Neben characters have been considered in \cite[2.3.2]{hida-hmf} and this is so defined so that the character $\varepsilon_+$ is the 
central character of the automorphic form that corresponds to the Hilbert modular form over $GL_2(\mathbb{A}_F)$. Note that any character 
$\psi:T(\widehat\Z)\lra\C^\times$ which is continuous is of finite order, and we have an ideal $\fc(\psi)$ which is maximal among the integral ideals $\fc$ 
satisfying $\psi(x)=1$ for all $x\in T(\widehat\Z)=\widehat\cO^\times$ with $x-1\in\fc\widehat\cO$. The ideal $\fc(\psi)$ is called the conductor of $\psi$.
Then consider the congruence subgroup
\begin{eqnarray*}
 \widehat\Gamma_0(\fN)&=&\left\lbrace\begin{pmatrix}a & b\\ c &d \end{pmatrix}\in G(\widehat\Z)\mid c\equiv 0\mod\fN\widehat\cO\right\rbrace.
\end{eqnarray*}
Now put $\veps^-=\veps_1\veps_2^{-1}$, and assume that $\veps^-$ factors through ${(\cO/\fN)}^\times$, i.e., $\fc(\veps^-)\supset\fN$.
Then $\veps^-$ and $\veps_2$ induce a continuous character of the compact group $\widehat\Gamma_0(\fN)$ which we also denote by $\veps$ and defined by 
\begin{equation}\label{character-eps}
 \veps:\widehat\Gamma_0(\fN)\lra\C^\times: \begin{pmatrix}a & b\\ c& d\end{pmatrix}\mapsto \veps_2(ad-bc)\veps^-(a_\fN)=\veps_1(ad-bc)(\veps^{-})^{-1}(d_\fN).
\end{equation}

Let $S_\kappa(\fN,\veps;\C)$ denote the space of \emph{Hilbert cusp forms} of level $\fN$ and character $\veps$ which is defined with respect to the congruence subgroup $\widehat\Gamma_0(\fN)$. Each 
$f\in S_\kappa(\fN,\veps;\C)$ satisfies the classical automorphy condition (\cite[2.3.5]{hida-hmf}) for a Hilbert modular form, 
and it admits a Fourier expansion (cf. \cite[Prop 2.26]{hida-hmf}).
By considering the Hilbert cusp forms whose Fourier coefficients are in a subalgebra $A$ of $\C$, one can also consider the subspace
$S_\kappa(\fN,\veps;A)$ of $S_\kappa(\fN,\veps;\C)$. In fact, 
there is an interpretation of $S_\kappa(\fN,\veps;A)$ as the space of $A$-rational global sections of a line bundle on a variety defined over $A$. Then by
the flat base-change theorem (cf. \cite[Lemma 1.10.2]{hida-gme}), we have
\begin{equation*}
 S_\kappa(\fN,\veps;A)\otimes_A\C=S_\kappa(\fN,\veps;\C).
\end{equation*}
Therefore, for any $\ol\Q_p$-algebra $A$, we may define
\begin{equation*}
 S_\kappa(\fN,\veps;A)=S_\kappa(\fN,\veps;\ol\Q)\otimes_{\ol\Q,i_p}A.
\end{equation*}
The formal $q$-expansion of an $A$-rational form $f$ has values 
in the formal monoid algebra 
$A[[q^\xi]]_{\xi\in F}$ of the multiplicative semi-group $F_+$ made up of totally positive elements.
This can be extended to $\C_p$, after normalizing the Fourier coefficients, to obtain the space 
$S_\kappa(\fN,\veps;\C_p)$, and this contains the space $S_\kappa(\fN,\veps;A)$ for any subalgebra $A$ of $\C_p$  
(\cite[2.3.4]{hida-hmf}).

On this space $S_\kappa(\fN,\veps;A)$, there are Hecke operators acting (cf. \cite[2.3.4]{hida-hmf}). The Hecke operators form an  $A$-subalgebra of 
$\mathrm{End}_A(S_\kappa(\fN,\veps;A))$.
We denote the $A$-subalgebra of Hecke 
operators by $h_\kappa(\fN,\veps;A)$. As in the case of elliptic modular forms, the Hecke algebra contains two operators 
$T_p(\varpi_\p)$ and $U_p(\varpi_\p)$ whose action on a common Hecke eigenform gives the $\varpi_\p$-th
Fourier coefficient as an eigenvalue.

Suppose that $W$ is a sufficiently large complete discrete valuation ring inside $\ol\Q_p$ containing the values of $\veps$. We set 
$\mathbf{\Lambda}=W[[\mathbf\Gamma]]$, and let $W[\veps]\subset\ol\Q_p$ be the $W$-subalgebra generated by the values of $\veps$ (over the finite adeles). 
Now consider the \emph{nearly-ordinary} Hecke algebras $h_\kappa^{\mathrm{n.ord}}(\fN\p^r,\veps;W)$.
( It is defined to be the image of Hida's idempotent operator $e$ on $h_\kappa(\fN\p^r,\veps;W)$).
If the weight $\kappa=(0,I)$, then we set (cf. \cite[3.1.5]{hida-hmf}):
\begin{equation*}
 h^{\mathrm{n.ord}}(\fN,\veps;W):=h_{(0,I)}^{\mathrm{n.ord}}(\fN,\veps;W).
\end{equation*}
\begin{defn}\label{arithmetic}Recall that $\kappa=(\kappa_1,\kappa_2)$ induces a character on the torus $T_G$.
If $\kappa_2-\kappa_1+I\geq I$, then the pair $(\kappa,\veps)$ is called \emph{arithmetic}. For an integral domain $\II$ finite and flat over $W[[T_G(\Z_p)]]$,
if a $W$-algebra homomorphism $(P:\II\lra W)\in\mathrm{Spf}(\II)(W)$ coincides with an arithmetic weight on an open subgroup of $T_G(\Z_p)$, then $P$ is referred
to as an \emph{arithmetic} point. The set of arithmetic points of $\mathrm{Spf}(\II)$ with values in $W$ is denoted by $\mathrm{Spf}^{arith}(\II)(W)$.
\end{defn}
Let $\Sigma_p$ denote the set of primes of $F$ lying above $p$. A pair $(\kappa,\veps)$, with $\kappa\in X(T_G)$ such that $\kappa_j,\veps_j$ factors through the 
local norm maps
\begin{equation*}
\begin{split}
 &T(\Z_p)\lra\prod_{\p\mid p}\Z_p^\times \quad\mbox{ defined by }\quad
 N_p((x_\p)_\p)=(N_\p(x_\p))_\p, \mbox{ where } N_\p(x_\p)=N_{F_\p/\Q_p}(x_\p).
 \end{split}
\end{equation*}
is called \emph{locally cyclotomic}.
If $(\kappa,\veps)$ factor through the global norm map $N_{F/\Q}:T(\Z_p)\lra\mathbb{G}_m(\Z_p)=\Z_p^\times$, then $(\kappa,\veps)$ is called \emph{cyclotomic}.

The pair $(\kappa,\veps)$ induces a character $T_G(\widehat\Z)\lra W^\times$ given by 
\begin{equation*}
 \begin{pmatrix} a& 0\\ 0 &d\end{pmatrix}\mapsto \veps_1(a)a_p^{-\kappa_1}\veps_2(d)d_p^{-\kappa_2}.
\end{equation*}
This further induces a $W$-algebra homomorphism $\pi_{\kappa,\veps}:W[[T_G(\Z_p)]]\lra W$.
A $W$-point $P$
of the formal spectrum $\mathrm{Spf}(W[[T_G(\Z_p)]])$ is called \emph{arithmetic} if $P=\mathrm{ker}(\pi_{\kappa,\veps})$ with $\kappa_2-\kappa_1-I\geq0$. 
Similarly, an arithmetic point $P\in\mathrm{Spec}(W[[T_G(\Z_p)]])(W)$ associated with $(\kappa,\veps)$ is called \emph{locally cyclotomic} (resp. \emph{cyclotomic})
if $(\kappa,\veps)$ is locally cyclotomic (resp. cyclotomic). Thus \emph{locally cyclotomic} (resp. \emph{cyclotomic}) points are arithmetic. Let $\II$ be an integral domain
which is an algebra over $W[[T_G(\Z_p)]]$. Then a point $P\in\mathrm{Spec}(\II)(W)$ is said to be \emph{locally cyclotomic} (resp. \emph{cyclotomic}) if the
structure homomorphism $W[[T_G(\Z_p)]]\lra\II$ is locally cyclotomic (resp. cyclotomic). The maps $\pi_{\kappa,\veps}$ are also 
referred to as the \emph{weight} $(\kappa,\veps)$-\emph{ specializations}.

Hida also defines a \emph{locally cylotomic} Hecke algebra $\bh^{\mathrm{n.ord}}_{\mathrm{cyc}}(\fN,\veps;W[[\Gamma_F]])$ such that the locally cyclotomic points defined above also correspond to the forms in the dual of this Hecke algebra.
Roughly, it corresponds to the subspace of $\bh^{\mathrm{n.ord}}(\fN,\veps;W[[\Gamma_F]])$ which is defined over congruence subgroups
where the characters on the torus of $GL_2(\cO_\p)$ factors through the local norm for all primes $\p\mid p$ (\cite[page 165]{hida-hmf}).

The duality between the space of cusp forms and Hecke algebras gives rise to an algebra homomorphism $\lambda_f\in\homcat{\mathrm{alg}}{h_\kappa(\fN,\veps;W)}{W}$
(\cite[Theorem 2.28]{hida-hmf}).
We have the following theorem due to Shimura, Deligne, Serre, Carayol, Ohta, Wiles, Blasius, Rogawski, Taylor, and the version that we have presented here is from 
{\cite[Theorem 2.43]{hida-hmf}}.
\begin{theorem}
 Let $h=h_\kappa(\fN,\veps;W)$ and $k(P)$ denote the field of fractions of $h/P$. Let $\gal_F=\gal(\ol\Q/F)$. Then there exists a continuous semi-simple Galois representation 
 $\rho_f:\gal_F\lra GL_2(k(P))$, such that 
 \begin{enumerate}
  \item $\rho_f$ is unramified outside $p\fN$,
  \item $\mathrm{Tr}(\rho_f(Frob_\mathfrak l))=\lambda(T(\varpi_\mathfrak{l}))$ for $\mathfrak{l}\nmid\p\fN$.
 \end{enumerate}
\end{theorem}
Let $f$ be nearly ordinary at all primes $\p\mid p$. Then we have the following theorem (\cite[Theorem 2.43 (3)]{hida-hmf}).
\begin{theorem}
 Let $f$ be nearly ordinary at all primes $\p\mid p$, i.e., $f|U_p(\varpi_\p)=\lambda(U_p(\varpi_\p))f$ with a $p$-adic unit $\lambda(U_p(\varpi_\p))$. Then 
 \begin{equation*}
  \rho_f\mid_{D_\p}\cong\begin{pmatrix}
                         \epsilon_\p & \ast\\
                         0           & \delta_\p
                        \end{pmatrix}
 \end{equation*}
for the decomposition subgroup $D_\p$ at $\p$, and $\delta_\p([\varpi_\p,F_\p])=\lambda(U_p(\varpi_\p))$, for the local Artin symbol $[\varpi_\p,F_\p]$. 
In particular, $\delta_\p([\varpi_\p,F_\p])=\veps_{1,\p}(u)u^{-\kappa_1}$ for $u\in\cO^\times_\p$.
\end{theorem}
The Hecke algebra $h=h_\kappa(\fN,\veps;W)$ is a direct sum of finitely many local rings. Let $\II$ be a local component of $h$. Then, as in \cite[Theorem 3.29]{hida-mfg}, 
we have the following theorem. 
\begin{theorem}\label{big-rep}
 Let $\II$ be a local ring of the Hecke algebra $h=h_\kappa(\fN,\veps;W)$ such that the residual representation $\bar\rho:\gal_F\lra GL_2(\II/\m_\II)$ is absolutely
 irreducible. Then there exists a unique Galois representation $\rho=\rho_\II:\gal_F\lra GL_2(\II)$ (up to isomorphism) such that
\begin{enumerate}
  \item $\rho_\II$ is unramified outside $p\fN$,
  \item $\mathrm{Tr}(\rho_\II(Frob_\mathfrak l))=T_\II(\varpi_\mathfrak{l})$ for $\mathfrak{l}\nmid\p\fN$,
 \end{enumerate} 
 where $T_\II(\varpi_\mathfrak{l})$ is the image of the Hecke operator $T(\varpi_\mathfrak{l})$ in the component $\II$.
\end{theorem}
From the deformation theoretic point of view, that a Hilbert modular form $f$ arises from the Hecke algebra $\bh^{\mathrm{n.ord}}_{\mathrm{cyc}}(\fN,\veps;W[[\Gamma_F]])$ is equivalent to requiring that for all $\p\mid p$, the local Galois representation
at $\p$, $\rho_f\mid_{D_\p}\cong\begin{pmatrix}
                         \epsilon_\p & \ast\\
                         0           & \delta_\p
                        \end{pmatrix}$
with the two characters $\epsilon_\p$ and $\delta_\p$ of a fixed type outside $\gal(F_\p(\mu_{p^\infty})/F_\p)$ \cite[3.2.1, 3.2.2]{hida-hmf} and the condition (Q4') in \cite[3.2.8]{hida-hmf}. 

As in \cite[3.2.8]{hida-hmf}, we will denote the \emph{ locally cyclotomic }deformation functor that we are considering by 
$\Phi^{cyc}:CNL_W\lra SETS$. 
This functor is defined to be the set of isomorphism classes of representations $\rho:\gal_F\lra\mathrm{GL}_2(A)$ 
satisfying the deformation conditions in \cite[3.2.8]{hida-hmf}. Here it is required to assume that the modular form $f$ has maximal
level among all the forms equivalent to $f$.
By Mazur's theorem, 
(see for example \cite[Theorem 1.46]{hida-hmf}), the functor $\Phi^{cyc}$ is represented by a universal couple $(\cR_F,\bosym{\rho}^{cyc}_F)$, with $\cR_F\in CNL_W$.
\section{Deformation rings and base change}\label{deformation}
\subsection{Base change of deformation rings}
Let $E$ be a totally real field and $f_0\in S_\kappa(\fN,\veps;W)$ with $\kappa=(0,I)$ be a Hilbert modular eigenform 
defined over the totally real field $E$. Let $\rho_E$ denote the Galois representation that is attached to $f_0$. 
We assume that the Galois representation $\rho_E$ satisfies the conditions (h1)-(h4) in \cite[page 185]{hida-hmf}.
Roughly, the first condition on the conductor of a character $\underline\varepsilon$ being prime to $p$ ensures that 
$\bh^{\mathrm{n.ord}}_{\mathrm{cyc}}(\fN,\veps;W[[\Gamma_F]])$ is a torsion-free $W[[\Gamma_F]]$-module of finite type.
The other conditions are 
$\rho_E\mid_{D_\p}$ reducible at 
for all $\p\mid p$ and that the determinant of $\rho_E\mid_{D_\p}$ is a twist by a finite order character of the cyclotomic character, and 
the reduction of $\rho_E$ modulo $\m$ restricted to the inertia at the \emph{``special''} primes is unipotent with infinite image.

Let $F$ be a finite totally real Galois extension of $E$ such that the Galois group $\Delta:=\mbox{Gal}(F/E)$ is a \emph{finite} $p$ group. 
The Galois group $\Delta$ is \emph{not necessarily} cyclic. Since $\Delta$ is a group with order a power of $p$, it is a \emph{solvable} group.
Therefore the Hecke eigenform $f_0\in S_\kappa(\fN,\veps;W)$ admits 
a \emph{unique} base-change lift, say, $f\in S_\kappa(\fN',\veps;W)$, which is defined over the totally real field $F$ (cf. \cite[3.3.3]{hida-hmf}), 
for an appropriate choice of $\fN'$, such that the associated
Galois representation $\rho_f:\gal_F\lra\gl{W}$ is equivalent to the restriction $\rho_E\mid_{\gal_F}$. 
Since $\bar\rho_E$ is absolutely irreducible and $\Delta$ is a finite $p$-group, by \cite[Lemma 1.62]{hida-hmf}, the representation $\bar\rho_F$ is absolutely
irreducible.
Moreover, if $\rho_E$ satisfies
the conditions (h1)-(h4) over $E$, then the restriction $\rho_F$ satisfies the conditions (h1)-(h4) over $F$.
Therefore, from now on, we assume the following: 
\begin{description}
 \item \emph{All the representations of the Galois group satisfy the conditions (h1)-(h4).}
\end{description}
Let $\Phi^{cyc}_E$ and $\Phi^{cyc}_F$ denote the locally cyclotomic deformation for $\bar\rho_E$ and $\bar\rho_F$ respectively.
 Then, $\Phi^{cyc}_F$ is representable. Let $(\cR_F,\bosym{\rho}^{cyc}_F)$ denote the universal deformation ring and the universal deformation
of $\bar\rho_F$.
As $\bosym{\rho}^{cyc}_E\mid_{\gal_F}$ is a deformation in $\Phi^{cyc}_F$, we have a non-trivial algebra homomorphism $\alpha:\cR_F\lra \cR_E$ such that 
$\alpha\circ\bosym{\rho}_F^{cyc}\cong\bosym{\rho}_E^{cyc}\!\!\mid_{\gal_F}$. The morphism $\alpha$ is referred to as the base change morphism. 
We now  describe the action of $\Delta$ on $\Phi_F^{cyc}$ and $\cR_F$. 
\paragraph{$\Delta$-action on $\Phi_F^{cyc}$:}
Let $\sigma\in \Delta$ and $\rho\in\Phi_F^{cyc}(A)$, where $A$ is an $\cO$-algebra in \cnl. 
Consider any $c(\sigma)\in\gl\cO$ such that $c(\sigma)\equiv\bar\rho(\sigma)\pmod{\m_{\cO}}$. 
Then the action of $\sigma$ on $\rho$ is defined by
\begin{equation*}
 \rho^{\sigma}(g):=c(\sigma)^{-1}\rho(\sigma g\sigma^{-1}) c(\sigma)\in\Phi_F^{cyc}(A).
\end{equation*}
The strict equivalence class of $\rho^{\sigma}$ is well defined and depends only upon the class of $\sigma$ in 
$\Delta$. This gives a well-defined action of $\Delta$ on $\Phi_F^{cyc}$.
\paragraph{$\Delta$-action on $\cR_F$:}
For any $\sigma\in \Delta$, since $\ol{\rho^{\sigma}}=\ol{\rho_F}$, $\cR_F$ is the universal deformation ring for $\rho^{\sigma}$.
Therefore, there is a 
morphism $\cR_F\longby{\wt\sigma}\cR_F$ in \cnl. Similarly, we have $\cR_F\longby{\wt{(\sigma^{-1})}}\cR_F$ in \cnl. Composing these two morphisms 
gives the identity, so that $\wt\sigma$ is an automorphism in \cnl. Extending this to $\cO[\Delta]$, we find that $\cR_F$ is a module over 
$\cO[\Delta]$. In other words, the universality of the deformation rings gives rise to 
the automorphisms which define the action.

We consider the augmentation ideal of $\cR_F$ given by
\begin{equation*}
 I_\Delta(\cR_F):=\langle\sigma x-x\mid x\in\cR_F,\sigma\in\Delta\rangle
\end{equation*}                                                             
Let $(\cR_F)_\Delta:=\cR_F/{I_\Delta(\cR_F)}$. Since the determinant of the locally cyclotomic deformation
functor is fixed (see \cite[3.2.8]{hida-hmf}), by \cite[Prop 5.41]{hida-mfg}, we have the following proposition
\begin{propn}\label{base-change}
 $(\cR_F)_\Delta\cong\cR_E$.
\end{propn}
\subsection{Control of K\"ahler differentials}
Let $A$ and $B$ be complete local noetherian algebras. Let $A$ be a $B$-algebra and let $\Om{A}{B}$ denote the $A$-module of  K\"ahler differentials of $A$ over $B$.

The following proposition is 
a generalization to a $p$-adic Lie group $\cG$ of a  result of Hida for cyclotomic $\Z_p$-extension of $E$. 
For a closed subgroup $\cC$ of $\cH$, the action of $\cH$ on the ring $\cR_\cC$ factors through $\cC$.
\begin{propn}\label{control-kahler1}
Consider the nearly-ordinary deformation functor $\Phi_F^{cyc}$.
Let $A$ be a closed $\cO$-subalgebra of $\cRin$ such that $A$ is in \cnl and $\cG$ acts on 
it trivially. Let $B$ be an $A$-algebra in \cnl such that $\cG$ acts on it trivially, and $\pi:\cR_0\lra B$ be an $A$-algebra homomorphism. Then, for any closed normal subgroup
$\cC\subset\cH\subset\cG$, we have
\begin{equation*}
(\Om{\cR_\cC}{A}\widehat\otimes_{\cR_\cC} B)_{\cH}\cong \Om{\cR_\cH}{A}\widehat\otimes_{\cR_\cH} B.
\end{equation*}
\end{propn}
\begin{proof} Let $R:=\cR_\cC\hat\otimes_{\cR_\cC}B$ and $R':=\cR_\cH\hat\otimes_{\cR_\cH}B$.
Consider the following homomorphism of algebras in $\cnl$
\begin{equation*}\xymatrixcolsep{3pc}\xymatrixrowsep{3pc}\xymatrix{
\cR_\cH\hat\otimes_{\Lambda_F'}B\ar[r]^{\alpha\otimes id}\ar@/^3pc/@{->}[urrd]^{\lambda_F'} 
                                   &\cR_\cC\hat\otimes_{\Lambda_F'}B\ar[r]^{\mu_F'} & B,
}
\end{equation*}
where $\lambda_F'=\mu_F'\circ(\alpha\otimes id)$.
We then have,
\begin{equation*}\begin{split}
 ker(\lambda'_F)\otimes_R B\cong\Omega_{{\cR_\cC}/A}\otimes_{\cR_\cC} B \\
 ker(\mu'_F)\otimes_{R'} B\cong\Omega_{{\cR_\cH}/A}\otimes_{\cR_\cH} B.
\end{split}
\end{equation*}
Then we have the following exact sequence:
\begin{equation*}
 0\lra I_\cH(R)\lra ker(\lambda'_F)\longby{\alpha} ker(\mu')\lra 0.
\end{equation*}
Tensoring with $B$ over $R$ and writing $J:=I_\cH(R)$, we get another exact sequence:
\begin{equation*}
 (J/J^2)\otimes_R B=J\otimes_R B\longby{i}\Omega_{\cR_\cC/A}\otimes_{\cR_\cC} B\lra\Omega_{\cR_\cH/A}\otimes_{\cR_\cH} B\lra 0.
\end{equation*}
We claim that $im(i)=I_\cH(ker(\lambda_F')\otimes_{\cR_\cC} B)$.
Consider the map $j:B\lra R$, given by $j(b)=1\otimes b$, and let $B'=im(j)$. Then 
$\lambda'_{F}\circ j=id$,
and $\cH/\cC$ acts trivially on the image $j(B)$. For any $y\in R$, the element $x:=y-j\circ\lambda'_F(y)\in ker{\lambda'_F}$. As $\cH$ acts trivially on $j(B)$, we have
$(\sigma-1)(y)=(\sigma-1)(x)$. Therefore, $(\sigma-1)R=(\sigma-1)ker(\lambda'_F)$.
Since $\sigma$ is a $B'$-algebra automorphism of $R$
and $J/J^2$ is a $B'$-module, therefore
\begin{equation*}
 y(\sigma-1)y'\equiv (\sigma-1)yy' \pmod{J^2}, \mbox{ where }y, y'\in J.
\end{equation*}
We therefore have a $B$-linear map $\sigma-1:ker(\lambda_F')\lra R$. Consider an element $a(\sigma-1)b$ in the image of $J$. Then $a(\sigma-1)b=(\sigma-1)ab\pmod{J^2}$, and 
$(\sigma-1)ab\in(\sigma-1)R=(\sigma-1)ker(\lambda_F')$. As $J$ is generated by elements of the form $a(\sigma-1)b$, it follows that 
$im(i)=I_\cH(ker(\lambda_F')\otimes_{\cR_\cC} B)$.
\end{proof}
\begin{cor}\label{control-kahler2}
 Let $A_\infty$ be an $\cO$-algebra with a continuous action of $\cG$ which is a pro-object in $\cnl$. Suppose that $\cRin$ has a structure of
$A_\infty$-algebra and that the $\cG$-action on $A_\infty$ and $\cRin$ are compatible. Thus $\cR_\cC$ is an $A_\cC$-algebra for 
$A_\cC=(A_\infty)_\cC$. Let $B$ be an algebra in $\cnl$ and $\pi:\cR\lra B$ be an $A_\infty$-algebra homomorphism. Then, for any closed subgroups 
$\cC\subset\cH\subset\cG$, we have:
\begin{equation*}
(\Om{\cR_\cC}{A_\cC}\otimes_{\cR_\cC} B)_{\cH}\cong \Om{\cR_\cH}{A_\cH}\otimes_{\cR_\cH} B.
\end{equation*}
\end{cor}
\begin{proof}
By the above proposition, we have,
\begin{equation*}\begin{split}
(\Om{\cR_\cC}{A}\widehat\otimes_{\cR_\cC} B)_{\cH}\cong \Om{\cR_\cH}{A}\widehat\otimes_{\cR_\cH} B,\\
(\Om{A_\cC}{\cO}\widehat\otimes_{A_\cC} B)_{\cH}\cong \Om{A_\cH}{\cO}\widehat\otimes_{A_\cH} B.
\end{split}
\end{equation*}
These isomorphisms give rise to the following commutative diagram:
\begin{equation*}
\xymatrixcolsep{4pc}\xymatrixrowsep{4pc}
\xymatrix{
(\Om{A_\infty}{\cO}\widehat\otimes_{A_\infty} B)_\cH\ar[r]\ar[d]^\cong   
                        &(\Om{\cRin}{A}\widehat\otimes_{\cRin} B)_{\cH}\ar[r]\ar[d]^\cong  
                                                       &(\Om{\cRin}{A_\infty}\widehat\otimes_{\cRin} B)_{\cH}\ar[r]\ar[d]   
                                                                                                                    & 0          \\
\Om{A_\cH}{\cO}\widehat\otimes_{A_\cH} B \ar[r]   
                       &\Om{\cR_\cH}{\cO}\widehat\otimes_{A_\cH} B\ar[r]
                                                      &\Om{\cR_\cH}{A_\cH}\widehat\otimes_{\cR_\cH} B\ar[r]             & 0.
}
\end{equation*}
Hence
\begin{equation*}
(\Om{\cR_\cC}{A_\cC}\otimes_{\cR_\cC} B)_{\cH}\cong \Om{\cR_\cH}{A_\cH}\otimes_{\cR_\cH} B.
\end{equation*}
\end{proof}
\section{Adjoint Selmer groups and K\"ahler differentials}\label{adjoint}
\subsection{Selmer groups}
Let $p$ be an odd prime. We fix an algebraic closure $\bar\Q$ and embeddings $\Q\inj\bar\Q$ and $\Q\inj\bar\Q_l$ for every prime $l$. 
For a prime $p$ in $\Q$, let $D_p$ denote the decomposition group under this embedding. For a prime $\p$ in a finite extension $F$ of $\Q$, let
$D_\p$ denote the decomposition group at $\p$ defined by the above embedding.
Let $\cO$ be a finite extension of $\Z_p$ and \cnl denote the category of \emph{complete noetherian local rings which are $\cO$-algebras.} 

We recall the definition of the \emph{adjoint representation} associated to a $2$-dimensional Galois representation. 
Let $\II\in\cnl$ and $\mathbb M$ be the quotient field of $\II$. Consider a two-dimensional representation $\bosym\rho:\gal_F\lra\gl{\II}$ as in Theorem \ref{big-rep}.
Let $\mathbb L:=\mathbb I^2$. Then $\rho$ induces an action of $\gq$ on $M_2(\II)$, the ring of $2\times 2$-matrices over $\II$, by conjugation, i. e., 
$\sigma(x):= \bosym\rho(\sigma)x\bosym\rho(\sigma)^{-1}$.
Then the \emph{adjoint representation} is defined by 
\begin{equation*}
\ad{\bosym\rho}:=\{\eta\in \mbox{End}_\II(\mathbb L)\mid Trace(\eta)=0 \}.
\end{equation*}
It is easy to see that this is a 3-dimensional representation of $\gal_F$.
\begin{defn}Let $D_\p$ denote the decomposition group at a prime $\p$ dividing $p$. Then the representation $\rho$ is said to be \emph{nearly ordinary} at 
$\p$, if there is a 
two-step filtration of $\mathbb L$ given by
\begin{equation}\label{filtration}
\mathbb{L} \supset {\cF}_\p^+{\mathbb L} \supset {0}
\end{equation}
as $D_\p$-modules, such that  $\cF_\p^+\mathbb L$ is \emph{free} of rank \emph{one} over $\II$. 
\end{defn}
If $\bosym\rho$ is nearly ordinary, then it induces on $\ad{\bosym\rho}$ the following three-step filtration stable under $D_\p$: 
\begin{equation*}
\ad{\bosym\rho}\supset \cF_\p^-\ad{\bosym\rho}\supset \cF_\p^+\ad{\bosym\rho}\supset 0
\end{equation*}
where
\begin{equation*}\begin{split}
\cF_\p^-\ad{\bosym\rho}=&\{\eta\in\ad{\bosym\rho}\mid \eta(\cF_\p^+\mathbb L)\subset \cF_\p^+\mathbb L\}, \mbox{ and }\\
\cF_\p^+\ad{\bosym\rho}=&\{\eta\in\ad{\bosym\rho}\mid \eta(\cF_\p^+\mathbb L)=0\}.
\end{split}
\end{equation*}
In terms of matrices, if we choose a basis of $\mathbb L$ containing a generator of $\cF_p^+ \ad{\bosym\rho}$ and identify
$\mbox{End}_\II(\mathbb L)$ with $M_2(\II)$ using this basis, then $\cF_p^-\ad{\bosym\rho}$ is made up of \emph{upper triangular matrices with trace zero.} On 
the other hand, $\cF_p^+\ad{\bosym\rho}$ is made up of \emph{upper nilpotent matrices.}

Recall that $\mathbb M$ is the quotient field of $\II$. We put $\mathbb V:=\mathbb L\otimes_{\II}\mathbb M$ and 
$\mathbb A:=\mathbb V/\mathbb L$.
Let $\ad{\mathbb V}:=\ad{\bosym\rho}\otimes\mathbb M$ and $\ad{\mathbb A}:=\ad{\mathbb V}/\ad{\bosym\rho}$.

The above filtration allows us to define the following \emph{local} conditions:
\begin{equation}\label{loc-con}
\begin{split}
\mathscr L(F_\q)=  \begin{cases}  ker\left[\Hone{F_{\q}}{\AA}\lra \Hone{F_{\q}}{\AA}/\cF^+_{\q}\ad{\AA}\right]
\mbox{\quad for\quad} \q\mid p,\\
 ker\left[\Hone{F_{\q}}{\AA} \lra \Hone{F_{\q}}{\AA}\right]\quad\mbox{for}\quad \q\nmid p.\end{cases}
\end{split} 
\end{equation}
\begin{defn}
The Selmer group of $\ad{\bosym\rho}$ over $F$ is defined by 
\begin{equation*}
 \sg{F}{\ad{\bosym\rho}}:=ker \left[\Hone{F^\Sigma/F}{\AA}\lra\prod_{\q}\dfrac{\Hone{F_{\q}}{\AA}}{\mathscr L(F_\q)}\right].
\end{equation*} 
\end{defn}
Let $\Sigma$ be any finite set of primes of $F$ containing the primes above $p$, the infinite primes and the primes ramified in $F$. Let $F^\Sigma$ 
denote the maximal extension of $F$ that is unramified outside $F$. 
Suppose that $\inft{F}$ is any pro-$p$, $p$-adic Lie extension of $F$ that is contained in $F^\Sigma$ and 
$\inft{F}_{,\q}:=\varinjlim_n F_{n,\q_n}$, where $F_n$ are finite extensions of $F$ and $\{\q_n\}$ is a compatible sequence of primes. Then
by restriction we define local conditions $\mathscr L(F_{\infty,\q})$as in \eqref{loc-con}.
\begin{defn}
 We define the \emph{Selmer group} of $\ad{\bosym\rho}$ over $\inft{F}$ by
\begin{equation*}
 \sg{\inft{F}}{\ad{\bosym\rho}}:=ker \left[\Hone{F^\Sigma/\inft{F}}{\ad{\AA}}\lra\prod_{\q\mid\Sigma}
\dfrac{\Hone{F_{\infty,\q}}{\AA}}{\mathscr L(F_{\infty,\q})}\right].
\end{equation*}
\end{defn}
There is an action of the Galois group $\cG:=\mathrm{Gal}(\inft{F}/F)$ on $\sg{\inft{F}}{\ad{\bosym\rho}}$ via \emph{conjugation:} 
if $[c]\in\Hone{F^\Sigma/\inft{F}}{A}$ is any cocycle class and $g\in\cG$, then the action is given by
$(g\ast c)(\sigma):=\tilde g c(\tilde g^{-1}\sigma\tilde g)$, for a lift $\tilde g$ of $g$ to $\gal_{{F}}$.
%
%
%
\subsection{Selmer complexes}
Let $U:=\mathrm{Spec}(\cO_F)\backslash\Sigma$. 
Let $(F_n)_{n\in\N}$ be a sequence of finite extension of $F$ in $\inft{F}$ with $F_n\subset F_{n+1}$ and $\inft{F}=\cup_n F_n$. 
We consider the following representations, considered as a pro-\'etale sheaf:
\begin{enumerate}
 \item $\mathbf{T}:=\varprojlim_{n}(f_n)_\ast(f_n)^\ast(\ad{\bosym\rho})$, where 
$f_n:\mathrm{Spec}(F_n)\lra\mathrm{Spec}(F)$ are natural maps.
 \item $\mathbf T_\p^0:=\varprojlim_{n}(f_{n,\p})_\ast(f_{n,\p})^\ast(\cF^+_{\p}\ad{\bosym\rho})$, for each prime $\p|p$, under the natural maps
$f_{n,\p}:\mathrm{Spec}(F_{n,\p})\lra\mathrm{Spec}(F_\p)$.
\end{enumerate}
Consider the inhomogenous cochains $C(\gal(F^\Sigma/F),\mathbf T)$ and for every prime $v$ of $\cO_F$, the cochains $C(D_v,\mathbf T)$.
Let $C(U,\mathbf T):=C(\gal(F^\Sigma/F),\mathbf T)$ 
and $C(F_v,\mathbf T):=C(D_v,\mathbf T)$. 
We recall that we have fixed algebraic closures $\ol F\inj\ol{F_v}$, for all places $v$ of $F$. We then define the following
complexes:
\begin{defn} 
 \begin{enumerate}
  \item The complex $C_c(U,\mathbf T)$ is defined to be the mapping fiber of the map:
  \begin{equation*}
   C(U,\mathbf T)\lra\bigoplus_{v\in\Sigma}C(F_v,\mathbf T).
  \end{equation*}
 \end{enumerate}
\end{defn}
\begin{defn}
 Let $\ell\neq p$, be any prime of $\Q$, and $u$ be any prime above $\ell$ in $F$. 
 Then we consider the subcomplex $C_f(F_u,\mathbf T)$ of $C(F_u,\mathbf T)$ defined by:
 \begin{equation*}
  C^i_f(F_u,\mathbf T)=\begin{cases}
                C^0(F_u,\mathbf T), & \mbox{ if } i= 0\\
                \mathrm{ker}\left(((C^1(F_u,\mathbf T))_{d=0}\lra \Hone{F_u^{ur}}{\mathbf T}  \right), &\mbox{ if } i=1\\
                0, &\mbox{ if } i\neq 0,1.
               \end{cases}
 \end{equation*}
\end{defn}
\begin{remark}
 In the derived category, we have canonical isomorphisms:
 \begin{equation*}
  C_f(F_u,\mathbf T)\cong C(F_u^{ur}/F_u,\mathbf T^{I_u})\cong [1-\phi:\mathbf T^{I_u}\lra \mathbf T^{I_u}], 
 \end{equation*}
 where $I_u$ is the inertia subgroup at $u$, and $\phi$ denotes the geometric Frobenius, and the last complex lives in degrees 0 and 1.
\end{remark}
\begin{defn}
 The imprimitive Selmer complex $SC(U,\mathbf T,\{\mathbf T_\p^0\}_{\p\mid p})$ is defined to be the mapping fiber of the map
 \begin{equation*}
  C(U,\mathbf T)\lra \bigoplus_{v\mid p}C(F_v,\mathbf T/\mathbf T_v^0)\oplus\bigoplus_{v\in\Sigma\backslash\{p\}}C(F_v,\mathbf T).
 \end{equation*}
\end{defn}
For more properties of these Selmer complexes, see \cite[Section 3]{barth}. These properties are generalizations of the corresponding properties
for Selmer complexes for Galois representations defined over local fields. The following connection between Selmer groups and Selmer complexes
is crucial. We denote the Pontryagin dual of the primitive and imprimitive Selmer groups by $\sgd{F}{\mathbf T}$. 
Recall that $\gal(\inft{F}/F)$ is a $p$-adic Lie group which is denoted by $\cG$. For every prime $v$ of $F$, recall that $D_v$ denotes the 
decomposition subgroup at $v$. Then we have a natural map $D_v\lra\cG$, whose kernel is denoted by $\cG(v)$ and image is denoted by $\cG_v$.
\begin{propn}[{\cite[Prop 3.24]{barth}}]\label{selmer-dual-complex-sequence}
We have the following exact sequences of $\II$-modules:
\begin{equation*}
\begin{split}
 & 0\lra \sgd{F}{\mathbf T}\lra\mathrm{H}^2(SC(U,\mathbf T,\{\mathbf T_\p^0\}_{\p\mid p}))
                                         \lra\bigoplus_{v\mid p}\II[[\cG]]\otimes_{\cO[[\cG_v]]}(\mathbf T_v^0(-1))_{\cG(v)}\\
 & \lra \mathbf T(-1)_{\gal_{\inft{F}}}\lra \mathrm{H}^3(SC(U,\mathbf T,\{\mathbf T_\p^0\}_{\p\mid p}))\lra 0.
 \end{split}
\end{equation*}
Here $\sgd{F}{\mathbf T}\cong\sgd{\inft{F}}{\ad{\bosym\rho}}$.
\end{propn}


The Selmer group attached to the adjoint representation is related to K\"ahler differentials as follows ( see \cite{mt} or \cite{ht}).

Consider the representation $\bar\rho_F$ and the locally cyclotomic deformation functor $\Phi_F^{cyc}$. As in section \ref{deformation}, consider the representation $\bar\rho_F$ that is attached to a Hilbert modular eigenform $f_0\in S_{\kappa}(\fN,\veps;W)$, with $\kappa=(0,I)$, and also satisfying
the conditions (h1)-(h4).
\begin{theorem}\cite[{Prop 3.87}]{hida-hmf}
 Let $\Phi_F^{cyc}$ be the locally cyclotomic deformation functor of $\bar\rho_F$. Suppose that $\Phi_F^{cyc}$ is represented by the 
universal deformation ring $\cR_F$ and $\bosym{\rho}_F$ is
the representation of $\mathrm{Gal}_F$ into $\gl{\cR_F}$. Then for any $A\in\cnl$ and $\wt\rho\in\Phi_F^{cyc}(A)$,  
with $\varphi$ denoting the morphism $\cR_F\lra A$, there exists a canonical isomorphism
\begin{equation*}
 \sgd{F}{\ad{\wt\rho_F}\otimes_A A^\ast}\cong\Om{\cR_F}{W[[\Gamma_F]]}\otimes_{\cR_F,\varphi}A.
\end{equation*}
In particular,
\begin{equation*}
 \sgd{F}{\ad{\bosym{\rho}_F}}\cong\Om{\cR_F}{W[[\Gamma_F]]}.
\end{equation*}
\end{theorem}

\section{Deformation rings in an \emph{admissible} tower $\Ein/E$}\label{admissible}
\subsection{Admissible $p$-adic Lie extension}  
In this section, we study the deformation rings of the functor $\Phi_F^{cyc}$,
when $F$ varies over finite Galois subextensions of an \emph{admissible $p$-adic Lie extension} $\Ein$ over $E$, whose definition
we recall below.
\begin{defn}
 An  \emph{admissible $p$-adic Lie extension} $\Ein$ of $E$ is a Galois extension of $E$ such that (1) $\Ein/E$ is unramified outside a finite set of 
primes of $E$; (2) $\Ein$ is totally real; (3) $\Ein$ is a $p$-adic Lie extension; and (4) $\cy E\subset\Ein$.
\end{defn}
Let $\Ein$ be an \emph{admissible } $p$-adic Lie extension. Let $\Ein:=\cup_n E_n$, where $E_n$ is a finite Galois extension of $E$ for every $n$.
Consider the functor $\Phi_{E_n}^{cyc}$ and let $\cR_{E_n}$ denote the universal deformation ring over $E_n$. Since the extension $E_n/E$ is a pro-$p$
extension, all the deformation conditions are satisfied for the restriction.
Then for every $n$, we have the base change morphisms $\cR_{E_n}\lra\cR_E$. Consider the projective limit
\begin{equation*}
\cR_\infty:=\varprojlim_n\cR_{E_n}.
\end{equation*}
The action of $\Delta_n:=G(E_n/E)$ on $\cR_{E_n}$ induces an action of $\cG$ on $\cRin$. 
Moreover, for finite Galois extensions $F_m$ inside $\cy{E}$, we also define 
\begin{equation*}
\cR_{cyc}:=\varprojlim_m\cR_{F_m}.
\end{equation*}
\begin{propn}
Let $\cH:=\gal(\Ein/E_{cyc})$. Then $\cH$ acts on $\cRin$, and we have a morphism of rings $(\cRin)_\cH \lra \cR_{{cyc}}$,
which is an isomorphism of algebras. 
\end{propn}
\begin{proof}
This follows from the the base-change isomorphism Prop \ref{base-change}. 
\end{proof}
\begin{cor}\label{noetherian}
If $\cRin$ is noetherian ring, 
then  $\cR_{cyc}$ is a noetherian ring.
\end{cor}
\begin{remark}
 For the deformation rings of the nearly ordinary functor and the fixed determinant, analogous results over $\inft{E}/E$ can be proven using \cite[Cor 3.2]{hs2}. 
\end{remark}
\subsection{Example}
Let $E$ be a totally real field, and let $f$ be a nearly $p$-ordinary Hilbert modular eigenform. Let $\rho$ be the representation of 
$\gal_E$ that is associated to $f$. Let $\bosym{\rho}$ be the nearly ordinary deformation for $\rho$.
Let $\Ein$ be a $p$-adic Lie extension of $E$ containing the cyclotomic $\Z_p$-extension $\cy E$. We assume that $\cy{E}$ is totally ramified at primes above $p$.
Let $\{E_n\}_n$ be an increasing sequence of subfields with $E_0=E$ such that $\Ein=\cup_n E_n$.
Let $\cR_{E_n}^{n.ord}$ be the universal \emph{nearly ordinary deformation ring} and $h_{E_n}^{n.ord}$ be the \emph{nearly ordinary Hecke algebra}, for the Galois representation
$\rho$ restricted to $\gal_{E_n}$.

Let $\p$ denote a prime of $E$ lying above $p$,
and we also denote by $\p$ the unique prime of $F_n:=E_n\cap\cy E$ above $p$. 
We consider the universal nearly ordinary representation $\rho:\gal_E\lra\gl{\cR_{E_0}^{n.ord}}$. Then restricted to the decomposition subgroup $D_{0,\p}$ at a prime lying over 
$p$, we have
\begin{equation*}
 \rho\mid_{D_{0,\p}}\cong\begin{pmatrix}\wt{\epsilon}_{\p} &*\\ 0 &\wt{\delta}_{\p}	\end{pmatrix}, 
 \mbox{ with } \wt\delta_{\p}\equiv\bar\delta_\p\pmod{\m_0} \mbox{ in } D_{0,\p},
\end{equation*}
where $\m_0$ is the maximal ideal of $\cR_{E_0}^{n.ord}$.

Let $\rho_n:=\rho\mid_{\gal_{F_n}}$.
We also denote the unique prime of $F_n$ lying above $\p$ by $\p$.
Let $D_{n,\p}$ denote the decomposition group at the prime $\p$. Then 
\begin{equation*}
\rho_n|_{D_{n,\p}}\cong\begin{pmatrix}\wt{\epsilon}_{n,\p} &*\\ 0 &\wt{\delta}_{n,\p}	\end{pmatrix}, 
\mbox{ with } \wt\delta_{n,\p}\equiv\bar\delta_\p\pmod{\m_n} \mbox{ in } D_{n,\p}.
\end{equation*}
Let $\wt\delta_{cyc,\p}$ be the restriction of $\wt\delta_\p$ to $D_{cyc,\p}:=\gal_{E_{cyc,\p}}$, and let $\Lambda_{cyc}$ be the projective limit of the universal deformation 
rings for $\delta_{n,\p}=\wt\delta_{\p}\mid_{D_{n,\p}}$. Let $\cR_{F_n}^{n.ord}$ be the universal nearly ordinary deformation ring for the representation $\rho_n$.
Let $\wt\Lambda_n$ be the subalgebra of $\cR_{F_n}^{n.ord}$ topologically 
generated by the image of $\wt\delta_{cyc,\p}$ over $\cO$. Assume that the order of $\wt\epsilon_{n,\p}\mod\m_n$ is prime to $p$.
As $\wt\delta_{cyc,\p}$ restricted to the $p$-wild inertia subgroup factors through $\Gamma_{n,\p}$, and the tame part has values
in $\cO$, therefore $\wt\Lambda_n\cong\cO[[\wt\delta_{n,\p}(Frob_\p)-\delta_{n,\p}(Frob_\p)]]$ inside $\cR_{F_n}^{n.ord}$, for the Frobenius element
  $Frob_\p$ in $D_{cyc,\p}$. Let $\wt\delta_0(Frob_\p)=a(\p)\in\cR_0^{n.ord}$. Recall the following result of Hida.
\begin{propn}\label{inf-noetherian}{\cite{hida-mfg}}
Let $\wt\Lambda_0=\cO[[x_\p]]_{\p|p}$ (under the normalization $\gamma_\p\mapsto 1+x_\p$), and 
 \begin{equation*}
 Jac_{\wt\Lambda_0}:=\det\left(\dfrac{\partial a(\p)}{\partial x_{\p'}}\right)_{\p,\p'\mid p} \in\wt\Lambda_0^\times.
 \end{equation*}
  Then ${\Lambda_{cyc}}=\wt\Lambda_0$.
\end{propn}
Let $\wt\Lambda_\infty$ be the projective limit of the universal deformation rings for $\wt\delta\mid_{\gal_{E_n,\p}}$ for all primes $\p$ of $E_n$ lying above $p$. 
Then, $(\Om{\wt\Lambda_{\infty}}{\cO}\widehat\otimes\cO)_\cH\cong\Om{\wt\Lambda_{cyc}}{\cO}\widehat\otimes\cO\cong\cO^{\oplus r}$, where $r=$ number of primes of $E$ 
lying above $p$. By Nakayama lemma, we have the following result.
\begin{cor}
Under the same assumptions as in the previous theorem, the dual Selmer group $\Om{\wt\Lambda_{\infty}}{\cO}\widehat\otimes\cO$ is a finitely generated module over $\cO[[\cH]]$.
\end{cor}

For any number field $L$, let $S_L$ be the set of primes of $L$ lying above $p$ and $D_{\p}^{ab,p}$ be the maximal $p$-profinite
abelian quotient of the decomposition subgroup $D_\p$ at $\p$ in $\gal_{L}$. We write $D_L=\prod_{\p\in S_L}D_{\p}^{ab,p}$ 
and $I_L=\prod_{\p\in S_L}I_{\p}^{ab,p}$, where $I_{\p}^{ab,p}$ is the image of inertia subgroup in $D_{\p}^{ab,p}$. 
In \cite[Theorem 6.3]{hs2}, Hida gave the structure of Hecke algebras along the 
cyclotomic tower of a number field. One can have a slightly more general result as follows.
\begin{theorem}\label{trivial-zeros}
Let $\II$ be an irreducible component of $\cR_E^{n.ord}$. 
Let $s=\mid S_E\mid$ be the number of primes of $E$ lying above $p$, and $J:=Jac_{\II}$ be the Jacobian. Let $D_j$ be the image of $D_{\inft E}$ inside $D_{E_j}$.
We also put $A_j:=\cO[[D_j]]$ and $\Lambda:=\cO[[I_E]]$.
\begin{enumerate}
\item Let $M:=\Om{\cRin}{\cO[[\inft{D}]]}\otimes_{\cRin}\II$ and $Jac_{\II}\neq0$. Let $I_0$ be the image of $I_{\inft{E}}$ in $I_E$. Then we have the following short-exact 
sequence:
\begin{equation*}
 0\lra \sgd{\inft{E}}{\ad{\bosym\rho}} \stackrel{}{\lra} M\times(\Om{\II}{\cO[[I_0]]}\otimes_{\II}\II)\stackrel{}{\lra}\Om{\II}{A_0}\otimes_{\II}\II\lra 0.
\end{equation*}
\item Let $\cyn{E}{n}$ be the cyclotomic $\Z_p$-extension of $E_n$, and $\Gamma_n:=\gal(\cyn{E}{n}/E_n)$. Then the module $\sg{\cyn{E}{n}}{\ad{\bosym\rho}}$ is torsion over 
$\II[[\Gamma_n]]$ for all $n$; and is pseudo-isomorphic to $\II^s\oplus \Om{\cR_{\cyn{E}{n}}}{\mathcal O[[\cy D]]}\otimes\II$.
\item Let $\Phi_n(T)$ be the characteristic ideal of $\Om{\cR_{\cyn{E}{n}}}{\mathcal O[[\cy D]]}\otimes\II$, and $\Psi_n(T)$ the characteristic ideal of 
$\sg{\cyn{E}{n}}{\ad{\bosym\rho}}$. Then 
        \begin{equation}
         \Psi_n(T)=\Phi_n(T)T^s, \Phi_n(0)\neq0 \mbox{ and } \Phi_n(0)\mid J\eta,
        \end{equation}
where $\eta$ is the characteristic ideal of the $\II$-module $\sg{E_n}{\ad{\bosym\rho}}$.
\end{enumerate}
\end{theorem}
\begin{proof} 
We give a proof for (i), as the proof of the other statements follow analogously as in \cite[Theorem 6.3]{hs2}.
Let $\cR_j:=\cR^{\phi'}_{E_j}$, $A_j:=\cO[[D_j]]$ and $\Lambda:=\cO[[I_E]]$. 
Put $J_j:=\mathrm{ker}(\cR_j\lra\II)$. Then, we have the following commutative diagram with exact rows and columns, for all $j=1,\cdots,\infty$:
\begin{equation*}
 \xymatrix{
 0\ar[r]\ar[d]                                           &\Om{A_j}{\cO[[I_j]]}\otimes_{A_j}\II\ar@{=}[r]\ar[d]^{e}        &\Om{A_j}{\cO[[I_j]]}\otimes_{A_j}\II\ar[r]\ar[d]^{f}     &0\\
 \dfrac{J_j}{J_j^2}\otimes_{\II}\II\ar[r]\ar[d]^{\cong}  &\Om{R_j}{\cO[[I_j]]}\otimes_{R_j}\II\ar[r]^{b}\ar[d]^{g}        &\Om{\II}{\cO[[I_j]]}\otimes_{R_j}\II\ar[r]\ar[d]^{h}     &0\\
 \dfrac{J_j}{J_j^2}\otimes_{\II}\II\ar[r]\ar[d]          &\Om{R_j}{A_j}\otimes_{R_j}\II\ar[r]^{d}\ar[d]                   &\Om{\II}{A_j}\otimes_{R_j}\II\ar[r]\ar[d]                &0\\
 0 \ar[r]                                                &0 \ar[r]                                                        &0.
 }
\end{equation*}
Since the Jacobian $Jac_{\II}\neq0$, the maps $e$ and $f$ are injective. Therefore, for $j=\infty$, we have the following short-exact sequence:
\begin{equation*}
 0\lra \Om{A_j}{\cO[[I_j]]}\otimes_{A_j}\II \stackrel{\beta}{\lra} M\times(\Om{\II}{\cO[[I_0]]}\otimes_{\II}\II)\stackrel{\alpha}{\lra}\Om{\II}{A_0}\otimes_{\II}\II\lra 0,
\end{equation*}
where $\alpha(m,a)=d(m)-h(a)$ and $\beta(a)=(g(a),b(a))$. 
As $\sgd{\inft{F}}{\ad{\bosym\rho}}\cong\Om{A_j}{\cO[[I_j]]}\otimes_{A_j}\II$, we have 
\begin{equation*}
 0\lra \sgd{\inft{E}}{\ad{\bosym\rho}} \stackrel{\beta}{\lra} M\times(\Om{\II}{\cO[[I_0]]}\otimes_{\II}\II)\stackrel{\alpha}{\lra}\Om{\II}{A_0}\otimes_{\II}\II\lra 0.
\end{equation*}
\end{proof}
This result gives a finer structure of the dual Selmer group $\sgd{\inft{E}}{\ad{\bosym\rho}}$ of the nearly ordinary representation $\bosym\rho$. 
Over the cyclotomic $\Z_p$-extension, Hida interprets the finer structure of $\sgd{\cy{E}}{\ad{\bosym\rho}}$ in terms of trivial zeros of the $p$-adic L-function. 
\section{Noncommutative Iwasawa theory of $\sgd{\Ein}{\ad{\phi}}$}\label{non-commmutative}
\subsection{Ore sets and the category $\mathfrak M^\II_\cH(\cG)$}
Let $\Ein/E$ be a $p$-adic Lie extension such that $\cy{E}\subset\Ein$. Let $\cG:=\gal(\Ein/E)$ and $\cH=\gal(\Ein/\cy{E})$.
We will consider an analogue of
the \emph{Ore set}, that was first considered by Venjakob for a formulation of the Iwasawa Main conjecture over $p$-adic Lie extensions 
(see \cite{cfksv}). 
We recall the Ore set that was considered by Venjakob.
\begin{defn} Let $\cO$ be a finite extension of $\Z_p$. Then the set
 \begin{equation*}
 S:=\{x\in\cO[[\cG]]\mid\cO[[\cG]]/x \mbox{ is a finitely generated module over } \cO[[\cH]]\}.
 \end{equation*}
 is a left-right Ore set.
\end{defn}
The following Ore set is a natural and obvious generalization of the one which has been considered by Venjakob, Coates et al in \cite{cfksv} and in Fukaya-Kato \cite{fk}.
Let $\II$ be an irreducible component of the universal locally cyclotomic deformation ring $\cR_E$ for the functor $\Phi_E^{cyc}$.
\begin{defn}
The set defined by
\begin{equation*}
\sS:=\{x\in\IIG\mid\IIG/x \mbox{ is a finitely generated module over } \II[[\cH]]\}
\end{equation*}
is a left-right Ore set.
\end{defn}
\begin{lemma}
The set $\sS$ is a multiplicatively closed set.
\end{lemma}
\begin{proof}
For two elements $x, y \in\IIG$ consider the following exact sequence
\begin{equation*}
 0\lra x\IIG/xy \lra \IIG/xy \lra \IIG/x \lra 0.
\end{equation*}
The surjection $\IIG/y\lra x\IIG/xy\lra 0$ implies that $x\IIG/xy$ is finitely generated over $\II[[\cH]]$ and the lemma follows.
\end{proof}
\begin{defn} Let $\m$ denote the maximal ideal of $\II$. We define
\begin{equation*}\sS^\ast:=\cup_n \m^n\sS.\end{equation*}
\end{defn}
The set $\sS^*$ is also a multiplicative Ore set. In his thesis, Barth \cite{barth} has also considered an Ore set which is different from ours.
\begin{defn} We define $\mathfrak M^\II_\cH(\cG)$ to be the category of all modules which are finitely generated over $\IIG$ and $\sS^\ast$-torsion.
\end{defn}
For the maximal ideal $\m$ of $\II$, we define
\begin{eqnarray*}
M[\m]&:=&\{x\in M\mid ax=0 \mbox{ for some } a\in\m\}\\
M(\m)&:=&\cup_n M[\m^n].
\end{eqnarray*}
As $\II$ is an commutative integral domain, it is easy to see that $M[\m]$ and $M(\m)$ are submodules of $M$ over $\IIG$.

As in \cite[Lemma 2.1]{cfksv}, we have the following characterization of the Ore set $\sS$.
\begin{lemma} Let $\varphi_\cH:\IIG\lra\II[[\Gamma]]$ and $\psi_\cH:\IIG\lra\Omega(\Gamma)$ be the natural surjections. Then
 \begin{enumerate}
  \item $\sS$ is the set of all $x$ in $\IIG$ such that $\II[[\Gamma]]/\II[[\Gamma]]\varphi_\cH(x)$ is a finitely generated $\II$-module;
  \item $\sS$ is the set of all $x$ in $\IIG$ such that $\Omega(\Gamma)/\Omega(\Gamma)\psi_\cH(x)$ is finite.
 \end{enumerate}
\end{lemma}
\begin{proof}
 For any element $x\in\IIG$, we put $M=\IIG/\IIG x$. Then
\begin{equation*}
 M_\cH=\II[[\Gamma]]/\II[[\Gamma]]\varphi_\cH(x),\quad M/{\m_\cH}M=\Omega(\Gamma)/\Omega(\Gamma)\psi_\cH(x),
\end{equation*}
where $\m_\cH$ denotes the maximal ideal of $\II[[\cH]]$. Therefore the assertions follow from Nakayama's lemma ( \cite[Lemma 5.2.18]{nsw}).
\end{proof}
\begin{propn}
 A finitely generated module $M$ over $\IIG$ is $\sS$-torsion if and only if $M$ is finitely generated over $\II[[\cH]]$.
\end{propn}
\begin{cor}
 A finitely generated module $M$ over $\IIG$ is $\sS^\ast$-torsion if and only if $M/M(\m)$ is finitely generated over $\II[[\cH]]$.
\end{cor}
The previous proposition follow analogously as in \cite[Prop 2.3]{cfksv}, and the following lemma is an easy generalization of \cite[Lemma 5.2]{ven-characteristic}.
\begin{lemma}
 Let $T$ be equipped with a continuous linear action of $\cG$, and $\psi$ be any representation of $\cG$ defined over $\II$. 
 Let $\phi:T\cong\II^r$ be an ismorphism of $\II$-modules. Then we have the following isomorphism
 \begin{equation*}
  \IIG\otimes_\II T\cong\IIG\otimes_\II\II^r
 \end{equation*}
which is induced by the mapping $g\otimes t$, with $g\in\cG$ and $t\in T$, to $g(\psi(g^{-1})t)$.
\end{lemma}
From this lemma, the following Lemma follows analogously as in the proof of \cite[Lemma 3.2]{cfksv}.
\begin{lemma}\label{twisting}
Let $\psi$ be any irreducible representation $\psi$ of $\Delta$ of degree $n_\psi$, and $U$ be finitely generated $\sS$-torsion module over $\IIG$. 
Then the twist $U\otimes_\II\II^{n_\psi}$, with the diagonal action of $\cG$
is also finitely generated and $\sS$-torsion over $\IIG$.
\end{lemma}
There is another way of describing the category $\mathfrak{M}_\cH^\II(\cG)$. Consider the canonical injection $i:\IIG\lra\IIG_\sS$.
Indeed, as $\cG$ is a uniform pro-$p$ group, $\IIG$ has no zero divisors, hence the injection.
First recall that $K_0(\IIG,\IIG_\sS)$ is an abelian group, whose group law is denoted additively. Consider triples 
$(P,\alpha,Q)$, with $P$ and $Q$ finitely generated projective modules over $\IIG$ and $\alpha$ is an isomorphism between $P\otimes_{\IIG}\IIG_\sS$ and 
$Q\otimes_{\IIG}\IIG_\sS$ over $\IIG_\sS$. A morphism between $(P,\alpha,Q)$ and $(P',\alpha',Q')$ is naturally defined to be a pair of $\IIG$-module homomorphism $g:P\lra P'$
and $h:Q\lra Q'$ such that 
\begin{equation*}
 \alpha'\circ(\mathrm{id}_{\IIG_\sS}\otimes g)=(\mathrm{id}_{\IIG_\sS}\otimes h)\circ\alpha.
\end{equation*}
Note that it is an isomorphism if both $g$ and $h$ are isomorphisms. We denote the isomorphism class by $[(P,\alpha,Q)]$.
Then the abelian group $K_0(\IIG,\IIG_\sS)$, is defined by the following generators and relations. 
Generators are the isomorphism classes $[(P,\alpha,Q)]$ and the relations are given by 
\begin{enumerate}
 \item $[(P,\alpha,Q)]=[(P',\alpha',Q')]$ if $(P,\alpha,Q)$ is isomorphic to $[(P',\alpha',Q')]$
 \item $[(P,\alpha,Q)]=[(P',\alpha',Q')]+[(P'',\alpha'',Q'')]$ \\
 for every short exact sequence $0\lra [(P',\alpha',Q')]\lra[(P,\alpha,Q)]\lra [(P'',\alpha'',Q'')]\lra 0$ in $\mathscr C_i$.
 \item $[(P_1,\beta\circ\alpha,P_3)]=[(P_1,\alpha,P_2)]+[(P_2,\alpha,P_3)]$,
 for the map $P_1\stackrel{\alpha}{\lra}P_2\stackrel{\beta}{\lra}P_3$.
\end{enumerate}
Recall the category $\mathscr C_i$,
whose objects are bounded complexes of finitely generated projective $\IIG$-modules whose cohomologies are $\sS$-torsion. Then the abelian group $K_0(\mathscr C_i)$ is 
defined with the following set of generators and relations. The generators are given by $[C]$, where $C$ is an object of $\mathscr C_i$. The relations are given by
\begin{enumerate}
 \item $[C]=0$ if $C$ is acyclic,
 \item $[C]=[C']+[C'']$, for every short-exact sequence $0\lra C'\lra C\lra C''\lra 0$ in $\mathscr C_i$.
\end{enumerate}
It is known that $K_0(\IIG,\IIG_\sS)\cong K_0(\mathscr C_i)$. Moreover, if $\mathscr{H}_\sS$ is the category of all finitely generated $\IIG$-modules which are 
$\sS$-torsion with a 
finite resolution by finitely generated projective modules then $K_0(\IIG,\IIG_\sS)\cong K_0(\mathscr H_\sS)$. For details see Weibel \cite{weibel}.   
Therefore $K_0(\IIG,\IIG_\sS)$ is isomorphic to $K_0(\mathfrak{M}_\cH^\II(\cG))$. We then have the following exact sequence of localization:
\begin{equation}\label{localization}
 \kone{\IIG}\lra\kone{\IIG_\sS}\stackrel{\partial}{\lra} K_0(\IIG,\IIG_\sS)\lra K_0(\IIG) \lra K_0(\IIG_\sS).
\end{equation}
Regarding the connecting homomorphism $\partial$, we have the following generalization of \cite[Lemma 5]{kakde} and \cite[Prop 3.4]{cfksv}.
\begin{lemma}
 The connecting homomorphism $\partial$ is surjective.
\end{lemma}
\begin{proof}
 We give only a brief sketch of the proof. Let $P$ be a pro-$p$ open normal subgroup of $\cG$, and $L$ be a finite extension of $\Q_p$ such that all the irreducible 
representations of $\Delta=\cG/P$ are defined. Then, we have an isomorphism of rings $L[\Delta]\stackrel{\cong}{\lra}\prod_{\psi:\mathrm{irred}}M_{n_\psi}(L)$, where
$\psi$ runs over all the irreducible representations of $\Delta$ and $n_\psi$ is the dimension of $\psi$. Let $\II=\cO[[X_1,\cdots,X_r]]$ and $\mathbb{K}:=L[[X_1,\cdots,X_r]]$. 
Then tensoring with $\II$, we have
\begin{equation*}
 \mathbb{K}[\Delta]\stackrel{\cong}{\lra}\prod_{\psi:\mathrm{irred}}M_{n_\psi}(\mathbb{K}).
\end{equation*}
Analogously as in Coates et. al \cite{cfksv}, we construct a map 
\begin{equation*}
 \lambda: K_0(\IIG)\lra\prod_{\psi:\mathrm{irred}}K_0(\mathbb K),
\end{equation*}
as the composition $\lambda=\lambda_4\circ\lambda_3\circ\lambda_2\circ\lambda_1$ of the following natural maps
\begin{align*}
 \lambda_1&:K_0(\IIG)\lra K_0(\II[\Delta]),\\
 \lambda_2&:K_0(\II[\Delta])\lra K_0(\mathbb K[\Delta]),&\\
 \lambda_3&:K_0(\mathbb K[\Delta])\lra K_0(\KK[\Delta]),\\
 \lambda_4&:K_0(\KK[\Delta])\stackrel{\cong}{\lra}\prod_{\psi:\mathrm{irred}}K_0(M_{n_\psi}(\mathbb{K}))\stackrel{\cong}{\lra}\prod_{\psi:\mathrm{irred}}K_0(\KK).
\end{align*}
Here, the map $\lambda_1$ is defined analogously as in \cite[Lemma 3.5]{cfksv}, and $\lambda_2, \lambda_3$ are induced by the inclusion of rings. The map $\lambda_4$ is induced
by the isomorphism above followed by Morita equivalence. After this, the proof proceeds analogously as in the proof of \cite[Lemma 5]{kakde} using Lemma \ref{twisting}. 
\end{proof}
As a generalization of Conjecture 5.1 in \cite{cfksv}, we can hope that the following is true.
\begin{conj}
 The dual Selmer group  $\sgd{\Ein}{\ad{\bosym\rho}}$ is in the category $\mathfrak M^\II_\cH(\cG)$.
\end{conj}
We compare the Ore sets $S$ and $\sS$ in the following proposition.
\begin{propn}
 Let $\phi_k:\II\lra\cO$ be a specialization map. Then $\phi_k(\sS)= S$. 
\end{propn}
\begin{proof} 
 Let $x\in\sS$. Then there exists a positive integer $m$, such that $\II(\cH)^m\surj\IIG/x$. Applying $\phi_k$, we get the following diagram
 $$
  \xymatrix{
  \II(\cH)^m\ar[r]\ar[d]   &\IIG/x\ar[r]\ar[d]   & 0\\
  \cO[[\cH]]^m\ar[r]\ar[d] &\cO[[\cG]]/\phi_k(x)\ar[d]  &\\
  0                  & 0.
  }
 $$
Since the specialization map is surjective, the vertical maps induced by the specialization map $\phi_k$ are also surjective. 
Therefore $\phi_k(x)\in S$ (\cite[Lemma 2.1]{cfksv}).

Conversely, let $y\in S$. Then, we have a surjection $\cO[[\cH]]^m\lra\cO[[\cG]]/y\lra 0$ for some $m$. Since $\phi_k$ is surjective, there exists $z\in\II(\cH)$ such that
$\phi_k(z)=y$. Further, $\cO[[\cG]]\cong\IIG/\ker{\phi_k}$. Therefore, $\cO[[\cG]]/y\cong\dfrac{\IIG/\ker{\phi_k}}{z}\cong \dfrac{\IIG/z}{\ker{\phi_k}}$, which is
finitely generated over $\cO[[\cH]]\cong\II(\cH)/\ker{\phi_k}$. Therefore, $\dfrac{\IIG/z}{\mathfrak n}$ is finitely generated over $\II(\cH)/\mathfrak n$, 
where $\mathfrak n$ is the maximal ideal
of $\II(\cH)$. By Nakayama's lemma, $\IIG/z$ is finitely generated over $\II(\cH)$. Hence $z\in\sS$.
\end{proof}
\begin{cor} For any specialization map $\phi_k$, $\phi_k(\sS^\ast)=S^\ast$.
\end{cor}

\begin{theorem}
 Let $\phi_k$ be a specialization map with kernel $P_k$, and $W=\II/P_k$. Then any finitely generated $\IIG$-module $M$ is $\sS$-torsion if and only if 
 $M/P_kM$ is $S$-torsion.
\end{theorem}
\begin{proof}
 The module $M/P_kM$ is $S$-torsion if and only if $M/P_kM$ is finitely generated over $W[[\cH]]$. By Nakayama Lemma \cite{nsw}, this is equivalent to $M$ being finitely generated over 
 $\II[[\cH]]$. This is further equivalent to $M$ being finitely generated and $\sS$-torsion. 
\end{proof}
As a consequence, we have the following result.
\begin{theorem}
Consider the representation $\bosym{\rho}_\II:\gal_F\lra\gl{\II}$ and $\phi_k:\II\lra\cO$ be any surjective morphism of local algebras which give rise to a locally cyclotomic
point $P$ of weight $k$. 
 The dual Selmer group $\sgd{\inft{F}}{\ad{{\bosym\rho}_\II}}$ is $\sS$-torsion if and only if $\sgd{\inft{F}}{\ad{\rho_P}}$ is $S$-torsion.
\end{theorem}
Now suppose that $M=\sgd{\inft{F}}{\ad{\rho_P}}$ is in the category $\mathfrak{M}_\cH(\cG)$. Then $M/M(p)$ is $S$-torsion. 
Then for $\mathcal{M}=\sgd{\inft{F}}{\ad{{\rho}_\II}}$, the natural surjection $\mathcal{M}\lra M/M(p)$ factors through the
submodule 
$\mathcal{M}(\m)$ of $\mathcal{M}$. As in the above proof, we can see that $\mathcal{M}/\mathcal{M}(\m)$ is annihilated by $\sS$. 
Therefore, $\mathcal{M}$ is in the category $\mathfrak{M}_\cH^\IIG$. We therefore have the following consequence:
\begin{theorem}\label{big-torsion-small-torsion}
 Consider the representation $\bosym{\rho}_\II:\gal_F\lra\gl{\II}$ and $\phi_k:\II\lra\cO$ be any surjective morphism of local algebras which give rise to a locally cyclotomic
point $P$ of weight $k$. 
 The dual Selmer group $\sgd{\inft{F}}{\ad{{\rho}_\II}}$ is $\sS^\ast$-torsion if and only if $\sgd{\inft{F}}{\ad{\rho_P}}$ is $S^\ast$-torsion.
\end{theorem}

\subsection{Deformation rings over $F_\infty$}
We now give some results which are extensions of the results over the cyclotomic $\Z_p$-extension \cite[Th 5.9, Cor 5.10, 5.11]{hida-hmf} to the case of $p$-adic Lie extensions. Recall the conditions {\bf (h1)-(h4)} in section \ref{deformation}. We also assume the 
condition {\bf (sf)} regarding the square-freenes of the conductor.
\begin{propn}
Let the conditions {\bf (sf), (h1)-(h4)} and absolute irreducibility over $F(\mu_p)$ hold for the representation $\bar\rho_f$. Then 
the local complete algebra $\cR_n$ is reduced for $n=0,1,2,\cdots,\infty$ is reduced with trivial nilradical.
\end{propn}
\begin{proof}
 By \cite[Th 3.50]{hida-hmf}, the deformation $\cR_n$ is isomorphic to a local ring $\mathbb{T}_n$ of the Hecke algebra 
 $\mathbf{h}_{cyc}^{n.ord}(\mathfrak{N},\epsilon;W[[\Gamma_F]])$. This Hecke algebra is reduced under the conditions {\bf (h1)-(h4)}. Therefore
 $\cRin=\varprojlim_n\cR_n$ is reduced.
\end{proof}
\begin{propn}\label{arith-spe} 
 Let $P$ be a locally cyclotomic arithmetic point of weight $k$ and $\rho_P$ denote the representation of the specialization with respect to the point $P$. Suppose that the 
 conditions {\bf (sf)} and {\bf (h1)-(h4)} hold for the representation $\rho_P$. Then 
 \begin{equation*}
  \sgd{F_\infty}{\ad{\rho_P}\otimes_{W}W^*}\cong\Om{\cRin}{W}\otimes_{\cRin} \cR_0/P,
 \end{equation*}
as $W[[\cG]]$-modules. 
Further, $\sgd{F_\infty}{\ad{\rho_P}\otimes_{W}W^*}$ is a $W[[\cG]]$-module of finite type which is torsion. Here $W^*$ is the Pontryagin dual of $W$.
\end{propn}
\begin{proof} Let $\pi_n:\cR_n\lra\cR_0$ be the base change morphism.
 Let $P_n=\pi^{-1}_{n}(P)$ and consider the module $\Om{\cR_n}{W}$. 
 Note that $\cRin/P_\infty=\cR_n/P_n$, and 
 \begin{equation*}
 \sgd{F_n}{\ad{\rho_P}\otimes_{W}W^*}\cong\Om{\cR_n}{W}\otimes_{\cRin}\cR_n/P_n
 \end{equation*}
 Taking projective limits, we have
 \begin{equation*}
 \sgd{F_\infty}{\ad{\rho_P}\otimes_{W}W^*}\cong\Om{\cRin}{W}\otimes_{\cRin} \cR_0/P.
 \end{equation*} 
 By control theorem \ref{control-kahler1}, we have $(\sgd{\inft{F}}{\ad{\rho_P}\otimes W^\ast})_\cG\cong\sgd{F}{\ad{\rho_P}\otimes W^\ast}$ which is a finitely generated 
 module over $W$. Hence, $\sgd{\inft{F}}{\ad{\rho_P}\otimes W^\ast}$ is finitely generated over $W[[\cG]]$. As 
 $\sgd{F}{\ad{\rho_P}\otimes_{W}W^*}\cong\Om{\cR_0}{W}\otimes_{\cRin}\cR/P$ is finite, it is $W[[\cG]]$-torsion by Nakayama Lemma.
\end{proof}
\begin{theorem}\label{noetherian-mhg} Let $\cG$ be a pro-$p$, $p$-adic analytic Lie group of dimension 2. 
 If $\cRin$ is noetherian, then the dual Selmer group $\sgd{F_\infty}{\ad{\rho_P}\otimes_{W}W^\ast}$ is in the category $\mathfrak M_{\cH}(\cG)$.
 Moreover, $\sgd{F_\infty}{\ad{\rho_P}\otimes_{W}W^\ast}$ is torsion as a module over $W[[\cH]]$.
\end{theorem}
\begin{proof}
 If $\cRin$ is noetherian, then the deformation ring $\cR_{cyc}=(\cRin)_\cH$ is also noetherian, and hence $\sgd{\cy{F}}{\ad{\rho_P}\otimes_{W}W^\ast}$, 
 is finitely generated over $W$, i.e., the $\mu$-invariant over the cyclotomic $\Z_p$-extension of $F$ of $\sgd{\cy{F}}{\ad{\rho_P}\otimes_{W}W^\ast}$ is zero, 
 (\cite[Cor 5.11]{hida-hmf}). Then by \cite[Th 5.3]{hachi-ven}, the $\mu$-invariant over the extension $F_\infty$ is also zero. Hence, we have,
 $\sgd{\cy{F}}{\ad{\rho_P}\otimes_{W}W^\ast}(p)=0$. Therefore $\sgd{\inft{F}}{\ad{\rho_P}\otimes_{W}W^\ast}$ is in the category $\mathfrak{M}_\cH(\cG)$.
 
If  $\sgd{F_\infty}{\ad{\rho_P}\otimes_{W}W^\ast}$ has a submodule over $W[[\cH]]$ which is torsion-free, then the module 
$\sgd{F_\infty}{\ad{\rho_P}\otimes_{W}W^\ast}\otimes\mathbb{F}$, for the residue 
field $\mathbb{F}$ of $W$, contains a submodule isomorphic to $\mathbb{F}[[\cH]]$. 
Then the module $\Om{\cRin}{W}\otimes_{\cRin} \cR_0/P\otimes\mathbb{F}$ is not finite, and by \cite[Cor 5.11]{hida-hmf}, $\cRin$ is not noetherian.
\end{proof}
Let $P$ be a locally-cyclotomic point. Then $\widehat\cRin:=\varprojlim_{n}\cR_{\infty,P}/P^n$ is the localization-completion at the point $P$. Over the cyclotomic 
$\Z_p$-extension,
the localization completion $\widehat{\cR_{cyc}}^{}:=\varprojlim_{n}\cR_{cyc,P}/P^n$ is a complete noetherian algebra over the fraction field $K=\mathrm{Frac}(W)$, 
(\cite[Cor 5.12]{hida-mfg}). 

\begin{theorem}\label{loc-noetherian}
Let $\cG$ be a pro-$p$, $p$-adic Lie group which is admissible and $\cH$ be a normal subgroup of $\cG$ such that $\cG/\cH\cong\Gamma$. 
Let $\widehat\varphi:\widehat\cRin\lra\widehat\cRin/P\widehat\cRin=K$ be the natural projection map.
Then 
\begin{enumerate}
\item $(\widehat{\cRin})_\cH\cong\widehat{\cR_{cyc}}$,
\item $\left(\Om{\cRin}{W}\widehat\otimes\cRin/P\right)\widehat\otimes K\cong\Om{\widehat{\cRin}}{K}\widehat\otimes_{\widehat\cRin,\widehat\varphi}K$ as $K[[\cG]]$-modules,
\item the module $\Om{\widehat{\cRin}}{K}\widehat\otimes_{\widehat\cRin,\widehat\varphi}K$ is finitely generated over $K[[\cG]]$.
\end{enumerate}
\end{theorem}
\begin{proof}
\begin{enumerate}
\item This follows as in the base-change isomorphism of Theorem \ref{base-change}.
\item We have $\Om{\cRin}{W}\otimes_{\cRin}\cRin/P\cong P\cRin/P^2\cRin$. Noting that $K$ is the residue field of $\widehat\cRin$, and localizing at the ideal $P\cRin$, we have:
\begin{equation*}
 \left(P/P^2\right)\widehat\otimes K\cong P\widehat\cRin/P^2\widehat\cRin,
\end{equation*}
from which the isomorphism follows. This is an isomorphism of $K[[\cG]]$-modules.
\item As $\Om{\cRin}{W}\widehat\otimes\cRin/P$ is finitely generated over $W[[\cG]]$, the module $\left(\Om{\cRin}{W}\widehat\otimes\cRin/P\right)\widehat\otimes K$ is 
finitely generated over $K[[\cG]]$. 
\end{enumerate}
\end{proof}
\begin{propn}
Let $\inft{F}$ be a totally ramified $p$-adic Lie extension over $F$, and $e=|\Sigma_p|$ the number of primes of $F$ above $p$. Then $\mathrm{dim}\,\cR_{m,P}=e+1$.
 Further, let $P$ be a locally cyclotomic point over $\cR_n$, which we may regard as a point over $\cRin$. Then for any finite index subgroup $\Delta_m$, with $m\geq n$, we have
 \begin{equation*}
  (\cR_{\infty,P})_{\Delta_m}\cong\cR_{n,P}.
 \end{equation*}
\end{propn}
\begin{proof}
Since $\cR_{m,P}$ is an integral domain of dimension $e+1$, and the base change morphism $\cR_{m,P}\lra\cR_{n,P}$ of $W$-algebras is surjective, we have
$\mathrm{dim}\,\cR_{m,P}=\mathrm{dim}\,\cR_{n,P}$. It follows that $\cR_{m,P}\cong\cR_{n,P}$, and the result follows.
\end{proof}
\subsection{The category $\mathfrak{M}_{\cH}(\cG)$}
We begin by recalling few facts about pseudocompact rings and pseudocompact modules from \cite{brumer}. 
A \emph{pseudocompact ring} $\boldsymbol\Lambda$ is a complete Hausdorff topological ring which admits a system of open neighborhoods of 0 consisting of two-sided ideals $I$ 
for which $\boldsymbol\Lambda/I$ is an Artin ring. 
A complete Hausdorff topological $\boldsymbol\Lambda$-module $M$ is said to be \emph{pseudocompact} if it has a system of open neighborhoods of 0 consisting of submodules 
$N$ for which $M/N$ has finite length.

Consider the field $K$ from the previous section.
The ring $K[[\cG]]$ is pseudocompact with respect to the topology given by the power of the augmentation ideals $\{J_\cG^n\mid n\in\N\}$. 
Indeed, $K[[\cG]]=\varprojlim_UK[\cG/U]$, where $U$ runs over the open normal subgroups of $\cG$, and each of the rings $K[\cG/U]$ is Artinian as $\cG/U$ is a finite group.
Moreover, $K[\cG/U]=K[[\cG]]/J_{U}$ for the augmentation ideal $J_{U}$, which is the kernel of the canonical projection map $K[[\cG]]\lra K[\cG/U]$.

Since $\cG$ is topologically finitely generated 
powerful pro-$p$ group, the topology generated by the basis of neighborhoods $\{J_\cG^n\mid n\in\N\}$ is equivalent to the topology given by the augmentation ideals 
$\{J_{\cG_n}\mid \cG_n\mbox{ open normal subgroup of }\cG, n\in\N\}$.
\begin{theorem}
Under the natural action of $\cG$, we have
 \begin{equation*}
 \left(\Om{\widehat{\cRin}}{K}\widehat\otimes_{\widehat\cRin,\widehat\varphi}K\right)_\cH\cong\Om{\widehat{\cR_{cyc}}}{K}\widehat\otimes_{\widehat\cR_{cyc},\widehat\varphi}K.
 \end{equation*}
 Let $\Om{\widehat{\cRin}}{K}\widehat\otimes_{\widehat\cRin,\widehat\varphi}K$ be a pseudocompact $K[[\cG]]$-module. 
Then the module $\Om{\widehat{\cRin}}{K}\widehat\otimes_{\widehat\cRin,\widehat\varphi}K$ is finitely generated over $K[[\cH]]$.
\end{theorem}
\begin{proof}
The first statement follows from the isomorphism $(\widehat{\cRin})_\cH\cong\widehat{\cR_{cyc}}$. 
If $Y:=\Om{\widehat{\cRin}}{K}\widehat\otimes_{\widehat\cRin,\widehat\varphi}K$ is pseudocompact with respect to the $\{J_{\cH_n}(Y)\}$-topology, then
applying Nakayama Lemma for pseudocompact modules, the module $Y$ is finitely generated over $K[[\cH]]$ provided $Y_{\cH}$ is finitely generated over $K$. 
This is easy to see as
$Y_{\cH}\cong\Om{\widehat{\cR_{cyc}}}{K}\widehat\otimes_{\widehat\cR_{cyc},\widehat\varphi}K$, and 
$\Om{\widehat{\cR_{cyc}}}{K}\widehat\otimes_{\widehat\cR_{cyc},\widehat\varphi}K\cong\Om{\cR_{cyc}}{W}\widehat\otimes_{\cR_{cyc},\widehat\varphi}K$, 
which is a finitely generated module over $K$.
\end{proof}
\begin{lemma} 
Let $M:=\Om{\cRin}{W}\widehat\otimes\cRin/P$ and $M':=M/M(p)$.
Then there exist a finitely generated $\Z_p[[\cH]]$-submodule $M_0$ of $M'$ such that 
$(M'/M_0)\widehat\otimes K=0$. In particular, $M'/M_0$ is $W$-torsion.
\end{lemma}
\begin{proof}
As $M/M(p)$ is $\Z_p$ torsion-free, the natural map $M/M(p)\lra M\widehat\otimes\Q_p$ is an injective map of $W[[\cH]]$ modules.
By Theorem \ref{loc-noetherian} and the previous proposition, $m\otimes p^{-e}$ generates $M\widehat\otimes W$, so finitely many of them generate $M\widehat\otimes K$ over $K[[\cH]]$. Let
$\{x_i\otimes p^{-t_i}\mid i=1,\cdots,k\}$ generate $M\widehat\otimes K$. 
As each $x_i\otimes p^{-t_i}=p^{-t_i}(x\otimes1)$, we can take the generators to be $\{x_i\otimes1\mid i=1,\cdots,k\}$.

Consider the module $M_0$ generated by $\{x_i+ M(p)\mid i=1,\cdots,k\}$ over $W[[\cH]]$.
Then $M_0$ is a submodule of $M'$ over $W[[\cH]]$. From the following short exact sequence of $W[[\cH]]$-modules:
\begin{equation*}
 0\lra M_0\lra M'\lra M'/M_0\lra 0,
\end{equation*}
by  tensoring with $K$, we get the following short exact sequence: 
\begin{equation*}
 0\lra M_0\widehat\otimes K\lra M'\widehat\otimes K\lra (M'/M_0)\widehat\otimes K\lra 0.
\end{equation*}
As the images of $M_0\widehat\otimes K$ and $M'\widehat\otimes K$ generate the same module, we have 
$M_0\widehat\otimes K=M'\widehat\otimes K$.
Therefore, $(M'/M_0)\widehat\otimes K=0$. It follows easily that $M'/M_0$ is $W$-torsion.
\end{proof}
\begin{theorem}
 Let $\Om{\widehat{\cRin}}{K}\widehat\otimes_{\widehat\cRin,\widehat\varphi}K$ be a pseudocompact $K[[\cG]]$-module.
The module $M'=M/M(p)$ is finitely generated over $\Z_p[[\cH]]$. In other words, the dual adjoint selmer group $\Om{\cRin}{W}\widehat\otimes\cRin/P$ is in the category
$\mathfrak{M}_\cH(\cG)$.
\end{theorem}
\begin{proof}
Note that $M$ and hence $M'/M_0$ is finitely generated over $W[[\cG]]$.
Let $y_1,\cdots,y_t$  be the generators of $M'/M_0$ over $W[[\cG]]$. By the lemma above, these generators are $p$-torsion. 
Let $k$ be such that $p^k$ annihilate $y_1,\cdots,y_t$. Then $p^k$ will annihilate all the other elements 
of $M'/M_0$. Hence $p^kM'$ is a submodule of $M_0$ over $W[[\cH]]$. As $W[[\cH]]$ is noetherian, $p^kM'$ is also finitely generated over $W[[\cH]]$.

Consider the multiplication by $p^k$ map $M'\stackrel{p^k}{\lra} M'$. The image of this map is $p^kM'$ and the kernel is trivial as $M'$ has no $p$-torsion. Therefore
$M'\cong p^kM'$ as modules over $W[[\cH]]$. Hence $M'$ is also finitely generated over $W[[\cH]]$. 
\end{proof}
\section{$\mu$-invariants and $\mathfrak{M}_\cH(\cG)$}\label{mu-mhg}
\subsection{Artin representations}
For a totally real field $F$, let $\rho:\gal_F\lra GL_2(\cO)$ be an Artin representation, i.e., a continuous representation with finite image. 
We also assume that it is nearly ordinary at all primes dividing $p$. Let $K=(\ol F)^{ker(\rho)}$ and $\Delta=\gal(K/F)$. We denote by $S$, the set consisting of all
the ramified primes of $\rho$ including the primes lying over $p$ and the infinite primes. As before, we write $V$ for the representation space of $\rho$, $T$ for a 
$\gal_F$-stable lattice of $V$ and $A:=V/T$.

Let $K$ be \emph{totally real}. 
Then identifying $\gal(\cy{K}/\cy{F})$ with $\Delta$, for the mod-$\pi$ reduction $\ol\rho$ of $\rho$, 
consider the  inflation and restriction exact sequence:
\begin{equation*}
 0\lra \Hone{\Delta}{\ad{\ol\rho}^{\gal_{\cy{K}}}}\lra\Hone{\gal(K_S/\cy{F})}{\ad{\ol\rho}}\lra\Hom{}{\gal(K_S/\cy{K})}{M_2(\FF)}^\Delta
\end{equation*}
As $K$ is \emph{totally real}, $\gal(K_S^{ab,p}/\cy{K})$ is finitely generated and torsion over $\Z_p[[\Gamma]]$.
If the $\mu$-invariant of $\gal(K_S^{ab,p}/\cy{K})$ over $\Z_p[[\Gamma]]$ is zero, then 
$\Hom{}{\gal(K_S^{ab,p}/\cy{K})}{\FF}$ and hence $\sgd{\cy{F}}{\ad{\bar\rho}}$ are finite. 
It follows that $\sgd{\cy{F}}{\ad{\rho}}$ has $\mu$-invariant equal to zero, and 
hence belongs to the category $\mathfrak{M}_\cH(\cG)$. 
Indeed, the dual Selmer group $\sgd{\inft{F}}{\ad{\rho}}_\cH$ is finitely generated over $\cO$, and by Nakayama Lemma,
$\sgd{\inft{F}}{\ad{\rho}}$ is finitely generated over $\cO[[\cH]]$.
This vanishing of the $\mu$-invariant for $\Hom{}{\gal(K_S^{ab,p}/\cy{K})}{\FF}$ is a conjecture of Iwasawa. 
We then have the following result.
\begin{theorem}\label{totally-even}
 Let $K$, the field cut out by $\rho$, be totally real, and $\inft{F}$ be a $p$-adic Lie extension containing $\cy{F}$. Then, assuming Iwasawa's conjecture on the vanishing
 of the $\mu$-invariant for $\gal(K_S^{ab,p}/\cy{K})$, the dual Selmer group $\sgd{\inft{F}}{\ad{\rho}}$ is in the category $\mathfrak{M}_\cH(\cG)$.
\end{theorem}
Let $K$ be \emph{not totally real}.
Let $\rho$ be an Artin representation, such that $\Delta$ has order prime to $p$. Let $\rho$ be ordinary at all primes of $F$ dividing $p$. Let $\delta_\p$ denote the 
unramified quotient of $\rho\mid_{G_\p}$.
In this case, we are inspired by the computations of the Selmer group in \cite{cv}. There the computations are done for representations which 
are induced by a finite order character of a real quadratic extension of $\Q$. Similar computations work here also, as we briefly explain below. A computation covering more general Artin representations is given in \cite{gv-artin}.

We consider the fields $F_n$ be as before and the compositum $K_n=KF_n$.
Let $I_{n,p}=\prod_{\p|p}I_{n,\p}$ where $I_{n,\p}$ denotes the abelianization of the inertia subgroup of $\gal_{K_n}$ at $\p$, and $I_{n,\p,\cO}=I_{n,\p}\otimes_{\Z_p}\cO$. 
As $\Delta$ acts on $I_{n,\p,\cO}$, we have the decomposition 
$I_{n,\p,\cO}=I_{n,\p,\cO}^{(\delta_\p)}\times J_{n,\p,\cO}^{(\delta_\p)}$, where the subscript denotes the $\delta_\p$-isotypical component, and $J_{n,\p,\cO}^{(\delta_\p)}$ denotes the other components.
As $\sg{F_n}{A}\subset\Hom{}{{\gal(M_n/K_n)}}{{A}}^\Delta$, it follows from the definition that any element of $\sg{F_n}{A}$ vanishes on $J_{n,\p,\cO}^{(\delta_\p)}$. 

Writing $I_{n,\p,\cO}=\prod_{\p\mid p}I_{n,\p,\cO}, I_{n,p,\cO}^{(\delta)}:=\prod_{\p\mid p}I_{n,\p,\cO}^{(\delta_\p)}$, $J_{n,p,\cO}^{(\delta)}:=\prod_{\p\mid p}J_{n,\p,\cO}^{(\delta_\p)}$, and 
denoting by $L_n$ the $p$-Hilbert class field of $K_n$, we have the short exact sequence of Galois modules:
\begin{equation*}
 0\lra I_{n,p,\cO}\lra\gal(M_n/K_n)_\cO\lra\gal(L_n/K_n)_\cO\lra0,
\end{equation*}
where the groups with the subscript $\cO$ denotes the scalar extension by $\cO$.
Unwinding the definition of $\sg{F_n}{A}$, gives the following exact sequence of $\cO[[\cG]]$-modules:
\begin{equation*}
 0\lra \mathrm{Hom}(\gal(L_n/K_n)_\cO,A)^\Delta\lra \sg{F_n}{A} \lra \mathrm{Hom}(I_{n,p,\cO}/J_{n,p,\cO}^{(\delta)},A)^\Delta\lra0.
\end{equation*}
Taking projective limit defined by the norm maps over $K_n$ and $F_n$, we get the following exact sequence of $\cO[[\cG]]$-modules:
\begin{equation*}
0\lra \mathrm{Hom}(\gal(L_\infty/K_\infty)_\cO,A)^\Delta\lra \sg{F_\infty}{A}
    \lra \mathrm{Hom}({I}_{\infty,p,\cO}^{}/{J}_{\infty,p,\cO}^{(\delta)},A)^\Delta\lra0,
\end{equation*}
where ${I}_{\infty,p,\cO}=\prod_{\p\mid p}{I}_{\infty,\p,\cO}, {J}_{\infty,p}^{(\delta_\p)}=\varprojlim_n{J}_{n,p,\cO}^{(\delta_\p)}$, the inverse limits are defined by the norm maps.
Similarly, taking projective limit over all finite extensions of $K$ contained in $\cy{K}$, we get the following exact sequence:
\begin{equation*}
0\lra \mathrm{Hom}(\gal(\cy{L}/\cy{K})_\cO,A)^\Delta\lra \sg{\cy{F}}{A}
            \lra \mathrm{Hom}({I}_{cyc,p,\cO}^{}/{J}_{cyc,p,\cO}^{(\delta)},A)^\Delta\lra0,
\end{equation*}
where ${I}_{cyc,p,\cO}=\prod_{\p\mid p}{I}_{cyc,\p,\cO}, {J}_{cyc,p,\cO}^{(\delta)}=\varprojlim_n{J}_{n,p,\cO}^{(\delta)}$, the inverse limits are defined by the norm maps.
As in the proof of \cite[Theorem 3.1]{cv}, it follows from Class Field Theory that the dual Selmer group $\sgd{\cy F}{\ad{\rho}}$ is torsion over $\cO[[\Gamma]]$. In general, the selmer group $\sgd{\cy{F}}{A}$ is also torsion over $\cO[[\Gamma]]$, (see \cite[Theorem 1]{gv-artin}).

For an Artin representation $\beta$ of $\Delta$, we consider the representation space $V_{\beta}$ and a Galois stable lattice $T_\beta$ and put 
$A_\beta=V_\beta/T_\beta$. Viewing $T\otimes T_\beta$ as a lattice in $V\otimes V_\beta$, we consider the image of $V^{\delta_\p}\otimes V_\beta$ in $A\otimes A_\beta$.
Following Greenberg in \cite{gr-modular}, we consider the \emph{Selmer atoms} for $A\otimes A_\beta=(V\otimes V_\beta)/(T\otimes T_\beta)$, defined by
\begin{equation*}
 \sg{F_n}{\rho\otimes\beta}=ker\left(\Hone{F_n}{A\otimes A_\beta}
                                         \lra\prod_{v\in\Sigma-\Sigma_p}\Hone{F_{n,v}}{A\otimes A_\beta}\times\Hone{F_{n,\p}}{A^{(\delta_\p)}\otimes A_\beta}\right).
\end{equation*}
\begin{theorem}\label{iwasawa-mhg}
Let the $\cO[[\Gamma]]$-module $\gal(\cy{L}/\cy{K})_\cO$ have $\mu$-invariant equal to zero. Then 
\begin{enumerate}
 \item ${I}_{cyc,p,\cO}^{}/{J}_{cyc,p,\cO}^{(\delta)}$ also has $\mu$-invariant equal to zero;
 \item $\sgd{\cy F}{{\rho}}$ has $\mu$-invariant equal to zero;
 \item for $\cG$ a $p$-adic Lie group of dimension 2,
 $\sgd{F_\infty}{{\rho}}$ is in the category $\mathfrak{M}_{\cH}(\cG)$;
 \item for any $p$-adic Lie extension $\inft{F}$ of $F$, 
  $\sgd{F_\infty}{\ad{\rho}}$ is in the category $\mathfrak{M}_{\cH}(\cG)$.
\end{enumerate}
\end{theorem}
\begin{proof}
\begin{enumerate}
 \item Consider the exact sequence
\begin{equation*}
 0\lra I_{cyc,p,\cO}\lra\gal(\cy{M}/\cy{K})_\cO\lra\gal(\cy{L}/\cy{K})_\cO\lra0.
\end{equation*}
This gives rise to the following exact sequence of their $p$-primary parts:
\begin{equation*}
 0\lra I_{cyc,p,\cO}(p)\lra\gal(\cy{M}/\cy{K})_\cO(p)\lra\gal(\cy{L}/\cy{K})_\cO(p).
\end{equation*}
As any prime ramified in $K/F$ is not totally split in the cyclotomic $\Z_p$-extension $\cy{K}/K$,
it follows from the corollaries of \cite[Cor 11.3.6, 11.3.16, 11.3.17]{nsw}, that the $\mu$-invariants of the $\cO[[\Gamma]]$-modules 
$\gal(\cy{M}/\cy{K})_\cO(p)$ and $\gal(\cy{L}/\cy{K})_\cO(p)$ 
are equal.
By the assumption, $\gal(\cy{M}/\cy{K})_\cO(p)$ is finite. Therefore
$I_{cyc,p,\cO}(p)$, $I_{cyc,p,\cO}^{(\delta)}(p)$, $J_{cyc,p,\cO}^{(\delta)}(p)$, $(I_{cyc,p,\cO}/J_{cyc,p,\cO}^{(\delta)})(p)$ are all finite. 
So, $I_{cyc,p,\cO}/J_{cyc,p,\cO}^{(\delta)}$ is finitely generated over $\cO$.
\item This follows from (i). 
\item Since $\cH$ has dimension equal to one, it has $p$-cohomological dimension one. So the kernel and cokernel of the restriction map
$\sg{\cy F}{\rho}\lra\sg{F_\infty}{\rho}^\cH$ are co-finitely generated over $\cO$. Also $\sg{\cy F}{\rho}$ is cofinitely generated
over $\cO$, by (ii) above. Taking Pontryagin duals and applying 
Nakayama Lemma, $\sgd{F_\infty}{\rho}$ is finitely generated over $\cO[[\cH]]$.
\item Since we have an exact control theorem for the Selmer group of $\ad{\rho}$, $\sgd{\inft F}{\ad\rho}_\cH$ is finitely generated over $\cO$. Again, by Nakayama Lemma, $\sgd{\inft F}{\ad\rho}$ is finitely generated over $\cO[[\cH]]$.
\end{enumerate}
\end{proof}
\begin{remark}It is a conjecture of Iwasawa that the $\mu$-invariant of the $\Z_p[[\Gamma]]$-module $\gal(\cy{L}/\cy{K})$ vanishes. 
This is a theorem due to Ferrero and Washington for number fields $K$ which are abelian over $\Q$.
\end{remark}
\subsection{CM elliptic curves}\label{mu-cm}
Let $E$ be an elliptic curve with CM by an imaginary quadratic extension $M$ of $\Q$ which is defined over $\Q$.  Let $\theta$ be 
the reduction modulo $p$ of 
the Grossencharacter of $E$ viewed as a $p$-adic Galois character. Then $\ol\rho_E=\Ind{M}{\Q}{\theta}$ is the reduction modulo $p$ of the  representation 
of $\gal_\Q$ attached to the Tate module of $E$. 
Further, $\det{\ol\rho_E}=\omega$, where $\omega$ is the mod-$p$ reduction of the cyclotomic character.
Writing $\theta_c$ for the conjugate character of $\theta$, we have $\theta\theta_c=\omega$ on $\gal_M$.

Let $K=\Q(E[p])$ be the field of $p$-torsion points of $E$. 
Let $p$ be a prime of good ordinary reduction for $E$, then the prime $p$ splits completely in $M$, and 
the Galois group $\Delta=\gal(M/\Q)$ is the normalizer of a split Cartan subgroup of $GL_2(\FF_p)$, with the  subgroup 
$\Delta_0=\gal(K/M)$ abelian of  index 2 in $\Delta$. 
Further, $\Delta_0\cong C_1\oplus C_2$, where $C_1,C_2$ are both cyclic subgroups of the same order, say $n$. Let the order $n^2=|\gal(K/M)|$ be \emph{even} and prime-to-$p$. 

 Consider the Hida family of $\rho_E$, which is given by a representation $\boldsymbol\rho:\gal_{\Q}\lra GL_2(\mathbf{h})$, where $\mathbf{h}$ is the localization of 
 the universal Hecke algebra at a maximal ideal which is determined by $\rho_E$. Then there exists a prime ideal $P_1\in\mathrm{Spec}(\mathbf{h})$ such that 
 $\boldsymbol{\rho}$ specializes to a representation $\rho_1$ of a weight-one eigenform $f_1$.  Let $\mathbf{T}$ be a free $\mathbf{h}$-module of rank 2. Then there exists a 
 $G_{\Q_p}$-stable $\mathbf{h}$-submodule $\mathbf{T}_0\subset\mathbf{T}$ such that $\mathbf{T/T_0}$ is unramified at $p$. Both $\mathbf{T_0}$ and $\mathbf{T}$ are free 
 $\mathbf{h}$-modules of rank 1. Let the specializations for $\mathbf{T_0}$ and $\mathbf{T/T_0}$ be $\varepsilon_p$ and $\delta_p$ respectively.
It is a well-known theorem due to Deligne and Serre on weight one eigenforms that $\rho_1$ is an Artin representation, and it comes from a weight 1 eigenform, which is ordinary at $p$ with $\delta_p$ unramified. In fact, $\rho_1=\Ind{\Q}{M}{\wt\xi}$, for some \emph{finite order} character $\wt\xi$ of $\gal_M$, 
and $\ol\rho_E\equiv\ol\rho_1$ modulo $p$. 
Let $\xi$ be the reduction of $\wt\xi$ modulo $p$. Then the reduction modulo $p$ of $\rho_1$ is $\ol\rho_1=\Ind{M}{\Q}{\xi}$.

As the characters $\wt\xi,\wt\xi_c$ are defined over $\Z_p$, the representation $\rho_1$
is defined over $\Z_p$.
Since each of the components $C_1, C_2$ of $\Delta_0$ are cyclic of order dividing $(p-1)$, any character $\psi$ of $\Delta_0$
defined over a finite field takes values in $\FF_p$. Therefore, any irreducible representation $\alpha=\Ind{M}{\Q}{\psi}$ over a finite extension of $\FF_p$, is actually defined over $\FF_p$, and it is the reduction modulo $p$ of the representation 
$\wt\alpha=\Ind{M}{\Q}{\wt\xi}$ defined over $\Q_p$.
Also, as $p$ is unramified in $M$, it is linearly disjoint from $\Q(\mu_p)$,
which is a field contained in $K$. Let $\langle c\rangle=\gal(M/\Q)$, and for any character $\psi$ of $\Delta_0$, its Galois
conjugate is denoted by $\psi_c$.

It is also easy to see that all irreducible representations of $\Delta$ over $\FF_p$ are of degree at most 2, and induced from characters of $\Delta_0$.
\begin{lemma}As representations of $\gal_\Q$, we have the isomorphisms
\begin{eqnarray*}
\ol\rho_1\otimes\alpha=\Ind{M}{\Q}{\xi}\otimes_{\FF_p}\Ind{M}{\Q}{\psi}\cong\Ind{M}{\Q}{(\xi\psi)}\oplus\Ind{M}{\Q}{(\xi\psi_c)}
\mbox{ over }\FF_p. 
\end{eqnarray*}
\end{lemma}
\begin{proof}
 Let $U$ be any representation of $\Delta$. Then, we have, by Frobenius reciprocity, 
 \begin{align*}
 \Hom{\Delta}{U}{\Ind{M}{\Q}{\xi}\otimes\Ind{M}{\Q}{\psi}}&\cong\Hom{\Delta}{U\otimes(\Ind{M}{\Q}{\psi})\hat\quad}{\Ind{M}{\Q}{\xi}}\\
                                                          &\cong\Hom{\Delta_0}{U\otimes(\Ind{M}{\Q}{\psi})\hat\quad}{\xi}\\
                                                          &\cong\Hom{\Delta_0}{U}{\Ind{M}{\Q}{\psi}\otimes\xi}\\
                                                          &\cong\Hom{\Delta_0}{U}{\xi\psi}\oplus\Hom{\Delta_0}{U}{\xi\psi_c}\\
                                                          &\cong\Hom{\Delta}{U}{\Ind{M}{\Q}{(\xi\psi)}}\oplus\Hom{\Delta}{U}{\Ind{M}{\Q}{(\xi\psi_c)}}.
 \end{align*}
 Therefore $\Ind{M}{\Q}{\xi}\otimes_{\FF_p}\Ind{M}{\Q}{\psi}\cong\Ind{M}{\Q}{(\xi\psi)}\oplus\Ind{M}{\Q}{(\xi\psi_c)}$.
 \end{proof}
\begin{theorem}\label{mu-twists}
 The Selmer group $\sg{\cy\Q}{\ol\rho_1\otimes\alpha}$ is finite for all irreducible Artin representations $\alpha$ of $\Delta$ over 
 $\FF_p$, if 
 $\sg{\cy\Q}{\Ind{M}{\Q}{\xi^r}\otimes\omega^a}$ is finite for all $0<r,a\leq(p-1)$.
\end{theorem}
\begin{proof} Let $\alpha=\Ind{M}{\Q}{(\psi)}$. 
Then the isomorphism of representations in the above Lemma induces exact sequences of cohomology groups, from which we get the following commutative diagram:
\begin{equation*}
\xymatrix{
\Hone{\cy{\Q}}{\Ind{M}{\Q}{(\xi\psi)}}\ar[r]\ar[d] &\Hone{\cy{\Q}}{\rho\otimes\alpha}\ar[r]\ar[d] &\Hone{\cy{\Q}}{\Ind{M}{\Q}{(\xi\psi_c)}}\ar[d]\\
\Hone{\Q_{cyc,p}}{\xi_c\psi_c}\ar[r] &\Hone{\Q_{cyc,p}}{\xi_c\otimes\alpha}\ar[r] &\Hone{\Q_{cyc,p}}{\xi_c\psi}.
}
\end{equation*}
Now, with the Selmer group
\begin{equation*}
\sgart{\cy\Q}{\Ind{M}{\Q}{\xi\psi_c}}=ker(\Hone{\cy{\Q}}{\Ind{M}{\Q}{(\xi\psi_c)}}\lra\Hone{\Q_{cyc,p}}{\xi_c\psi}),
\end{equation*}
we have the sequence which is exact at the middle:
\begin{equation*}
\sgart{\cy\Q}{\Ind{M}{\Q}{\xi\psi}} \lra \sg{\cy\Q}{\rho\otimes\alpha} \lra\sgart{\cy\Q}{\Ind{M}{\Q}{\xi\psi_c}}.
\end{equation*}
Therefore $\sg{\cy\Q}{\ol\rho_1\otimes\alpha}$ is finite if both $\sgart{\cy\Q}{\Ind{M}{\Q}{\xi\psi}}$ and
$\sgart{\cy\Q}{\Ind{M}{\Q}{\xi\psi_c}}$ are finite.
By the Lemmas above, the character $\psi=\xi^a\xi_c^b$ for some integers $a,b$. Since $\xi\xi_c=\omega$,
$\Ind{M}{\Q}{\xi\psi_c}\cong(\Ind{M}{\Q}{\xi^{-a+b+1}})\otimes\omega^a$, and 
$\sgart{\cy\Q}{\Ind{M}{\Q}{\xi\psi_c}}\cong\sg{\cy\Q}{\Ind{M}{\Q}{(\xi^{-a+b+1})\otimes\omega^{a}}}$.

If $\alpha$ is a one dimensional character of $\Delta$, then clearly $\ol\rho_1\otimes\alpha$ is isomorphic to
$(\Ind{M}{\Q}{\xi})\otimes\omega^a$. The theorem follows combining both the cases. 
\end{proof}
\begin{remark}\label{reducible-mu}
It is widely believed to be true that the Selmer groups $\sg{\cy\Q}{\Ind{M}{\Q}{\xi\psi}}$ and $\sgart{\cy\Q}{\Ind{M}{\Q}{\xi\psi_c}}$ have $\mu$-invariant equal to zero if the respective representations $\Ind{M}{\Q}{\xi\psi}$ and $\Ind{M}{\Q}{\xi\psi_c}$
are  irreducible. Some of the representations occurring in the theorem above could be reducible. 
In the following proposition, we show that the Selmer groups of those representations have $\mu$-invariant zero.
For this it is crucial to have $M$ to be \emph{ imaginary } quadratic. We also clarify how the same argument cannot be made to work if $M$ is \emph{real} quadratic.
\end{remark}
\begin{propn}
 Let $M$ be an imaginary quadratic extension of $\Q$ in which $p$ splits completely, and $\Ind{M}{\Q}{\xi\psi_c}$ be \emph{reducible}. Then 
 the $\mu$-invariant of $\sg{\cy\Q}{\Ind{M}{\Q}{\xi\psi_c}}$ is equal to 0. On the other hand, let $M$
 be a real quadratic extension of $\Q$ in which $p$ splits completely. Let $\tau$ be the real quadratic character of $M$,
 and for a character $\psi$ of $\gal_M$, $\psi_\tau$ be the conjugate character.
 Then there exist characters $\psi$ of $\gal_M$, such that $\Ind{M}{\Q}{\xi\psi_\tau}$ is \emph{reducible} and the 
 $\mu$-invariant of  $\sg{\cy\Q}{\Ind{M}{\Q}{\xi\psi_\tau}}$ is \emph{not} equal to 0.
\end{propn}
\begin{proof}
Let $M$ be imaginary quadratic and 
 $\Ind{M}{\Q}{\xi\psi_c}$ be \emph{reducible}. Then $\Ind{M}{\Q}{\xi\psi_c}\cong\omega^i\oplus\eta\omega^i$, where 
$\eta$ is the quadratic character attached to $M$. Clearly, the determinant is \emph{odd} as $M$ is \emph{ imaginary }quadratic.
This plays a \emph{crucial} role below in using the theorem of Ferrero-Washington. One of the characters $\omega^i$ and 
$\eta\omega^i$ is odd and the other even. Therefore, the $\FF_p[[\Gamma]]$-corank of $\Hone{\cy{\Q}}{\omega^i}$ is 0 or 1, according
as $i$ is even or odd, and that of $\Hone{\cy{\Q}}{\eta\omega^i}$ is 1 or 0. It follows that the $\FF_p[[\Gamma]]$-corank of $\Hone{\cy{\Q}}{\omega^i\oplus\eta\omega^i}$ is 1. 
Also, restricted to $G_p$, we have $\omega^i=\eta\omega^i$. 

Let $i$ be \emph{odd}. Then $\omega^i$ is odd and $\eta\omega^i$ is even. 
We then consider the following commutative diagram of exact sequences:
\begin{equation*}
\xymatrix{
\Hone{\cy{\Q}}{\eta\omega^i}\ar[r]\ar[d] &\Hone{\cy{\Q}}{\omega^i\oplus\eta\omega^i}\ar[r]\ar[d] 
                                         &\Hone{\cy{\Q}}{\omega^i}\ar[d]\\
0                           \ar[r] &\Hone{\Q_{cyc,p}}{\eta\omega^i}\ar[r]^\cong &\Hone{\Q_{cyc,p}}{\omega^i}.
} 
\end{equation*}
Since $\eta\omega^i$ is even, $\Hone{\cy{\Q}}{\eta\omega^i}$ has $\mu$-invariant equal to 0, by Ferrero-Washington theorem.
Again, by Ferrero-Washington theorem, the $\mu$-invariant of $ker[\Hone{\cy{\Q}}{\omega^i}\lra\Hone{\Q_{cyc,p}}{\omega^i}]$ is equal to 0. Therefore, the $\mu$-invariant of the kernel $\sg{\cy\Q}{\Ind{M}{\Q}{\xi\psi_c}}$ is equal to 0.

Let $i$ be \emph{even}. Then $\omega^i$ is even and $\eta\omega^i$ is odd. 
We then consider the following commutative diagram of exact sequences:
\begin{equation*}
\xymatrix{
\Hone{\cy{\Q}}{\omega^i}\ar[r]\ar[d] &\Hone{\cy{\Q}}{\omega^i\oplus\eta\omega^i}\ar[r]\ar[d] 
                                         &\Hone{\cy{\Q}}{\eta\omega^i}\ar[d]\\
0                           \ar[r] &\Hone{\Q_{cyc,p}}{\eta\omega^i}\ar[r]^\cong &\Hone{\Q_{cyc,p}}{\eta\omega^i}.
} 
\end{equation*}
As $\omega^i$ is even, $\Hone{\cy{\Q}}{\omega^i}$ has $\mu$-invariant equal to 0, by Ferrero-Washington theorem.
Again, by Ferrero-Washington theorem, the $\mu$-invariant of $ker[\Hone{\cy{\Q}}{\eta\omega^i}\lra\Hone{\Q_{cyc,p}}{\eta\omega^i}]$ is equal to 0. This is proved in \cite[Lemma 5.9]{gr-cetraro}. Therefore, the $\mu$-invariant of the kernel 
$\sg{\cy\Q}{\Ind{M}{\Q}{\xi\psi_c}}$ is equal to 0.

Similarly, the $\mu$-invariant of $\sg{\cy\Q}{\Ind{M}{\Q}{\xi\psi}}$ can be seen to be 0 when $\Ind{M}{\Q}{\xi\psi}$ is reducible.

Let $M$ be a real quadratic extension, and 
$\Ind{M}{\Q}{\xi\psi_\tau}$ is \emph{reducible}, where $\xi$ is now a character of $\gal_M$ and $\tau$
the character of order 2 that defines $M$.
As $\psi$ runs over all the irreducible representations of $\Delta$, we might as well consider $\psi_\tau=\xi^{-1}$.
Then $\Ind{M}{\Q}{\xi\psi_\tau}=\Ind{M}{\Q}{\mathbf{1}}$, for the trivial character $\mathbf{1}$ of $\gal_M$. Then, by Mackey's
theorem, $\Ind{M}{\Q}{\mathbf{1}}=\mathbf1\oplus\tau$, for the quadratic character $\tau$ of $M$. Then for the two real places, say, $c,c'$ of $\gal_M$, $\mathbf{1}(c)=1=\mathbf{1}(c')$. Therefore, $\Ind{M}{\Q}{\mathbf{1}}$ is \emph{even}.
The twist $\Ind{M}{\Q}{\mathbf1}\otimes\omega$ is still even, but the components $\omega$ and $\tau\omega$ are odd.
Again, by \cite[Lemma 5.9]{gr-cetraro}, the $\FF_p[[\Gamma]]$-corank of $\Hone{\cy\Q}{\Ind{M}{\Q}{\mathbf1}\otimes\omega}$ is 2, and that of 
$\Hone{\Q_{cyc,p}}{\omega}$ is 1. It follows that $\sgart{\cy\Q}{\Ind{M}{\Q}{\mathbf1}\otimes\omega}$ has a positive 
$\FF_p[[\Gamma]]$-corank and therefore a positive $\mu$-invariant.
\end{proof}

\begin{theorem}\label{elliptic-CM}
 Let $E$ be an elliptic curve with CM which is defined over $\Q$. Let $p$ be a prime of good ordinary reduction for $E$, and $\rho_E$ be the representation 
 of $\gal_\Q$ attached to the Tate module of $E$. Let $K=\Q(E[p])$ be the field of $p$-torsion points of $E$, $\inft{K}=\Q(E[p^\infty])$, $\cG=\gal(\inft{K}/K)$, and $\cH=\gal(\inft{K})/\cy{K}$. Let $\Delta=\gal(K/\Q)$.
 Then, assuming Iwasawa's conjecture 
 on the vanishing of the $\mu$-invariant of $\gal(\cy{L}/\cy{K})$, the dual Selmer group $\sgd{\cy{K}}{{\rho_E}}$ has $\mu$-invariant equal to zero. 
 In other words, Iwasawa's conjecture 
 on the vanishing of the $\mu$-invariant of $\gal(\cy{L}/\cy{K})$ implies that the dual Selmer group $\sgd{\cy{K}}{{\rho_E}}$
  is in the category $\mathfrak{M}_\cH(\cG)$.
\end{theorem}
\begin{proof}It is well known that $\cG$ is of dimension 2 as a $p$-adic Lie group, and both $\cH$ and $\Gamma\cong\cG/\cH$ are of
dimension one. By \cite[Cor 5.5]{cfksv}, it is enough to show that the $\mu$-invariant of $\sgd{\cy K}{\rho_E}$ is zero.
For this, by \cite[Prop 4.2.5]{gr-modular}, it is enough to show that $\sg{\cy\Q}{{\ol\rho_E}\otimes\alpha}$ is finite for all irreducible representations $\alpha$ of $\Delta$ which are defined over  $\FF_p$. 
By the previous theorem, it is sufficient to show
that the $\mu$-invariant of $\sg{\cy\Q}{\Ind{M}{\Q}{\xi^r}\otimes\omega^a}$ is zero for all $0<r,a\leq(p-1)$.

Consider an Artin representation $\rho_1=\Ind{\Q}{M}{\wt\xi}$, which corresponds to an eigenform of weight one, and which is 
 present in the Hida family of $\ol\rho_E$. As mentioned above, $\ol\rho_1\cong\ol\rho_E$ and the representations $\Ind{\Q}{M}{\ol\theta^r}\cong\Ind{\Q}{M}{\ol\xi^r}$. 
If the representation $\Ind{M}{\Q}{\theta^r}$ is irreducible, then the $\mu$-invariant of $\sg{\cy\Q}{\Ind{M}{\Q}{\theta^r}\otimes\omega^a}$ is zero if and only if it is zero for $\sg{\cy\Q}{\Ind{M}{\Q}{\xi^r}\otimes\omega^a}$. If it were reducible,
then by Remark \ref{reducible-mu}, the $\mu$-invariant of $\sg{\cy\Q}{\Ind{M}{\Q}{\xi^r}\otimes\omega^a}$ is zero.
Finally, under the assumption that the $\mu$-invariant of $\gal(\cy{L}/\cy{K})$ is zero, it follows that 
$\gal(\cy{L'}/\cy{K'})$ for any Galois subfield $K'\subseteq K$ with $L'$ the maximal pro-$p$ abelian extension of $K'$ unramified
everywhere. Taking $K'$ to be the field cut-out by $ker(\Ind{M}{\Q}{\xi^r}\otimes\omega^a)$ in 
Theorem \ref{iwasawa-mhg}, the $\mu$-invariant of $\sg{\cy\Q}{\Ind{M}{\Q}{\xi^r}\otimes\omega^a}$ is zero.
\end{proof}
\subsection{Example: Adjoint of some CM elliptic curves}
So far, we have assumed that the initial representation of $\gal_F$ over a finite field is absolutely irreducible. This guarantees the existence of the universal deformation 
rings. We now consider one case where the representation is reducible and study it with respect to the conjecture on the category $\mathfrak{M}_\cH(\cG)$.

Let $E$ be an elliptic curve defined over $\Q$ with CM by an imaginary quadratic field $M=\Q(\sqrt{-D})$. Let $p$ be an \emph{odd} prime where $E$ has 
ordinary reduction. 
Let $\xi$ be the $p$-adic character of the Galois group $\gal_M$ attached to the CM elliptic curve $E$, and $\xi^c$ be the complex conjugate of $\xi$. Then the two dimensional 
representation of $\gal_{\Q}$ attached to the Tate module of $E$ is isomorphic to $\Indm{\xi}$.
Let $\eta$ be the Legendre symbol $(\frac{-D}{})$. We denote the reduction of $\eta,\xi$ and $\xi^c$ modulo $p$ by $\bar\eta,\bar\xi$ and $\bar\xi^c$.
Then, we have the following decomposition 
of $\gal_{\Q}$-modules:
\begin{equation}\label{decomp-adjoint}
\ad{E[p]} \cong\FF_p(\bar\eta) \oplus \Indm{\bar\xi/\bar\xi^c},
\end{equation} 
where $\Indm{\bar\xi/\bar\xi^c}$ is the induced module of $\FF_p(\bar\xi/\bar\xi^c)$. Here 
$\bar\xi$ denotes the reduction modulo $p$ of $\xi$ which is viewed as a continuous $p$-adic Hecke character. 
The following isomorphism is also easy to see (\cite[Lemma 3.13 (iii), (iv)]{chan-fine}).
\begin{equation*}
\ad{E[p]}\ol\chi_p \cong\FF_p(\bar\eta\ol\chi_p) \oplus \Indm{\bar\xi^2}, 
\end{equation*}
Consider the field extensions
\begin{equation*}
 K:=\Q(E[p]),\quad \cy K:=K\cy\Q, \quad F_\infty:=K(E[p^\infty]) \quad \mbox{ and }\quad \Delta=G(K/\Q). 
\end{equation*}
where $\cy\Q$ is the cyclotomic $\Z_p$-extension of $\Q$. 
Let $\cG:=\gal(F^\infty/K)$ and $\cH:=\gal(F^\infty/\cy K)$. 
Then the dual of Selmer group of $\ad{E[p^\infty]}$ over $F_\infty$ satisfies the $\mathfrak M_\cH(\cG)$-conjecture arising from 
\emph{non-commutative} Iwasawa theory if and only if it is finitely generated over the Iwasawa algebra $\Z_p[[\cH]]$ (\cite[Conj 5.1]{cfksv}). This 
last assertion is equivalent to the vanishing of the $\mu$-invariant of the dual of Selmer group of $\ad{E[p^\infty]}$ over $\cy K$ (\cite[Cor 5.5]{cfksv}).

Let $\xi':=\bxi/\bxi^c$. Then we have the following exact sequence of $G_p$-modules, which defines a local condition at $p$ for $\Indm{\FF_p(\xi')}$:
\begin{equation}\label{fil-CM}
 0\lra \FF_p(\xi')\lra\Indm{\FF_p(\xi')}\lra\FF_p(\xi')^c\lra 0.
\end{equation}
Using this local condition, we define the Selmer group $\sg{\cy\Q}{\Indm{\xi'}}$ 
as follows:
\begin{equation*}
\begin{split}
\sg{\cy\Q}{\Indm{\xi'}}:=ker\{    &\Hone{\cy\Q}{\Indm{\FF_p(\xi')}}  \\
            \lra    &\oplus_{l\neq p}\Hone{(\cy\Q)_l}{\Indm{\FF_p(\xi')}} \oplus\Hone{(\cy\Q)_p}{\FF_p(\xi')^c}\}.
\end{split}
\end{equation*}
By restriction, the above exact sequence induces the following exact sequence of $G_\p$-modules for every prime $\p$ lying above $p$:
\begin{equation}\label{fil-CM-K}
 0\lra \FF_p(\xi')\lra\Indm{\FF_p(\xi')}\lra\FF_p(\xi')^c\lra 0.
\end{equation}
Over the field $\cy K$, we can also define a Selmer group $\sg{\cy K}{\Indm{\xi'}}$ as follows:
\begin{equation*}
\begin{split}
\sg{\cy K}{\Indm{\xi'}}:=ker\{    &\Hone{\cy K}{\Indm{\FF_p(\xi')}}  \\
            \lra    &\oplus_{v\nmid p}\Hone{(\cy K)_v}{\Indm{\FF_p(\xi')}} \oplus_{\p|p}\Hone{(\cy K)_\p}{\FF_p(\xi')^c}\}.
\end{split}
\end{equation*}
Then 
the above decomposition of the representation $\ad{E[p]}$ and $\ad{E[p]}\bar\chi_p$ induces the following decomposition:
\begin{equation}\label{selmer-decomposition}
\begin{split}
 \sg{\cy K}{\ad{E[p]}} &\cong\sg{\cy K}{\Indm{\xi'}}\oplus\sg{\cy K}{\bar\eta},\\
 \sg{\cy K}{\ad{E[p]}\bar\chi_p} &\cong\sg{\cy K}{\Indm{\bar\xi^2}}\oplus\sg{\cy K}{\bar\eta\bar\chi_p}
\end{split}
\end{equation}
where $\sg{\cy K}{\bar\eta}$ is the Selmer group of $\bar\eta$. It is defined as follows
\begin{equation*}
\sg{\cy K}{\bar\eta}=ker\{\Hone{\cy K}{\FF_p(\bar\eta)}\lra\prod_{v\in\Sigma}\Hone{K_{cyc,v}}{\FF_p(\bar\eta)}\}
\end{equation*}
and the Selmer group $\sg{\cy K}{\bar\eta\bar\chi_p}$ is obtained by replacing the character $\bar\eta$ by $\bar\eta\bar\chi_p$ in the definition.
The first isomorphism is shown in \cite[3.13]{chan-selmer} and the second one also follows similarly.
As $\Indm{\xi'}$ is a trivial $\gal_{K}$-modules, the short exact sequence in \eqref{fil-CM-K} is a short exact sequence of trivial $\gal_{K_\p}$-modules for every $\p\mid p$.

For the representation $\Indm{\bxi^2}$ over $\FF_p$, we can consider the following exact sequence of $G_p$-modules:
\begin{equation}\label{fil-CM-xi2}
 0\lra \FF_p(\xi^2)\lra\Indm{\FF_p(\bxi^2)}\lra\FF_p(\bxi^2)^c\lra 0.
\end{equation}
Using this local condition, we can define a Selmer group $\sg{\cy\Q}{\Indm{\bxi^2}}$ 
as follows:
\begin{equation*}
\begin{split}
\sg{\cy\Q}{\Indm{\bxi^2}}:=ker\{    &\Hone{\cy\Q}{\Indm{\FF_p(\bxi^2)}}  \\
            \lra    &\oplus_{l\neq p}\Hone{(\cy\Q)_l}{\Indm{\FF_p(\bxi^2)}} \oplus\Hone{(\cy\Q)_p}{\FF_p(\bxi^2)^c}\}.
\end{split}
\end{equation*}
By restriction, the above exact sequence induces the following exact sequence of $G_\p$-modules for every prime $\p$ lying above $p$:
\begin{equation}\label{fil-CM-K-xi2}
 0\lra \FF_p(\bxi^2)\lra\Indm{\FF_p(\bxi^2)}\lra\FF_p(\bxi^2)^c\lra 0.
\end{equation}
Over the field $\cy K$, we can also define a Selmer group $\sg{\cy K}{\Indm{\bxi^2}}$ as follows:
\begin{equation*}
\begin{split}
\sg{\cy K}{\Indm{\bxi^2}}:=ker\{    &\Hone{\cy K}{\Indm{\FF_p(\bxi^2)}}  \\
            \lra    &\oplus_{v\nmid p}\Hone{(\cy K)_v}{\Indm{\FF_p(\bxi^2)}} \oplus_{\p|p}\Hone{(\cy K)_\p}{\FF_p(\bxi^2)^c}\}.
\end{split}
\end{equation*}
As $\Indm{\bxi^2}$ is a trivial $\gal_{K}$-module, the short exact sequence in \eqref{fil-CM-K-xi2} is a short exact sequence of trivial $G_\p$-modules for every 
$\p\mid p$.
Similarly, the Selmer group $\sg{\cy K}{\Indm{\bxi}}$ defined by
\begin{equation*}
\begin{split}
\sg{\cy K}{\Indm{\bxi}}:=ker\{    &\Hone{\cy K}{\Indm{\FF_p(\bxi)}}  \\
            \lra    &\oplus_{v\nmid p}\Hone{(\cy K)_v}{\Indm{\FF_p(\bxi)}} \oplus_{\p|p}\Hone{(\cy K)_\p}{\FF_p(\bxi)^c}\}
\end{split}
\end{equation*}
is also a Selmer group of trivial $\gal_K$-representations. Here, the local conditions defining the Selmer group of the trivial representations $\Indm{\bxi^2}\mid_{\gal_K}$
and $\Indm{\bxi}\mid_{\gal_K}$
over $G_\p$ are the same. Therefore, $\sg{\cy K}{\Indm{\bxi}}$ and $\sg{\cy K}{\Indm{\bxi^2}}$ are isomorphic as modules. Hence $\sg{\cy K}{\Indm{\bxi^2}}$ is finite if and only
if $\sg{\cy K}{\Indm{\bxi}}$ is finite. 

By \cite[Theorem 6.3]{chan-selmer} and the paragraph following it, the Selmer group $\sg{\cy K}{\Indm{\bxi}}$ for the CM 
elliptic curves $256a1,256a2,256d1,256d2,121b1,121b2$ which have good ordinary reduction at the prime $p=3$ is finite. 

The Selmer group $\sg{\cy K}{\bar\eta\bar\chi_p}$ is finite as Iwasawa's conjecture on the vanishing of $\mu$-invariant over $\cy K$ holds \cite[Cor 3.12]{chan-fine}. 
\begin{theorem}
Consider the elliptic curves $E=256a1,256a2,256d1,256d2,121b1,121b2$ which have good ordinary reduction at $p=3$. Then the Selmer group 
$\sg{\cy K}{\ad{E[p^\infty]}}[p]$ is finite. It follows that $\sg{\inft K}{\ad{E[p^\infty]}}$ belongs to the category $\mathfrak{M}_\cH(\cG)$.
\end{theorem}

\subsection{Periods of adjoint Galois representations}
We briefly recall the periods of the representation $\ad{\rho}$ where $\rho$ is the representation of $\mathrm{Gal}_F$ that is associated to the Hilbert 
modular form $f\in S_\kappa(\fN,\veps;W)$. 
Let $M_f$ denote the motive associated to the Hilbert modular cusp form $f$ over $E$. Let $c^+_\infty(M_f)$ and $c^\pm_p(M_f)$ denote the Deligne periods and 
the $p$-adic periods of $M_f$. Then, by \cite[Theorems 5.2.1(ii), 5.2.2]{hida-sgl}, 
we have
\begin{equation*}
 c^\pm_p(\ad{M_f}(1))=c^+_p(M_f(1))c^-_p(M_f)\delta_p(M_f(1)).
\end{equation*}
Let $\psi$ be any Artin representation of the Galois group $\gal_E$. Then $\ad{\rho_f}\otimes\psi$ is also critical at $0, 1$. Let $d_\psi$ 
denote the dimension of $\psi$ and 
$d_\pm$ be the dimension of the $\pm$-eigenspaces of the action of complex conjugation on $\psi$. Then 
\begin{equation*}
 c^\pm_p(\ad{M_f}\otimes\psi(1))=(2\pi\imath c^\pm_p(\ad{M_f}(1)))^{d_\psi}.
\end{equation*}
It is conjectured in \cite{deligne} that 
\begin{equation}\label{algebraic}
 \dfrac{L(\ad{\rho_f}(1),\psi,0)}{(2\pi\imath c^\pm_p(\ad{M_f}(1)))^{d_\psi}}\in\bar\Q.
\end{equation}
Here, we recall that the L-function $L(\ad{\rho_f}(1)\otimes\psi,s)$ is defined to be the Euler product defined as reciprocal of the product of the following polynomials:
\begin{eqnarray}\label{hecke-polynomials}
P_\q(\ad{\rho_f},\psi,T)                        &:=&  \det(1-Frob_\q^{-1}T\mid(\ad{\rho_f}\otimes\psi)^{I_\q})\in\cO[T], \q\neq\p;\\
 P_\p(\ad{\rho_f},\psi,T)                       &:=&  \det(1-Frob_\p^{-1}T\mid(\ad{\rho_f}\otimes\psi)^{I_\p})\in\cO[T];\\
 P_\p(\mathcal{F}_\p^+\ad{\rho_f},\psi,T)       &:=&  \det(1-Frob_\p^{-1}T\mid((\mathcal{F}_\p^+\ad{\rho_f})\otimes\psi)^{I_\p})\in\cO[T];\\
 P_\p((\mathcal{F}_\p^+\ad{\rho_f})^\ast,\psi,T)&:=&  \det(1-Frob_\p^{-1}T\mid((\mathcal{F}_\p^+\ad{\rho_f})^\ast\otimes\psi)^{I_\p})\in\cO[T],
\end{eqnarray}
where $\cO$ is a finite extension of $\Z_p$, $I_\q$ denotes the inertia subgroup at the prime $\q$, and $(\mathcal{F}_\p^+\ad{\rho_f})^\ast$ denotes the contragredient 
representation of $\mathcal{F}_\p^+\ad{\rho_f}$.
\subsection{Non-commutative Main conjecture}
The noncommutative Main conjecture of Iwasawa theory predicts that there is an element in the group $\kone{\IIG_{\sS^*}}$ such that its image 
under the connecting homomorphism of $K$-theory gives rise to the class of the dual Selmer group. 
\begin{defn}\label{sk-def}
Let $\II=\cO[[X_1,\cdots,X_r]]$, and 
$\mathbb{K}=K[[X_1,\cdots,X_r]]$, where $K$  is the quotient field of $\cO$. 
For any \textbf{\emph{finite}} group $\cG$, we consider the following groups (see \cite[Page 173]{ol}):
\begin{equation*}
 \begin{split}
 SK_1(\IIG)&:=\mathrm{ker}\left[\kone{\IIG}\lra\kone{\mathbb{K}[[\cG]]}\right],\\
  \konep{\IIG}&:=\kone{\IIG}/S\kone{\IIG}.
 \end{split}
\end{equation*}
\end{defn}
From the localization sequence \eqref{localization}, we get the following exact sequence
\begin{equation*}
 \konep{\IIG}\lra\konep{\IIG_\sS}\stackrel{\wt\partial}{\lra} K_0(\IIG,\IIG_\sS)\lra 0.
\end{equation*}

Now consider a specialization map $\phi_\kappa:\II\lra\cO$. Then this induces the following map 
\begin{equation*}
 \IIG_\sS\stackrel{\phi_\kappa}{\lra}\cO[[\cG]]_{S},
\end{equation*}
where $S$ is the multiplicative set in $\cO[[\cG]]$. This further induces the homomorphisms in the following commutative diagram
\begin{equation*}
 \xymatrix{
 \konep{\IIG_\sS}\ar[r]^{\wt\partial}\ar[d]_{\wt\phi_\kappa}  & K_0(\IIG,\IIG_\sS)\ar[d]\\
 \konep{\cO[[\cG]]_S}\ar[r]^{\partial}       & K_0(\cO[[\cG]],\cO[[\cG]]_S).
 }
\end{equation*}
Now let $\psi$ be any Artin representation of $\cG$, say $\psi:\cG\lra GL_n(\cO')$. Then this induces the following homomorphism of rings
\begin{equation*}
 \psi:\cO[[\cG]]\lra M_n(\cO''[[\Gamma]]),
\end{equation*}
for some finite extension $\cO''$ of $\cO$ and $\cO'$. Further, we have the following homomorphism
\begin{equation*}
 \Phi_\psi:\cO[[\cG]]_S\lra M_n(Q_{\cO''}(\Gamma)),
\end{equation*}
where $Q_{\cO''}(\Gamma)$ is the quotient field of $\cO''[[\Gamma]]$. Therefore, we have 
\begin{equation*}
 \konep{\cO[[\cG]]_S}\lra \kone{M_n(Q_{\cO''}(\Gamma)}\cong Q_{\cO''}(\Gamma)^\times.
\end{equation*}
Now, let $\varphi$ be the augmentation map $\cO[[\cG]]$ to $\cO$, and $\p$ be the kernel of this map. Then the map $\varphi$ can be extended to the map 
$\varphi':Q_{\cO''}(\Gamma)\lra L\cup\{\infty\}$, for some finite extension $L$ of $\Q_p$ and by putting $\varphi(x)=\infty$, if $x\notin\cO[[\cG]]_\p$.
The composition of the map $\wt\phi_\kappa$ in the above commutative diagram with the map $\varphi'$ gives us a map 
\begin{equation}\label{evaluation}
\begin{split}
 \konep{\IIG_\sS} & \lra L\cup\{\infty\}\\
               x & \mapsto x(\psi).
 \end{split}
\end{equation}
This map satisfies the following properties:
\begin{enumerate}
 \item Let $\cG'$ be an open subgroup of $\cG$. Let $\chi$ be a one dimensional representation of $\cG'$ and $\psi:=\mathrm{Ind}_{\cG'}^{\cG}\chi$. Consider 
 the norm map $\konep{\IIG_{\sS^\ast}}\lra\konep{\II[[\cG']]_{\sS}}$. Then for any $\wt x\in \konep{\IIG_{\sS^\ast}}$, we have
 \begin{equation*}
  \wt x(\psi)=N(\wt x)(\chi).
 \end{equation*}
 \item Let $\psi_1:\cG\lra GL_{n_1}(L)$ and $ \psi_2:\cG\lra GL_{n_2}(L)$ be two Artin representations for some field extension $L$ of $\Q_p$. Then for 
 any $\wt x\in\konep{\IIG_{\sS^\ast}}$, we have
 \begin{equation*}
  \wt x(\psi_1\oplus\psi_2)=\wt x(\psi_1)\wt x(\psi_2).
 \end{equation*}
 \item Let $U$ be a subgroup of $\cH$ which is normal in $\cG$. Then the homomorphism $\cG\lra\cG/U$ induces the homomorphism $\IIG_\sS\lra\II[[\cG/U]]_\sS$.
 Further, we get the homomorphism $\pi:\konep{\IIG_\sS}\lra\konep{\II[[\cG/U]]}$. Let $\psi:\cG/U\lra GL_n(L)$. Then we get an Artin representation \quad
 $\mathrm{inf}(\psi):\cG\lra\cG/U\lra GL_n(L)$. For any $\wt x\in\konep{\IIG_\sS}$, we have
 \begin{equation}
  \wt x(\mathrm{inf}(\psi))=\pi(\wt x)(\psi).
 \end{equation}
\end{enumerate}
\begin{conj}\label{main-conj}(Main Conjecture over $\II[[\cG]]$) Let $f\in S_\kappa(\fN\p,\veps;\cO)$ be obtained through the arithmetic specialization $\phi_\kappa:\II\lra\cO$, 
for some finite extension $\cO$ of $\Z_p$.
Let $\rho_{f}$ be the representation of $\gal_E$ that is associated to $f$, $V:=\ad{\rho_f}$, and $V_\p^+:=\mathcal{F}_\p^+\ad{\rho_f}$.
Let the selmer complex $SC(U,\mathbf T,\{\mathbf T^0_\p\}_{\p|p})$ for the Hida family of $\rho_f$ have $\sS$-torsion cohomologies.
Then  there exists an element $\wt\xi\in\konep{\IIG_\sS}$ such that $\partial(\wt\xi)=-[SC(U,\mathbf T,\{\mathbf T^0_\p\}_{\p|p})]\in K_0(\IIG,\IIG_\sS)$ 
Further, under the map in \eqref{evaluation}, the following interpolation properties are satisfied for all Artin representations $\psi$ of $\cG$ with degree $d_\psi$:
\begin{equation}\label{interpolate}
\begin{split}
 \wt\phi_\kappa(\wt\xi)(\psi)=&\dfrac{L_\Sigma(V(1),\psi,0)}{(2\pi\imath c^\pm_p(\ad{M_f}(1)))^{d_\psi}}\times\\
                              &             \left[\dfrac{P_\p(V\otimes\psi,T)}{P_\p(V_\p^+\otimes\psi,T)}\right]_{T=1}
                                   P_\p((V_\p^+\otimes\psi)^\ast,1)
                                   \prod_{\q\mid\fN,\q\neq\p}P_\q(V\otimes\psi,1).
\end{split}                                   
\end{equation}
Here $L_\Sigma(V(1),\psi,0)$ is the value of the L-function for the twisted adjoint representation with Euler factors for primes in the set $\Sigma:=\{\q\mid\fN\p\}$ 
removed.
\end{conj}
\begin{remark}\label{Gen-MC}
\begin{enumerate}
\item A similar Main conjecture is formulated in the thesis of Barth \cite{barth}. 
 Similar main conjectures can be formulated for any $p$-adic family of nearly ordinary Galois representations. We refer to sections 4.3 and 4.4 \emph{loc. cit } for a discussion 
 about the interpolation properties. See \cite[Theorem 4.1.12]{fk} for the interpolation property for L-values of motives. 
\item\label{torsion-complex}
 It can be shown as in the proof of
 \cite[Prop 4.3.7]{fk}, that the selmer complex $SC(U,\mathbf T,\{\mathbf T^0_\p\}_{\p|p})$ is $\sS^\ast$-torsion if the selmer group $\sgd{\Ein}{\ad{\bosym\rho}}$ is in the 
 category $\mathfrak{M}^\II_\cH(\cG)$.
\end{enumerate}
\end{remark}
\subsection{Main conjecture in the cyclotomic case}
Let $F$ be a totally real field, and $f$ be a Hilbert modular Hecke eigenform of weight $\kappa$,
level $\fN$. We also assume that $f$ is ordinary  at all the primes above the prime $p$. We denote the Galois representation associated to $f$ by $\rho_f$. Let $\bosym\rho$
be the universal ordinary deformation of $\rho_f$.
Let $\cG=\Gamma=\gal(\cy{F}/F)$. 
Then for $\II=\cO[[X_1,\cdots,X_r]]$, we get $\IIG\cong\II[[\Gamma]]$.

In the case $F=\Q$, the $p$-adic L-function of 
$\ad{\bosym\rho}$ over the quotient field of $\IIG$ is constructed by Hida and Tilouine. If $F$ is a totally real field with ring of integers class 
number one, then the $p$-adic L-function has been constructed by Hsin-Tai Wu (\cite{wu}). In general, for totally real number fields, it has been constructed by Rosso in \cite[Theorem 7.2]{rosso}. The Main 
conjecture over the field $\Q$ has been proven for the Selmer group of the adjoint in \cite{urban} and \cite{rosso}.

Let $\cL_p(X_1,\cdots,X_r,T)$ be the $p$-adic L-function of $\ad{\bosym\rho}$, where $T$ is the variable corresponding to the cyclotomic $\Z_p$-extension $\Gamma$.
Let $\mathscr P(X_1,\cdots,X_r,T)$ be the characteristic ideal of the dual Selmer group of $\ad{\bosym\rho}$.
Throughout this section, we assume that the Main conjecture holds over the cyclotomic $\Z_p$-extension of $F$. 
Then, $\cL_p(X_1,\cdots,X_r,T)=\mathscr P(X_1,\cdots,X_r,T)$, \emph{ up to units}. 
We now interpret this equality of the Iwasawa Main conjecture in terms of $K$-theory. 

In this situation, the category 
$\mathfrak{M}_\cH^\II(\cG)$ consists of all finitely generated modules which are $\sS^\ast$-torsion. 
\begin{lemma}
 Let the $\mu$-invariant of the Selmer group $\sgd{\cy{F}}{\ad{\rho_f}}$ be equal to zero. Then $\sgd{\cy{F}}{\ad{\rho_f}}$ is $S$-torsion.
 It follows that $\sgd{\cy{F}}{\ad{\bosym\rho}}$ is $\sS$-torsion,
 and the $p$-adic L-function $\cL_p$ is a unit in $\IIG_\sS$.
\end{lemma}
\begin{proof}
As the $\mu$-invariant is equal to 0, the dual selmer group $\sgd{\cy{F}}{\ad{\rho_f}}$ is $S$-torsion. By Theorem \ref{big-torsion-small-torsion}, 
$\sgd{\cy{F}}{\ad{\bosym\rho}}$ is $\sS$-torsion.
Since the Main conjecture holds, $\cL_p(X_1,\cdots,X_r,T)=\mathscr P(X_1,\cdots,X_r,T)$. 
As $\IIG_\sS$ is the localization at the prime ideal $\m_\II$, $\mathscr P(X_1,\cdots,X_r,T)\in\sS$, and hence $\cL_p(X_1,\cdots,X_r,T)$ is a 
unit in $\IIG_\sS$.
\end{proof}

Let the selmer complex $\mathcal{C}:=SC(U,\mathbf T,\{\mathbf T^0_\p\}_{\p|p})$ be $\sS$-torsion. 
Under the isomorphism $K_0(\II[[T]],\II[[T]]_{\sS})\lra K_0(\mathfrak{M}_\cH^\II(\cG))$, the 
image of $[SC(U,\mathbf T,\{\mathbf T^0_\p\}_{\p|p})]$ is $[\mathrm{H}^0(\mathcal C)]-[\mathrm{H}^1(\mathcal C)]+[\mathrm H^2(\mathcal C)]-[\mathrm H^3(\mathcal C)]$.

Consider the induced module 
$\mathbf{T}_0:=\oplus_{\p\mid p}\Z[\gal(\ol\Q_p/\Q_p)]\otimes_{\Z[\gal(\ol F_{\p}/F_{\p})]}\left(\II[[\Gamma]]\otimes\mathbf{T}_{\p}^{0}\right)$. 
By a straightforward generalization of \cite[Prop 4.3.13]{fk} over $\IIG$, we have 
\begin{equation}
\mathrm{H}^1(\mathcal C)\cong\mathrm{ker}(\mathbf T_0^{\gal_{\inft{F}}}\lra\oplus_{v\mid p}\mathbf T_0/\mathbf T_0^0(v)).
\end{equation}
As the components of $\mathbf{T}_0$ are of rank one over $\II$ with a non-trivial action of  $\gal_{\cy F}$, we have $\mathbf T_0^{\gal_{\cy{F}}}=0$. 
Hence $\mathrm{H}^1(\mathcal{C})=0$. The class $[\mathbf T(-1)_{\gal_{\cy{F}}}]=0\in K_0(\mathfrak{M}_\cH^\II(\cG))$, as 
$\mathbf T(-1)_{\gal_{\cy{F}}}=0$. It follows analogously as in \cite[Prop 4.3.16]{fk}, that
$[\II[[\cG]]\otimes_{\cO[[\cG_v]]}(\mathbf T_v^0(-1))_{\cG(v)}]=0$.

We now consider the following equality obtained from Proposition \ref{selmer-dual-complex-sequence}:
\begin{equation}
\begin{split}
 [\sgd{F}{\mathbf T}] &-[\mathrm{H}^2(SC(U,\mathbf T,\{\mathbf T_\p^0\}_{\p\mid p}))] +\sum_{v\mid p}[\II[[\cG]]\otimes_{\cO[[\cG_v]]}(\mathbf T_v^0(-1))_{\cG(v)}]\\
 &-[\mathbf T(-1)_{\gal_{\inft{F}}}] + [\mathrm{H}^3(SC(U,\mathbf T,\{\mathbf T_\p^0\}_{\p\mid p}))]= 0.
 \end{split}
\end{equation}
As $[\mathrm{H}^0(\mathcal C)]=[\mathrm{H}^1(\mathcal C)]=0$, it follows that the image of the class $[SC(U,\mathbf T,\{\mathbf T^0_\p\}_{\p|p})]$ in 
$K_0(\mathfrak{M}_\cH^\II(\cG))$ is $[\sgd{F}{\mathbf T}]$.

Now, let $Y=\sgd{\cy{F}}{\ad{\bosym\rho}}$. Then $Y$ is a finitely generated $\II[[T]]$-module which is annihilated by an element outside the maximal ideal $\m_\II$ of $\II$, and its characteristic ideal 
$\mathscr P$ belongs to $\II[[T]]_{\m_\II}^\times$. 
Now consider the class 
$\left[\left(Y,0,\frac{\II[[T]]}{\mathscr P\II[[T]]}\right)\right]$ in $K_0(\II[[T]],\II[[T]]_{\m_\II})$. Then, under the connecting homomorphism
\begin{equation*}
 \partial:\kone{\II[[T]]_{\m_\II}}\cong\II[[T]]_{\m_\II}^\times\lra K_0(\II[[T]],\II[[T]]_{\m_\II}),
\end{equation*}
we have $\partial(\mathscr P)=-\left[\left(Y,0,\frac{\II[[T]]}{\mathscr P\II[[T]]}\right)\right]$. Since the  Main Conjecture holds for $\ad{\bosym\rho}$ over 
$\cy{F}$, the characteristic power series $\mathscr P$ is the $p$-adic L-function for $\ad{\bosym\rho}$. We then have the following theorem.
\begin{theorem}
 Let $\sgd{\cy{F}}{\ad{\bosym\rho}}$ be annihilated by an element of $\IIG$ outside the maximal ideal $\m_\II$ of $\II$. Then the connecting homomorphism 
 $\partial:\kone{\IIG_\sS}\lra K_0(\IIG,\IIG_{\m_\II})$ maps the $p$-adic L-function to the class $-[\sgd{\cy{F}}{\ad{\bosym\rho}}]$.
\end{theorem}

\subsection{Remark on zeros}\label{remark-trivial-zeros}
As a consequence of Theorem \ref{trivial-zeros}, we get the following result:
\begin{propn}
 Recall the notations from Theorem \ref{trivial-zeros}. Let $\sgd{\inft{E}}{\ad{\bosym\rho}}$ be $\sS$-torsion. Then we have the following equality in $K_0(\IIG,\IIG_{\sS})$:
 \begin{equation*}
  [\sgd{\inft{E}}{\ad{\bosym\rho}}]=[\Om{\cRin}{\cO[[\inft{D}]]}\otimes_{\cRin}\II]+[\Om{\II}{\cO[[I_0]]}\otimes_{\II}\II].
 \end{equation*}
 If the noncommutative Main Conjecture holds, then the $p$-adic L-function for $[\sgd{\inft{E}}{\ad{\bosym\rho}}]$ arises as the product of the 
 $p$-adic L-function for $[\Om{\cRin}{\cO[[\inft{D}]]}\otimes_{\cRin}\II]$ and $[\Om{\II}{\cO[[I_0]]}\otimes_{\II}\II]$.
\end{propn}
\begin{proof} By Theorem \ref{trivial-zeros}, we have the following short-exact sequence:
\begin{equation*}
 0\lra \sgd{\inft{E}}{\ad{\bosym\rho}} \stackrel{}{\lra} \left(\Om{\cRin}{\cO[[\inft{D}]]}\otimes_{\cRin}\II\right)\oplus\left(\Om{\II}{\cO[[I_0]]}\otimes_{\II}\II\right)
                                                                                                             \stackrel{}{\lra}\Om{\II}{A_0}\otimes_{\II}\II\lra 0.
\end{equation*}
Then, we have the following equality in $K_0(\IIG,\IIG_{\sS})$:
\begin{equation*}
  [\sgd{\inft{E}}{\ad{\bosym\rho}}]+[\Om{\II}{A_0}\otimes_{\II}\II]=[\Om{\cRin}{\cO[[\inft{D}]]}\otimes_{\cRin}\II]+[\Om{\II}{\cO[[I_0]]}\otimes_{\II}\II].
 \end{equation*}
 Since the module $\Om{\II}{A_0}$ is $\II$-torsion, it is the trivial class, and we have,
\begin{equation*}
  [\sgd{\inft{E}}{\ad{\bosym\rho}}]=[\Om{\cRin}{\cO[[\inft{D}]]}\otimes_{\cRin}\II]+[\Om{\II}{\cO[[I_0]]}\otimes_{\II}\II].
\end{equation*}
Assuming the noncommutative Main Conjecture for the classes $[\Om{\cRin}{\cO[[\inft{D}]]}\otimes_{\cRin}\II]$ and $[\Om{\II}{\cO[[I_0]]}\otimes_{\II}\II]$
there exists elements  $\wt\psi,\wt\tau\in\konep{\IIG_\sS}$ such that their images under the connecting homomorphism $\partial$ are the modules 
$[\Om{\cRin}{\cO[[\inft{D}]]}\otimes_{\cRin}\II]$ and $[\Om{\II}{\cO[[I_0]]}\otimes_{\II}\II]$ in $K_0(\IIG,\IIG_{\sS})$. 
Therefore, $\partial$ maps the product $\wt\psi\wt\tau$ to the class $[\sgd{\inft{E}}{\ad{\bosym\rho}}]$.
\end{proof}
\begin{remark}
 Over the cyclotomic $\Z_p$-extension, we saw in Theorem \ref{trivial-zeros} that the $\II[[T]]$-module $\Om{\II}{\cO[[I_0]]}$ is pseudo-isomorphic to $T^s$, where 
 $s=S_F$ is the number of primes of $E$ above $p$. In the case of a $p$-adic Lie extension also, the pre-image of the class $[\Om{\II}{\cO[[I_0]]}\otimes\II]$ occurs as 
 a factor of the $p$-adic L-function of the class $[\sgd{\inft{E}}{\ad{\bosym\rho}}]$.
\end{remark}
\section{$K_1$ computations and congruences over $\IIG$}\label{k-one}
In this section, we extend the strategy that has so far been followed to prove the Main conjecture. This strategy, relying on a description of the $K$-groups, 
was first used by Burns and then by Kato, who used it to prove certain instances of the Main conjecture over number fields. Independently, Ritter and Weiss also proved 
instances of the Main conjecture. 
Inspired by the work of Burns and Kato, Kakde and Hara also proved instances of the Main conjecture for certain $p$-adic Lie 
extensions. 
We extend their results suitably and show that the Main conjecture that we have formulated can also be established for certain $p$-adic families of 
Galois representations. 
\subsection{General strategy}
The strategy involves reducing the proof of the Main conjecture over compact $p$-adic Lie groups to compact $p$-adic Lie groups of dimension one. For this, it is crucial to know 
that the completed group 
ring $\IIG$ is an adic ring. Indeed, if $\m_\II$ is the maximal ideal of $\II$ and $I_\cG$ the augmentation ideal of $\IIG$, then $J_\cG=\m_\II+I_\cG$ is the maximal ideal
of $\IIG$, and $\IIG$ is an adic ring with respect to the ideals $\{J_\cG^n:n\in\N\}$ in the sense of Fukaya-Kato (\cite[1.4.1]{fk}). Then it is shown in Fukaya-Kato 
(\cite[Prop 1.5.1]{fk}) that
\begin{equation*}\label{k-limit-iso}
 \kone{\IIG}\stackrel{\cong}{\lra}\varprojlim_n\kone{\IIG/J_\cG^n}.
\end{equation*}
Following \cite{burns1},  in \cite[\S 4]{kakde}, a series of reduction steps are made showing that the proof of the Main Conjecture for any arbitrary $p$-adic Lie group can 
be reduced to the case when the Galois group $\cG$ has dimension one, with $\cG\cong \Delta\times \cG_p$, where $\Delta$ is a finite cyclic group of order prime to 
$p$ and $\cG_p$ is a pro-$p$ compact $p$-adic Lie group of dimension one. 
The strategy that we are going to outline is only for the case when $\cG$ has dimension one. We believe that similar reduction steps can be done as in \cite[\S 4]{kakde}.

Consider the Iwasawa algebra $\II\cong\cO[[X_1,\cdots,X_r]]$, for some $r$. 
Then $\IIG\cong\prod_{\psi\in\widehat\Delta}\II[\psi](\cG_p)$, where $\II[\psi]$ is the algebra obtained by adjoining the values of $\psi$ to $\II$. This further
allows one to reduce the proof of the Main conjecture to the case when $\cG$ is a pro-$p$, compact $p$-adic Lie group of dimension one. 
Recall from Definition \ref{sk-def}, that for $\mathbb{K}=K[[X_1,\cdots,X_r]]$, with the quotient field $K$  of $\cO$, 
\begin{equation*}
 \begin{split}
 SK_1(\IIG)&:=\mathrm{ker}\left[\kone{\IIG}\lra\kone{\mathbb{K}[[\cG]]}\right],\\
  \konep{\IIG}&:=\kone{\IIG}/S\kone{\IIG}.
 \end{split}
\end{equation*}
\begin{defn}\label{Sigma_G}
Let $\cG$ be a 
$p$-adic Lie group of dimension 1. We denote by $\Sigma(\cG)$ any set of rank 1 subquotients of $\cG$ of the form $U^{ab}$ with $U$ an open subgroup of 
$\cG$ that has the following property:
\begin{description}
 \item[($\ast$)] For each Artin representation $\rho$ of $\cG$, there is a finite subset $\{U^{ab}_i:i\in I\}$ of $\Sigma(\cG)$ and for each index 
$i$ an integer $m_i$ and a degree one representation $\rho_i$ of $U^{ab}$ such that there is an isomorphism of virtual representations 
$\rho\cong\sum_{i\in I}m_i.\Ind{\cG}{U_i}{\Ind{U_i}{U_i^{ab}}{\rho_i}}$. 
\end{description}
\end{defn}
Let $U^{ab}$ be a subquotient satisfying the above property $(\ast)$. 
Then, we have the following natural homomorphism,
\begin{equation}
 \konep{\IIG_\sS}\lra \konep{\II(U)_\sS}\lra \kone{\II(U^{ab})_\sS}\lra\II(U^{ab})_\sS^\times\subset Q_\II(U^{ab})^\times.
\end{equation}
Taking all the $U^{ab}$ in $\Sigma(\cG)$ we get the following homomorphism
\begin{equation}
 \Theta_{\Sigma(\cG)}:\konep{\IIG}\lra\prod_{U^{ab}\in\Sigma(\cG)}Q_\II(U^{ab})^\times.
\end{equation}

\begin{propn} Let the $\mu$-invariant of the dual Selmer group $\sgd{\cy{E}}{\ad{\bosym\rho}}$ be equal to zero.
 Then the Main Conjecture in \ref{main-conj} is valid if and only if for 
 any set of subquotients $\Sigma(\cG)$ with the property (*) above in Definition \ref{Sigma_G}, 
 the following two conditions hold:
 \begin{enumerate}
 \item there exists a subgroup $\Phi$ of $\prod_{U^{ab}\in\Sigma(\cG)}\II(U^{ab})^\times$ such that $\Theta_{\Sigma(\cG)}:\kone{\IIG}{\lra}\Phi$ is an isomorphism;
 \item there exists a subgroup $\Phi_\sS$ of $\prod_{U^{ab}\in\Sigma(\cG)}\II(U^{ab})_\sS^\times$ such that $\Phi_\sS\cap(\prod_{U^{ab}\in\Sigma(\cG)}\II(U^{ab})^\times)=\Phi$ 
 and $\Theta'_{\Sigma(\cG)}(\kone{\IIG_\sS})\subset\Phi_\sS$.
 \end{enumerate}
\end{propn}
\begin{proof}
Let $C:=[\sgd{\Ein}{\ad{\bosym\rho}}]\in K_0(\IIG,\IIG_\sS)$. Consider the following commutative diagram:
\begin{equation}
 \xymatrix{
 \konep{\IIG}\ar[r]\ar[d]^{\Theta_{\Sigma(\cG)}} & \konep{\IIG_\sS}\ar[r]^{\wt\partial_\cG}\ar[d]^{\Theta'_{\Sigma(\cG)}} & K_0(\IIG,\IIG_\sS)\ar[r]\ar[d]^{\Theta_0} & 0\\
 \prod\limits_{U^{ab}\in\Sigma(\cG)}\konep{\II(U^{ab})}\ar[r]       
         & \prod\limits_{U^{ab}\in\Sigma(\cG)}\konep{\II(U^{ab})_\sS}\ar[r]^{\!\!\!\!\!\partial_\cG} & \prod\limits_{U^{ab}\in\Sigma(\cG)}K_0(\II(U^{ab}),\II(U^{ab})_\sS)\ar[r] 
                                                                                                                                                                      & 0
 }
\end{equation}
Let $g$ be any element in $\konep{\IIG_\sS}$ such that $\wt\partial_\cG(g)=-C$. Since the Main conjecture is valid for the extension $E^{U^{ab}}/E$, there exists $\xi_{U^{ab}}$
such that it is the pre-image of the class $-[\sgd{E^{U^{ab}}}{\ad{\bosym\rho}}]$. On the other hand, the commutativity of the square on the left also implies that 
under the map $\Theta'_{\Sigma(\cG)}$ the element $\left(g_{U^{ab}}\right)$ is mapped to $-[\sgd{E^{U^{ab}}}{\ad{\bosym\rho}}]$. Therefore the element 
$(g_{U^{ab}}^{-1}\xi_{U^{ab}})$ comes from the group $\prod_{U^{ab}\in\Sigma(\cG)}\konep{\II(U^{ab})}=\prod_{U^{ab}\in\Sigma(\cG)}{\II(U^{ab})}^\times$. The second condition
therefore implies that $(g_{U^{ab}}^{-1}\xi_{U^{ab}})\in\Phi$. 

By the isomorphism in the first condition, we find that there exists $u\in\konep{\IIG}$, such that $\Theta_{\Sigma(\cG)}(u)=(g_{U^{ab}}^{-1}\xi_{U^{ab}})$. Since the map 
$\Theta_{\Sigma(\cG)}$ is injective, the map $\konep{\IIG} \lra \konep{\IIG_\sS}$ is injective. 



Now, we define $\xi_\cG:=ug$, and we claim that 
this is the $p$-adic L-function defined over $\IIG$ that satisfies the interpolation formula. 
Clearly, $\wt\partial_\cG(\xi_\cG)=\wt\partial_\cG(u)+\wt\partial_\cG(g)=\wt\partial_\cG(g)=-C$ as $u$ comes from an element of $\konep{\IIG}$.
For the interpolation formula, for any Artin representation $\rho$ of $\cG$, consider the isomorphism $\rho\cong\sum_{i\in I}m_i.\Ind{\cG}{U_i}{\Ind{U_i}{U_i^{ab}}{\rho_i}}$
of virtual representations given by the condition (*) above. Then, we have,
\begin{equation*}
 \phi_\kappa(\xi_\cG)(\rho)=\prod_{i\in I}\phi_\kappa(\xi_\cG)(\Ind{\cG}{U_i}{\Ind{U_i}{U_i^{ab}}{\rho_i}})^{m_i}
                           =\prod_{i\in I}\phi_\kappa(\xi_{U_i^{ab}})(\rho_i)^{m_i}.
\end{equation*}
On the other hand, if the Main conjecture is true over the extension, then there exists $\xi\in\konep{\IIG_\sS}$. 
Let $\Theta'_{\Sigma(\cG)}(\xi)=(\xi_{U^{ab}})\in\prod_{\Sigma(\cG)}(\II(U^{ab}_\sS)^\times$. Note that the image $(\xi_{U^{ab}})\in\Phi_\sS$. 
By the interpolation formula, it is easy to 
see that the element $\xi_{U^{ab}}$ is the $p$-adic L-function over $U^{ab}$. Therefore $\xi_{U^{ab}}\in\Phi_\sS$.

This finishes the proof of the proposition.
\end{proof}
\begin{defn} Fix a lift $\wt\Gamma$ of $\Gamma$ in $\cG$. Then we can identify $\cG$ with $H\rtimes\Gamma$. Fix $e\in\N$ such that 
$\wt\Gamma^{p^e}\subset Z(\cG)$,
and put $\ol\cG:=\cG/{\Gamma^{p^e}}$ and $\sR:=\II(\wt\Gamma^{p^e})$. Then $\IIG\cong \sR[\ol\cG]^\tau$, the twisted group ring with multiplication 
\begin{equation*}
(h\wt\gamma^a)^\tau(h\wt\gamma^b)^\tau=\wt\gamma^{p^e[\frac{a+b}{b}]}(h\wt\gamma^a.h'\wt\gamma^b)^\tau,
\end{equation*}                                                                                                         
where $g^\tau$ is the image of $g\in G$ in $\sR[\ol\cG]^\tau$ (\cite[\S 5.1.1, \S 5.1.2]{kakde}).
\end{defn}
The Ore set $\sS$ that we have considered in the formulation of the Main Conjecture over $\IIG$ contains a multiplicative set which is crucial in setting up the 
strategy to prove the Conjecture.
\begin{lemma}\label{center-ore} Let $Z:=Z(\cG)$ and $\II(Z):=\II[[Z]]$. 
 Consider the subset $T=\II(Z)\setminus p \II(Z)$. 
 Then $T$ is a multiplicatively closed left and right Ore set of $\IIG$.
 Further, the inclusion of rings $\IIG_T\inj\IIG_\sS$ is an isomorphism.
\end{lemma}
\begin{proof}
 As $Z$ is central in $\cG$, it is easy to see that $T$ is a multiplicatively Ore set. Further, $T$ has no zero-divisors as it is contained in the domain $\II(Z)$. Therefore,
 the natural map $\IIG_T\lra\IIG_\sS$ induced by the inclusion $T\inj\sS$ is an injective.
 
 For surjectivity, consider the equality $\IIG_T=\II(Z)_T\otimes_{\II(Z)}\IIG$. We first observe that $Q(\II(Z))\otimes_{\II(Z)}\IIG= Q(\IIG)$. Indeed, it is easy to see that
 \begin{equation*}
  Q(\II(Z))\otimes_{\II(Z)}\IIG\inj Q(\IIG).
 \end{equation*}
Further, the ring 
 $\IIG=\II(Z)[\ol\cG]$ is a module of finite rank over $\II(Z)$ and $Q(\II(Z))$ is a field, so 
 the ring $Q(\II(Z))\otimes_{\II(Z)}\IIG$ is Artinian. It follows that every regular element is a unit. The inclusion $\IIG\inj Q(\II(Z))\otimes_{\II(Z)}\IIG$, then implies that
 every regular element of $\IIG$ is invertible in $Q(\II(Z))\otimes_{\II(Z)}\IIG$. It follows that the inclusion $Q(\II(Z))\otimes_{\II(Z)}\IIG\inj Q(\IIG)$ is surjective.
 
 Finally, if $x\in\IIG_\sS\subset Q(\IIG)$, then $x=a/t$, for some $a\in\IIG$ and $t\in\II(Z)$ with $t\neq 0$. Here, if $t\in p^n\II(Z)$, then $tx=a\in p^n\IIG_\sS$. Since 
 $a\in p^n\IIG$, it follows that $a\in p^n\IIG_\sS$. Cancelling the powers of $p$ from $a$ and $t$, the element $x=a'/t'$ with $t'\in T$. 
\end{proof}
\begin{remark}
The $p$-adic completion of $\IIG_T$ is denoted by $\widehat\IIG_T=\widehat{\II(Z)}_T[\ol\cG]^\tau$. We also note that the localizations with respect to $T$ and $p$ are equal
$\II(Z)_T=\II(Z)_{(p)}$.
\end{remark}

Let $P$ be a subgroup of $\ol\cG$ and $U_P$ be the inverse image of $P$ in $\cG$. Let 
\begin{enumerate}
 \item $N_{\ol\cG}P:=$ the normalizer of $P$ in $\cG$,
 \item $W_{\ol\cG}(P):=N_{\ol\cG}P/P$,
 \item $C(\ol\cG):=$ set of cyclic subgroups of $\ol\cG$, 
 \item for $P\in C(\ol\cG)$, the set $C_P(\ol\cG)$ denotes the set of cyclic subgroups $P'$ of $\ol\cG$ with $P'^p=P$ and $P'\neq P$.
\end{enumerate}
If $P\in C(\ol\cG)$, then $U_P$ is a rank one abelian subquotient 
of $\cG$,
and for every $P\in C(\ol\cG)$ set 
\begin{equation}
\sT_P:=\{\sum_{g\in W_{\ol\cG}(P)}g^\tau x(g^\tau)^{-1}\mid x\in \sR[P]^\tau\}. 
\end{equation}
In the same way, we define $\sT_{P,\sS}$ and $\widehat{\sT}_P$.

Let $P\leq P'\leq\ocG$. Then consider the homomorphism $\IIG\lra\IIG$ given by $x\mapsto \sum_{g\in P'/P}\tilde{g}x\tilde{g}^{-1}$, where $\tilde{g}$ is a lift of $g$.
We define $\sT_{P,P'}$ to be the image of this homomorphism. Similarly, we define $\sT_{P,P',\sS}$ and $\widehat\sT_{P,P',\sS}$, by considering the images of the same
map on $\II(U_P)_\sS\lra\II(U_P)_\sS$ and $\widehat{\II(U_P)}_\sS\lra\widehat{\II(U_P)}_\sS$.

For two subgroups $P, P'$ of $\ol\cG$ with $[P',P']\leq P\leq P'$ consider 
\begin{equation}
 \begin{split}
  &\mathrm{Tr}_P^{P'}:\II(U_{P'}^{ab})\lra\II(U_P/[U_{P'},U_{P'}]),\quad\mbox{(the trace map)},\\
  &\mathrm{Nr}_P^{P'}:\II(U_{P'}^{ab})^\times\lra\II(U_P/[U_{P'},U_{P'}])^\times,\quad\mbox{(the norm map)},\\
  &\Pi_P^{P'}:\II(U_{P'}^{ab})\lra\II(U_P/[U_{P'},U_{P'}]),\quad\mbox{(the projection map)}.
 \end{split}
\end{equation}

We also have these maps in the localized case and the $p$-adic completion case. We continue to denote them by $\mathrm{Tr}_P^{P'}, \mathrm{Nr}_P^{P'}$ and $\Pi_P^{P'}$.

Recall the map $\Theta_{\Sigma(\cG)}$. For every subgroup $P$ of $\ol\cG$, let $U_P$ denote the inverse image of $P$ in $\cG$. Then we have the following natural homomorphism
\begin{equation}
 \Theta_P^\cG:\konep{\IIG}\lra\konep{\II(U_P)}\lra\konep{\II(U_P^{ab})}=\II(U_P^{ab})^\times.
\end{equation}
Combining all these homomorphisms, we get the following homomorphism
\begin{equation}
 \Theta^\cG=(\Theta_P^\cG)_{P\leq\ol\cG}:\konep{\IIG}\lra\prod_{P\leq\ol\cG}\II(U_P^{ab})^\times.
\end{equation}
Similarly, we also consider the following homomorphisms:
\begin{equation}
 \Theta_\sS^\cG:\konep{\IIG_\sS}\lra\prod_{P\leq\ol\cG}\II(U_P^{ab})^\times_\sS
\end{equation}
and 
\begin{equation}
 \widehat\Theta^\cG:\konep{\widehat{\IIG}}\lra\prod_{P\leq\ol\cG}\widehat{\II(U_P^{ab})^\times}.
\end{equation}

For $P\in C(\ol\cG)$ with $P\neq (1)$, fix a homomorphism $\omega_P:P\lra\bar\Q_p^\times$ of order $p$, and also a homomorphism 
$\omega_1:=\omega_{\{1\}}:\wt\Gamma^{p^e}\lra\bar\Q_p^\times$ of order $p$. The homomorphism $\omega_P$ induce the following homomorphism which we 
again denote by the same symbol:
\begin{equation}
 \omega_P:\II[\mu_p](U_P)^\times\lra\II[\mu_p](U_P)^\times, g\mapsto \omega_P(g)g.
\end{equation}
For $P\leq \ol\cG$, consider the homomorphism $\alpha_P:\II(U_P)_{\sS}^\times\lra\II(U_P)_{\sS}^\times$ defined by
\begin{equation}
 \alpha_P(x):=\begin{cases}
               x^p &\mbox{ if } P=\{1\}\\
               x^p(\prod_{k=0}^{p-1}\omega_P^k(x))^{-1} &\mbox{ if } P\neq\{1\}\mbox{ and cyclic}\\
               x^p             &\mbox{ if } P \mbox{ is not cyclic}.
              \end{cases}
\end{equation}
Note that, for all $P\leq\ol\cG$, there is an action of $\cG$ and $\ol\cG$ act on $U_P^{ab}$ by conjugation since $\wt\Gamma^{p^e}$ is central. 
The following theorem is a generalization of results of Kakde, Kato, Burns, and  Ritter and Weiss to $\IIG$-modules. 
\begin{theorem}
 Let $\cG$ be a rank one pro-$p$ group. Then the set $\Sigma(\cG):=\{U_P^{ab}:P\leq\ol\cG\}$ satisfies the condition $(\ast)$. Further, an element 
$(\Xi_\cA)_{\cA}\in\prod_{\cA\in\Sigma(\cG)}\II(\cA)^\times$ belongs to $im(\Theta_{\cG})$ if and only if it satisfies all of the following
three conditions.
\begin{enumerate}
 \item For all subgroups $P, P'$ of $\ol\cG$ with $[P',P']\leq P\leq P'$, one has
\begin{equation}
 \mathrm{Nr}_P^{P'}(\Xi_{U_{P'}^{ab}})=\Pi_P^{P'}(\Xi_{U_{P'}^{ab}}).
\end{equation}
 \item For all subgroups $P$ of $\ol\cG$ and all $g$ in $\ol\cG$ one has $\Xi_{gU_{P}^{ab}g^{-1}}=g\Xi_{U_{P'}^{ab}}g^{-1}$.
 \item For every $P\in\ocG$ and $P\neq (1)$, we have
 \begin{equation*}
  \mathrm{ver}_P^{P'}(\Xi_{U_{P'}^{ab}})\equiv \Xi_{U_P^{ab}} \pmod{\sT_{P,P'}} (\mbox{ resp. } \sT_{P,P',\sS} \mbox{ and } \widehat\sT_{P,P'}).
 \end{equation*}
 \item For all $P\in C(\ol\cG)$ one has $\alpha_P(\Xi_{U_{P}^{ab}})\equiv\prod_{P'\in C_P(\ol\cG)}\alpha_{P'}(\Xi_{U_{P'}^{ab}})\pmod{p\sT_P}$.
\end{enumerate}
\end{theorem}
We give a proof of this theorem referring to \cite{kakde} for many of the details which remain true in our set-up.
\begin{defn}
 Let $\Phi_\sR^{\ol\cG}$ (resp. $\Phi_{\sR,\sS}^{\ol\cG}$ and $\widehat\Phi_{\sR}^{\ol\cG}$) denote the subgroup of $\prod_{P\leq\ol\cG}\II(U_P^{ab})^\times$ 
 (resp. $\prod_{P\leq\ol\cG}\II(U_P^{ab})_\sS^\times$ and $\prod_{P\leq\ol\cG}\widehat\II(U_P^{ab})_\sS^\times$)
 consisting of tuples $(\Xi_{U_P^{ab}})$ satisfying the conditions of the above theorem:
 \begin{enumerate}
 \item[(C1)] For all subgroups $P, P'$ of $\ol\cG$ with $[P',P']\leq P\leq P'$, one has
\begin{equation*}
 \mathrm{Nr}_P^{P'}(\Xi_{U_{P'}^{ab}})=\Pi_P^{P'}(\Xi_{U_{P'}^{ab}}).
\end{equation*}
 \item[(C2)] For all subgroups $P$ of $\ol\cG$ and all $g$ in $\ol\cG$ one has $\Xi_{gU_{P}^{ab}g^{-1}}=g\Xi_{U_{P'}^{ab}}g^{-1}$.
 \item[(C3)] For every $P\in\ocG$ and $P\neq (1)$, we have
 \begin{equation*}
  \mathrm{ver}_P^{P'}(\Xi_{U_{P'}^{ab}})\equiv \Xi_{U_P^{ab}} \pmod{\sT_{P,P'}} (\mbox{ resp. } \sT_{P,P',\sS} \mbox{ and } \widehat\sT_{P,P'}).
 \end{equation*}
 \item[(C4)] For all $P\in C(\ol\cG)$ one has $\alpha_P(\Xi_{U_{P}^{ab}})\equiv\prod_{P'\in C_P(\ol\cG)}\alpha_{P'}(\Xi_{U_{P'}^{ab}})\pmod{p\sT_P} (\mbox{ resp. } p\sT_{P,\sS} \mbox{ and } p\widehat\sT_{P}).$
\end{enumerate}
\end{defn}
As in \cite{kakde,kato-hberg,burns1,ritter-weiss}, the theorem is a consequence of an explicit description of the image of the groups $\kone{\IIG}$ and
$\kone{\IIG_{\sS^\ast}}$. We proceed as in \cite{kakde} to prove this theorem. 
In fact, the theorem is a combination of the generalizations of 
\cite[Theorem 52 and 53]{kakde}. We will give the main results leading to a proof of these theorems.

For any $P\leq\ol\cG$, consider the map
\begin{equation*}
 t_P^{\ol\cG}:\sR[\Conj{\ol\cG}]^\tau\lra \sR[P^{ab}]^\tau
\end{equation*}
defined by 
\begin{equation*}
t_P^{\ol\cG}(\bar g)=\sum_{x\in C(\ol\cG,P)}\{ ({\bar x}^{-1})(\bar g)(\bar x)\mid x^{-1}gx\in P \},
\end{equation*}
where $C(\ol\cG,P)$ is the set of left coset representatives of $P$ in $\cG$.
This is a well-defined $\sR$-linear map, independent of the choice of $C(\ol\cG,P)$. For any $P\in C(\ol\cG)$, we define
\begin{equation*}
 \eta_P:\sR[P]^\tau\lra \sR[P]^\tau,
\end{equation*}
by $\sR$-linearly extending the map,
\begin{equation*}
 \eta_P(h)=\begin{cases}
            h \quad\mbox{ if } h \mbox{ is a generator of } P\\
            0 \quad\mbox{ otherwise}.
           \end{cases}
\end{equation*}
In other words, $\eta_P(x)=x-\frac{1}{p}\sum_{k=0}^{p-1}\omega_P^k(x)$.

Now, define the homomorphism $\beta_P^{\ol\cG}:\sR[\Conj{\ol\cG}]^\tau\lra \sR[P^{ab}]^\tau$ by 
\begin{equation*}
 \beta_P^{\ol\cG}=\begin{cases}
            \eta_P\circ t_P^{\ol\cG} \quad\mbox{if } P\in C(\ol\cG)\\
            t_P^{\ol\cG}             \quad\mbox{if } P\leq\ol\cG \mbox{ is not cyclic}
           \end{cases}
\end{equation*}
and $\beta_\sR^{\ol\cG}$ is defined by 
\begin{equation*}
 \beta_\sR^{\ol\cG}=(\beta_P^{\ol\cG})_{P\leq\cG}:\sR[\Conj{\ol\cG}]^\tau\lra\prod_{P\leq\ol\cG}\sR[P^{ab}]^\tau.
\end{equation*}
\begin{defn}
 Let $\Psi_\sR^{\ol\cG}$ (resp. $\Psi_{\sR,\sS}^{\ol\cG}$) be the subgroup of $\prod_{P\leq\cG}(\sR[P^{ab}]^\tau)^\times$  (resp. $\prod_{P\leq\cG}(\sR[P^{ab}]^\tau)_\sS^\times$)
 consisting of all tuples $(a_P)$ with the following properties:
 \begin{enumerate}
  \item[(A1)] Let $P\leq P'\leq\cG$ such that $[P',P']\leq P$ and the following conditions hold:
  \begin{enumerate}
  \item if $P$ is a non-trivial cyclic group then $[P',P']\neq P$;
  \item if $P$ is not cyclic, then $\mathrm{tr}_P^{P'}(a_{P'})=\pi^{P'}_P(a_P)$;
  \item if $P$ is cyclic but $P'$ is not cyclic then $\eta_P(\mathrm{tr}^{P'}_P(a_{P'}))=\pi^{P'}_P(a_P)$;
  \item if $P'$ is cyclic, then $\mathrm{tr}^{P'}_P(a_{P'})=0$.
  \end{enumerate}
  \item[(A2)] $(a_P)_{P\in C(\ol\cG)}$ is invariant under conjugation action by every $g\in\ol\cG$.
  \item[(A3)] For all $P\in C(\ol\cG)$, $a_P\in\sT_{P}$.
 \end{enumerate}
\end{defn}
Then we have the following theorem as a generalization of \cite[Theorem 58]{kakde}.
\begin{theorem}\label{additive}
 The homomorphism $\beta^{\ol\cG}_\sR$ induces an isomorphism between $\sR[\Conj{\ol\cG}]^\tau$ and $\Psi_\sR^{\ol\cG}$.
\end{theorem}
The first step of the proof is to show that the image of $\beta^{\ol\cG}_\sR$ is contained in $\Psi_\sR^{\ol\cG}$. This proof is the same as the proof of \cite[Lemma 60]{kakde}.
The next step is to consider the following map and get left inverse of $\beta^{\ocG}_\sR$.
\begin{equation*}
\begin{split}
 \delta_P:\sR[P^{ab}]^\tau\lra\sR[Conj(\ocG)]^\tau\left[\frac{1}{p}\right]\\
 x\mapsto \begin{cases}
           \frac{1}{[\ocG:P]}[x], &\mbox{ if } P \mbox{ is cyclic}\\
           0,                     & \mbox{otherwise}.
          \end{cases}
\end{split}
\end{equation*}
Combining all these maps, we get the following map:
\begin{equation*}
\begin{split}
 \delta:&\prod_{P\leq\ocG}\sR[P^{ab}]^\tau\lra\sR[Conj(\ocG)]^\tau\left[\frac{1}{p}\right],\\
\delta &=\sum_{P\leq\ocG}\delta_P.
\end{split}
\end{equation*}
\begin{lemma}
 The composite map $\delta\circ\beta^{\ocG}_\sR$ is identity on $\sR[Conj(\ocG)]^\tau$. In particular, the map $\beta^{\ocG}_\sR$ is injective.
\end{lemma}
\begin{proof}
 Let $g\in\ocG$ and $P=(\bar g)$, and consider the collection $C$ of all the conjugates of $P$ in $\ocG$. Then,
 \begin{equation*}
  \begin{split}
   \delta(\beta^{\ocG}_\sR([\bar g])) &=\sum_{P'\in C}\delta_{P'}(\beta^{\ocG}_\sR([\bar g]))\\
                                      &=\sum_{P'\in C}\dfrac{1}{[\ocG:P]}[\beta^{\ocG}_\sR([\bar g])]\\
                                      &=\dfrac{1}{[\ocG:P]}\sum_{P'\in C}[N_{\ocG}P':P'][\bar g]]\\
                                      &=\dfrac{1}{[\ocG:N_{\ocG}P]}\sum_{P'\in C}[\bar g])]\\
                                      &=[\bar g].
  \end{split}
 \end{equation*}
\end{proof}
Next, proceeding as in \cite[Lemma 63]{kakde}, we get the following lemma:
\begin{lemma}\label{inj-beta}
 The map $\delta\mid_{\Psi_\sR^{\ol\cG}}$ is injective and its image lies in $\sR[Conj(\ocG)]^\tau$.
\end{lemma}
Finally, we can show Theorem \ref{additive}.
\begin{proof} Since $\delta\circ\beta^{\ocG}_\sR$ is identity on $\sR[Conj(\ocG)]^\tau$ and $\delta\mid_{\Psi_\sR^{\ol\cG}}$ is injective, $\delta\circ\beta^{\ocG}_\sR$
is also identity on $\Psi_\sR^{\ol\cG}$. Indeed, if $(a_P)\in\Psi_\sR^{\ol\cG}$, then $\delta(\beta^{\ocG}_\sR(\delta((a_P))))=\delta((a_P))$. As the image of $\beta^{\ocG}_\sR$
is contained in $\Psi_\sR^{\ol\cG}$ and $\delta$ is injective on $\Psi_\sR^{\ol\cG}$, we have $\beta^{\ocG}_\sR(\delta((a_P)))=(a_P)$. Therefore $\beta^{\ocG}_\sR$  is surjective.

By Artin's induction theorem, a linear representation of a finite group is a $\Q$-linear combination of representations induced from cyclic subgroups \cite[Theorem 17]{serre}. 
The injectivity follows using this result.
\end{proof}
\begin{propn}\label{id-beta}
Let $K$ be the quotient field of $\cO$. Then the map 
$\mathrm{id}_{K}\otimes\beta_\sR^{\ocG}:K\otimes\sR[\mathrm{Conj}(\ocG)]^\tau\lra\prod_{P\leq\ocG}K\otimes\sR[P^{ab}]^\tau$ is injective, and its image consists of all
tuples $(a_P)$ satisfying the following:
\begin{enumerate}
  \item Let $P\leq P'\leq\cG$ such that $[P',P']\leq P$ and the following conditions hold:
  \begin{enumerate}
  \item if $P$ is a non-trivial cyclic group then $[P',P']\neq P$;
  \item if $P$ is not cyclic, then $\mathrm{tr}_P^{P'}(a_{P'})=\pi^{P'}_P(a_P)$;
  \item if $P$ is cyclic but $P'$ is not cyclic then $\eta_P(\mathrm{tr}^{P'}_P(a_{P'}))=\pi^{P'}_P(a_P)$;
  \item if $P'$ is cyclic, then $\mathrm{tr}^{P'}_P(a_{P'})=0$.
  \end{enumerate}
  \item $(a_P)_{P\in C(\ol\cG)}$ is invariant under conjugation action by every $g\in\ol\cG$.
 \end{enumerate}
Hence, if $\mathrm{id}_K\otimes\beta_\sR^{\ocG}(a)=(a_P)$ with $a_P\in\sT_{P}, \forall\, P\in C(\ocG)$, then $a\in\sR[\mathrm{Conj}(\ocG)]^\tau$ and 
$a_P\in\sR[P^{ab}]^\tau, \forall\, P\leq\ocG$.
\end{propn}
\begin{proof}
 By Lemma \ref{inj-beta} above, the injectivity is clear. The statement about the image also follows from this Lemma. Clearly, if $\mathrm{id}_K\otimes\beta_\sR^{\ocG}(a)=(a_P)$,
 with $a_P\in\sT_{P}, \forall\, P\in C(\ocG)$, then as the map $\delta$ is determined by the $a_P$'s for cyclic $P$, it follows that the inverse image $a$ lies in 
 $\sR[\mathrm{Conj}(\ocG)]^\tau$ and $a_P\in\sR[P^{ab}]^\tau$, $\forall P\leq\ocG$.
\end{proof}

\subsection{The Logarithm map over $\II[[\wt\Gamma^{p^e}]]$}
Recall that $\sR:=\II[[\wt\Gamma^{p^e}]]$. Note that $\sR$ is a local ring. 
Our aim in this section is to construct a logarithm map on $\kone{\sR[\ol\cG]^\tau}$. 
This is done by generalizing the logarithm map which was considered by Ritter and Weiss and then later by Kakde. Their
constructions were inspired by the logarithm map introduced by Oliver. 

We assume that $\cG$ is a $p$-adic Lie group of rank $1$. 
For any subgroup $P$ of $\ol\cG$, we set 
\begin{equation}\label{R-notation}
\sR_P:=\II[[U_P]]=\II[[\wt\Gamma^{p^e}]][P]^\tau=\sR[P]^\tau.
\end{equation}
Consider the natural $\sR$-linear map 
\begin{equation*}
\kappa_P:\sR_P\lra \sR[\mathrm{Conj}(P)]^\tau, \kappa_P(g^\tau)=[g^\tau], 
\end{equation*}
for all $g\in P$, where $[g^\tau]$ denotes the conjugacy class in $P$. 

For any ring $A$, let $J(A)$ denote the Jacobson radical. 
Since $\cG=H\ltimes\Gamma$, the kernel of the composite homomorphism $U_P\inj\cG\surj\Gamma$ is $H\cap U_P$.
Let $I_{H\cap U_P}$ be the augmentation ideal of $\sR_{H\cap U_P}$.
Then for the ring $\sR_P$, 
the Jacobson radical $J(\sR_P)$ is generated by $\m_\II$ and $I_{H\cap U_P}$, where $\m_\II$ denotes the maximal ideal $\langle p,X_1,\cdots,X_r\rangle$ of $\II$.

In the proof of the following propositions, the following short exact sequence which is obtained by using $\sR_P/J(\sR_P)\cong\FF_p$ play a crucial role:
\begin{equation}\label{fund-exact}
  1\lra \kone{\sR_P,J(\sR_P)}\lra \kone{\sR_P}\lra \kone{\FF_p}.
\end{equation}
It is well known that $ \kone{\FF_p}\cong\FF_p^\times$.
\begin{propn}\label{Log}
Let $P\leq\ol\cG$. Then, for $x\in J(\sR_P)$, the logarithm defined by 
\begin{equation}\label{def-log}
\Log(1+x):=\sum_{n\geq 1}(-1)^{n+1}\frac{x^n}{n}.
\end{equation}
is well-defined, and it induces a homomorphism
\begin{equation*}
\log_P:\kone{\sR_P}\lra \sR[\mathrm{Conj}(P)]^\tau[\frac{1}{p}].
\end{equation*}
Moreover, this map is natural with respect to ring homomorphisms induced by group 
homomorphisms.
\end{propn}
\begin{proof}
First, we show that the map is well-defined by showing that the power series $\Log(1+x)$ converges in $\sR[P]^\tau[\frac{1}{p}]$. 

Since $H\cap U_P$ is a finite $p$-group, with say, $p^r$ elements, we have $(g-1)^{p^r}\in p\II[H\cap U_P]$, for any $g\in H\cap U_P$. Therefore, for any 
$x\in J(\sR_P)=\langle\m_\II, I_{H\cap U_P}\rangle$, we have $x^{p^r}\in \langle p,\m_\II\rangle\sR_P$. Hence $x^{n}\in \langle p,\m_\II\rangle^m\sR_P$ for large enough
$n,m$. 
Therefore, $x^i$ converges to $0$ as $i$ tends to infinity. 
Also, $x^{n}/n\in \langle p,\m_\II\rangle^m\sR_P[\frac{1}{p}]$ for large enough $n,m$. 
This implies that $x^i/i$ converges to $0$ as $i$ tends to infinity. Hence the series $\Log(1+x)$ converges in $\sR[P]^\tau[\frac{1}{p}]$. 

We now use arguments of Oliver to construct the map $\log_P$. Indeed, the proof of Oliver for \cite[Lemma 2.7]{ol}, shows that for any 
$x,y\in J(\sR_P)$, we have
\begin{equation*}
 \Log((1+x)(1+y))\equiv\Log(1+x)+\Log(1+y)\pmod{[\sR_P[\frac{1}{p}], J(\sR_P)[\frac{1}{p}]]}.
\end{equation*}
Then by the proof of \cite[Theorem 2.8]{ol}, $\Log(1+x)$ induces a well-defined homomorphism 
\begin{equation*}
 \log_P':\kone{\sR_P,J(\sR_P)}\lra (J(\sR_P)/[\sR_P, J(\sR_P)])[\frac{1}{p}].
\end{equation*}
Since $\sR_P/J(\sR_P)\cong\FF_q$, we have the following exact sequence
\begin{equation}\label{rel-K}
 1\lra \kone{\sR_P,J(\sR_P)}\lra \kone{\sR_P}\lra \kone{\FF_q},
\end{equation}
where $\FF_q$ is a finite extension of $\FF_p$. 
As $\kone{\FF_q}\cong\FF_q^\times$ and $(J(\sR_P)/[\sR_P, J(\sR_P)])[\frac{1}{p}]$ is torsion-free, the map $\log_P'$ can be extended uniquely 
to $\kone{\sR_P}$, which we call $\log_P$.
\end{proof}
\begin{remark}\label{general-log}
 The proof of \cite[Theorem 2.8]{ol} can also be generalized to show that we have the following homomorphism:
 \begin{equation*}
 \log_P^I:\kone{\sR_P,I}\lra (I/[\sR_P, I])\left[\frac{1}{p}\right],
\end{equation*}
for any ideal $I\subset J(\sR_P)$.
\end{remark}
\begin{lemma}\label{log-sub}
 Let $U_P$ be abelian and let $I$ be any ideal of $\sR[P]^\tau$ such that $I\subset p\sR[P]^\tau$. 
Then $\log_P$ is a well-defined map from $1+I$ to $I$, which is an isomorphism.
\end{lemma}
\begin{proof}
 We first show that the map is well-defined.
 Since $I\subseteq p\sR[P]^\tau$, we have $I^p\subseteq pI$. It follows that $I^{p^r}\subseteq p^rI$ for all $r\in\N$.
 This further implies that $I^n\subseteq nI$, for all $n\in\N$. 
 This is easy to see if $p\nmid n$, as $n$ is then a unit and $I^n\subseteq I=nI$. Next, if $\mathrm{ord}_p(n)=e$, then $I^{p^e}\subseteq p^eI$. As $n/{p^e}$ is a unit, raising
 to $n/{p^e}$-th power, we
 have $I^n\subseteq nI$.
 Now, let $x\in I$, then $x^n\in I^n\subseteq nI$. Therefore $x^n/n\in I$. 
 Noting that the series $\sum_{n\geq 0}(-1)^{n+1}\dfrac{x^n}{n}$ converges, it converges in $I$.
 
Conversely, we show that for each $x\in I$, the series $\mathrm{exp}_P(x)=\sum_{n\geq 0}\dfrac{x^n}{n!}$ is convergent in $1+I$. Since $I\subset p\sR[P]^\tau\cong\II(U_P)$, the Jacobson radical 
$J(\sR[P]^\tau)$ contains $I$. Further, $I^p\subseteq pI.J(\sR[P]^\tau)$. Suppose $I^{p^k}\subseteq p^k!I.J(\sR[P]^\tau)^k$, then raising both sides to a $p$-th power, we get
\begin{equation*}
 I^{p^{k+1}}\subseteq (p^k!)^pI^p.J(\sR[P]^\tau)^{kp}.
\end{equation*}
Clearly $(p^k!)^p I^p.J(\sR[P]^\tau)^{kp}\subseteq (p^{k+1}!)I.J(\sR[P]^\tau)^{k+1} $. 
Therefore, by induction, 
$$I^{p^n}\subseteq p^n!I.J(\sR[P]^\tau)^n, \mbox{ for all } n\in\N.$$ The lemma follows by observing that $\log_P$ and $\mathrm{exp}_P$ are inverses of each other.
\end{proof}
We now show that the logarithm maps are natural with respect to the transfer maps.
For each pair of subgroups $P$ and $P'$ of $\ol\cG$ with $P\leq P'$, consider the natural \emph{restriction map} on $K$-groups
\begin{equation*}
 \maptheta{P'}{P}:\kone{\II(U_{P'})}=\kone{\sR_{P'}}\lra\kone{\sR_P}=\kone{\II(U_P)}.
\end{equation*}
Moreover, we define an $\II$-linear map
\begin{equation*}
 \Res{P'}{P}:\sR[\mathrm{Conj}(P')]^\tau\lra \sR[\Conj{P}]^\tau
\end{equation*}
given by 
\begin{equation}\label{kappa}
 \Res{P'}{P}(\kappa_{P'}(g^\tau)):=\sum_x\kappa_P((x^\tau)^{-1}(gx)^\tau)
\end{equation}
where $x$ runs over all elements in a given set of left coset representatives of $P$ in $P'$ with $xgx^{-1}\in P.$
\begin{lemma}\label{log-theta-res}
 For each subgroup $P$ of $\ol\cG$, we have the following commutative diagram.
\begin{equation*}
 \xymatrix{
\kone{\IIG}\ar[r]^{\!\!\!\!\!\log_{\ol\cG}}\ar[d]^{\maptheta{\ol\cG}{P}} &\sR[\Conj{\ol\cG}]^\tau[\frac{1}{p}]\ar[d]_{\Res{\ol\cG}{P}}\\
\kone{\II(U_P)}\ar[r]^{\!\!\!\!\!\log_{P}} &\sR[\Conj{P}]^\tau[\frac{1}{p}].
}
\end{equation*}
\end{lemma}
\begin{proof}
The result follows by generalizing the proof of \cite[Theorem 1.4]{ol-tay}. We briefly indicate the appropriate generalizations that are needed.
Let $k\in\N$ and $\xi\in\IIG$. Then, for the class $[1+p^k\xi]\in\kone{\IIG}$ we can see as in \emph{loc. cit}, that 
\begin{equation*}
 \maptheta{\ol\cG}{P}([1+p^k\xi])\equiv[\Res{\ol\cG}{P}(1+p^k\xi)]\pmod{\kone{\sR[P]^\tau,\m_\II^{2k}\sR[P]^\tau}}.
\end{equation*}
Moreover, we have
\begin{equation*}
 \log_{\ol\cG}([1+p^k\xi])\equiv p^k\xi,\quad \log_P([1+p^k\Res{\ol\cG}{P}(\xi)]) \equiv p^k\Res{\ol\cG}{P}(\xi)\pmod{\m_\II^{2k-1}\sR[P]^\tau[\frac{1}{p}]}
\end{equation*}
Then, as $(1+p\xi)^{p^k}\in 1+\m_\II^{k+1}\IIG$,
\begin{align*}
 \Res{\ol\cG}{P}\circ\log_{\ol\cG}([1+p\xi])=&\dfrac{1}{p^k}\Res{\ol\cG}{P}\circ\log_{\ol\cG}([1+p\xi]^{p^k})\\
 \equiv&\dfrac{1}{p^k}\log_P\circ\Res{\ol\cG}{P}([1+p\xi]^{p^k})=
 \log_P\circ\Res{\ol\cG}{P}([1+p\xi])\pmod{\m_\II^{k+1}},
\end{align*}
for all $k>0$. 
Hence the diagram in the lemma
commutes for the subgroup $\kone{\sR[\ol\cG]^\tau,J(\sR[\ol\cG]^\tau)}$ of $\kone{\sR[\ol\cG]^\tau}$. Now note that the index of the subgroup 
$\kone{\sR[\ol\cG]^\tau,J(\sR[\ol\cG]^\tau)}$ of $\kone{\sR[\ol\cG]^\tau}$ is 
finite, and the group $\sR[\Conj{P}]^\tau[\frac{1}{p}]$ is torsion-free. Therefore the commutativity extends to the diagram in the assertion of the lemma.
\end{proof}
\subsection{The Integral logarithm over $\IIG$}
We now make the following assumption
\begin{description}
 \item[(Ur)] $\II=\cO[[X_1,\cdots,X_r]]$ with $\cO$ unramified over $\Z_p$.
\end{description}
Then on $\sR=\II(\wt\Gamma^{p^e})$,
consider the map $\varphi:\sR\lra \sR$ such that its restriction to $\cO$ is the Frobenius and  maps $X_j^n$ to $X_j^{pn}$ for all $j=1,\cdots,r$ and the $p$-th 
power map on $\wt\Gamma^{p^e}$. Further, we extend $\varphi$ to a map 
\begin{equation*}
\varphi_{conj}:\sR[\Conj{\ol\cG}]^\tau\lra \sR[\Conj{\ol\cG}]^\tau,
\end{equation*}
by 
\begin{equation*}
\kappa(g^\tau)\mapsto\kappa((g^\tau)^p). 
\end{equation*}
\begin{defn}
The map $\mathfrak{L}_\cG:\kone{\sR[\ol\cG]^\tau}\lra \sR[\Conj{\ol\cG}]^\tau[\frac{1}{p}]$ defined by
\begin{equation*}
\mathfrak{L}_\cG:=\log_\cG-p^{-1}\varphi_{conj}\circ\log_\cG 
\end{equation*}
is called the $p$-adic logarithm map over $\IIG$. 
\end{defn}
\begin{propn}\label{int-log}For the map $\mathfrak{L}_\cG:\kone{\sR[\ol\cG]^\tau}\lra \sR[\Conj{\ol\cG}]^\tau[\frac{1}{p}]$,
we have $Im(\mathfrak{L}_\cG)\subseteq \sR[\Conj{\ol\cG}]^\tau$.
\end{propn}
\begin{proof}As before, by the exact sequence in \eqref{fund-exact}, the index of $\kone{\sR[\ol\cG]^\tau,J(\sR[\ol\cG]^\tau)}$ in $\kone{\sR[\ol\cG]^\tau}$ is finite and prime to 
$p$. Therefore, it is enough to prove that $\mathfrak{L}_\cG(\kone{\sR[\ol\cG]^\tau,J(\sR[\ol\cG]^\tau)})\subseteq \sR[\Conj{\ol\cG}]^\tau$. 

Let $y\in\kone{\sR[\ol\cG]^\tau,J(\sR[\ol\cG]^\tau)}$. Then, by \cite{vaser}, 
there exists $x$ such that $y=[1-x]$ for some $x\in J(\sR[\ol\cG]^\tau)$. 
Then, we have,
\begin{equation*}
 \mathfrak{L}_\cG(y)=-\sum_{i\geq 1}\frac{x^i}{i}+\sum_{i\geq 1}\frac{\varphi_{conj}(x^i)}{pi}
         =-\sum_{i\geq 1,p\nmid i}\frac{x^i}{i}-\sum_{j\geq 1}\frac{x^{pj}-\varphi_{conj}(x^j)}{pj}.
\end{equation*}
Therefore, it is enough to prove that the sum $\sum_{j\geq 1}\frac{x^{pj}-\varphi_{conj}(x^j)}{pj}\in \sR[\Conj{\ol\cG}]^\tau$. This follows from  
the same argument as in the proof of \cite[Theorem 6.2]{ol}.
\end{proof}
We now extend the theorem of Oliver \cite[Theorem 6.4, 6.6]{ol} to $\sR_G$, for finite groups $G$ of prime order. For this, we recall the following exact sequence from 
\cite[Lemma 6.3(ii)]{ol}:
\begin{equation}\label{k_h_0}
 0\inj\mathbb{F}_p\lra\sR/\m_\sR\stackrel{1-\varphi}{\lra}\sR/\m_\sR\stackrel{\mathrm{Tr}}{\lra}\mathbb{F}_{p}\lra 0
\end{equation}
where $\mathbb{F}_p$ is the finite field of order $p$, $\varphi$ is the Frobenius and $\mathrm{Tr}$ denotes the trace map. Here we note that $\sR$ is a local field with maximal 
ideal $\m_\sR$.
\begin{propn}\label{oli-6.4}
 Let $G$ be a finite $p$-group and $z$ be an element of order $p$ in the center $Z(G)$. Then, we have the following exact sequence:
 \begin{equation}\label{exact-log}
1\lra \langle z\rangle \lra \kone{\sR_G,(1-z)\sR_G}\stackrel{\log_G}{\lra}\mathrm{H}_0(G,(1-z)\sR_G)\stackrel{\omega}{\lra}\FF_p\lra 0
 \end{equation}
\end{propn}
\begin{proof}
 The proof is the same as in the proof of \cite[Theorem 6.4]{ol}. We set $I=(1-z)\sR_G$ and $J$ the Jacobson radical of $\sR_G$. As $(1-z)^p\in pI$, the $p$-adic 
 logarithm induces a homomorphism $\log_G^I$ and an isomorphism $\log_G^{IJ}$. These maps fit in the following commutative diagram:
 \begin{equation*}
  \xymatrix{
           &\kone{\sR_G,(1-z)J}\ar[r]\ar[d]_{\log_G^{IJ}}^{\cong} &\kone{\sR_G,I}\ar[r]\ar[d]_{\log_G^I} &\kone{\frac{\sR_G}{(1-z)J},\frac{I}{(1-z)J}}\ar[r]\ar[d]_{\log_0} & 1\\
  0\ar[r]  & \mathrm{H}_0(G,(1-z)J)\ar[r]                 &\mathrm H_0(G,I)\ar[r]             &\mathrm H_0(G,\frac{I}{(1-z)J})\ar[r]                    &0
  }
 \end{equation*}
By a result of Bass (see \cite[Theorem 1.15]{ol}), we have the following identification
\begin{equation*}
 \kone{\frac{\sR_G}{(1-z)J},\frac{I}{(1-z)J}} \mathrel{\mathop{\longrightarrow}^{\alpha}_{\cong}} \sR_G/J\cong\sR/\m_\sR, 
\end{equation*}
where the map $\alpha$ is given by $\alpha(1+(1-z)\xi)=\xi$ for $\xi\in\sR_G/J$.
Further, $\mathrm H_0(G,I/(1-z)J)\cong\sR/\m_\sR$. Together with the exact sequence in \eqref{k_h_0}, 
we can see that the map $\log_0$ fits in the following exact sequence:
\begin{equation*}
 0\lra \FF_p\lra\sR_G/J\stackrel{\cong}{\lra}\kone{\frac{\sR_G}{(1-z)J},\frac{I}{(1-z)J}}\stackrel{\log_0}{\lra} \mathrm H_0\left(G,\frac{I}{(1-z)J}\right)\stackrel{\cong}{\lra}
 \sR/\m_\sR \lra\FF_p\lra 0,
\end{equation*}
as a result, we have the following commutative diagram:
\begin{equation*}
 \xymatrix{
                                     &0\ar[d]                                                         &  \\
                                     &\FF_p\ar[d]                                                     &  \\
 \kone{\sR_G,I}\ar[r]^{\alpha'}\ar[d]_{\log_G^I} &\sR_G/J\cong\sR/\m_\sR\ar[r]\ar[d]_{\log_0}^{\cong} & 1\\
 \mathrm H_0(G,I)\ar[r]^{\alpha''}\ar[rd]_\omega             &\sR/\m_\sR\ar[r]\ar[d]^{\mathrm{Tr}}            &0 \\
                                     &\FF_p\ar[d]                                                     & \\
                                     &0                                                               &
 }
\end{equation*}
where $\alpha'(1+(1-z)\sum r_ig_i)=\sum\ol{r_i}$ and $\alpha''((1-z)\sum r_ig_i)=\sum\ol{r_i}$. Indeed, the right column is exact, and as
\begin{equation*}
 \alpha''(\log_G^I(1+(1-z)rg))=\alpha''((1-z)(rg-r^pg^p))=r-\varphi(r)\in\sR/\m_\sR,
\end{equation*}
the square is also commutative. Therefore the maps $\log_G^I$ and $1-\varphi$ have isomorphic kernel and cokernel.
Noting that $\omega=\mathrm{Tr}\circ\alpha''$ and $\alpha'$ maps $\langle z\rangle$ isomorphically onto $\FF_p=\mathrm{ker}(1-\varphi)$, the exactness of the 
sequence in the Proposition follows.

\end{proof}
We now prove the key result regarding the integral logarithm. This is a generalization of the following exact sequence, (\cite[Def 70]{kakde}).
Since $\II=\cO[[X_1,\cdots,X_r]]$, the ring $\sR=\II[[\wt\Gamma^{p^e}]]$ isomorphic to an Iwasawa algebra with $r+1$ variables.
Let $\WW\cong (1+p\Z_p)^{r+1}$, then $\sR\cong\cO[[\WW]]$. 
We assume that the extension $\cO/\Z_p$ is unramified. Recall that the quotient field of $\cO$ is denoted by $K$. Then the following sequence is an exact sequence of groups:
\begin{equation*}
1\lra\mu_K\times\WW\lra\kone{\sR}\stackrel{\mathfrak L}{\lra}\sR\stackrel{\omega}{\lra}\WW\lra 1.
\end{equation*}
\begin{propn}\label{log-exact}
 Let $G$ be a finite $p$ group. Let $\WW\cong(1+p\Z_p)^{r+1}$ and $\II=\mathcal\cO[[X_1,\cdots,X_r]]$, such that the extension $\cO/\Z_p$ is unramified. Define the map
 \begin{equation*}
 \wt\omega:\sR[G]\lra \WW\times G^{ab}
 \end{equation*}
by $\wt\omega(\sum_{i}a_ig_i)=\prod_i(\omega(a_i),(g_i)^{\mathrm{Tr}(a_i\mod\m_\II)})$. Then the sequence
 \begin{equation*}
 1\lra\kone{\sR_G}/\left(\WW\times\kone{\sR_G}_{tors}\right)\stackrel{\mathfrak{L}_G}{\lra}\sR_G\stackrel{\wt\omega}{\lra}\WW\times G^{ab}\lra 1
 \end{equation*}
 is exact.
\end{propn}
\begin{proof}
 The proof is a generalization of the proof of Oliver. We also use induction on the order of $G$ to prove the result.
 Let $G=(1)$. Then $\sR=\II[[\wt\Gamma^{p^e}]]$. 
By \cite[Def 70]{kakde}, we have the exact sequence, 
\begin{equation*}
 1\lra\mu_K\times\WW\lra\kone{\sR}\stackrel{\mathfrak L}{\lra}\sR\stackrel{\omega}{\lra}\WW\lra 1.
\end{equation*}

Next, let $G$ be a non-trivial $p$-group. We then show that $\wt\omega\circ\mathfrak{L}_G=1$. It is enough to prove this when $G$ is abelian. 
Let $I$ be the augmentation ideal of $\sR[G]$. Consider $u=1+\sum r_i(1-a_i)g_i\in 1+I,$ where $r_i\in\sR$.
Then 
\begin{equation*}
\begin{split}
u^p\equiv & 1+p\sum r_i(1-a_i)g_i+\sum r_i^p(1-a_i)^pg_i^p\mod{pI^2}\\
   \equiv & 1+p\sum r_i(1-a_i)g_i+\sum r_i\{(1-a_i)^p-p(1-a_i)\}g_i^p\mod{pI^2}, \mbox{ by \cite[Lemma 6.3(i)]{ol}}\\
   \equiv & \varphi_{conj}(u)+p\sum r_i(1-a_i)(g_i-g_i^p)\mod{pI^2}\\
   \equiv & \varphi_{conj}(u)\mod{pI^2},
\end{split}
\end{equation*}
 i.e., $u^p/\varphi_{conj}(u)\in 1+pI^2$. Therefore $\mathfrak{L}_G(u)=\log_G(u)-p^{-1}\varphi_{conj}(\log_G(u))=\frac{1}{p}\log_G(u^p/\varphi_{conj}(u))\in I^2$.

On the other hand, for any $r\in\sR$ and $a,b,g\in G$, we have
\begin{equation*}
\begin{split}
 \wt\omega(r(1-a)(1-b)g)&=(\omega(r),g^{\mathrm{Tr}(r)})(\omega(-r),(ag)^{\mathrm{Tr}(-r)})(\omega(-r),(bg)^{\mathrm{Tr}(-r)})(\omega(r),(abg)^{\mathrm{Tr}(r)})\\
                        &=(\omega(r)\omega(-r)\omega(-r)\omega(r),1)\\
                        &=(1,1)\in \WW\times G.
 \end{split}
\end{equation*}
Therefore, $\mathfrak{L}_G(1+I)\subseteq I^2\subseteq\mathrm{ker}(\wt\omega)$, and hence
\begin{equation}\label{int_subset_ker_omega}
 \mathfrak{L}_G(\kone{\sR_G})=\mathfrak{L}_G(\sR^\times\times(1+I))=\langle\mathfrak{L}(\sR^\times),\mathfrak{L}(1+I)\rangle\subseteq\mathrm{ker}(\wt\omega).
\end{equation}
It follows that $\wt\omega\circ\mathfrak{L}_G=1$.
 
Assume that the theorem is true for all groups whose order is less than the order of $G$.
Now, let $z$ be an element of order $p$ in the center $Z(G)$, such that $z$ is a commutator if $G$ is nonabelian. The existence of such a commutator is shown in 
\cite[Lemma 6.5]{ol}. Let $\widehat{G}:=G/\langle z\rangle$. Let $\alpha:G\lra\widehat G$ denote the natural projection map. Then we have the following commutative diagram:
\begin{equation*}
 \xymatrix{
            &1\ar[d]                                                                     &1\ar[d]            &1\ar[d]                                        &          \\
1\ar[r]     &\kone{\sR[G],(1-z)\sR[G]}/\mathrm{tors}\ar[r]^{\quad\quad\mathfrak{L}_0}\ar[d]        
                                                                                         &\ol{\mathrm H}_0(G;(1-z)\sR[G])\ar[r]^{\quad\quad\omega_0}\ar[d] 
                                                                                                             &\mathrm{ker}(\alpha^{\mathrm{ab}})\ar[r]\ar[d] &1 \\
1\ar[r]     &\kone{\sR[G]}/\mathrm{tors}\ar[r]^{\mathfrak{L}_G}\ar[d]                    & {\mathrm H}_0(G;(1-z)\sR[G])\ar[r]^{\quad\quad\omega_G}\ar[d] 
                                                                                                             &\mathbb{W}\times G^{ab}\ar[r]\ar[d]^{\alpha} &1 \\
1\ar[r]     &\kone{\sR[\widehat{G}]}/\mathrm{tors}\ar[r]^{\mathfrak{L}_{\widehat G}}\ar[d]  
                                                                                         & {\mathrm H}_0(\widehat G;(1-z)\sR[\widehat G])\ar[r]^{\quad\quad\omega_{\widehat G}}\ar[d] 
                                                                                                             &\mathbb{W}\times\widehat{G}^{ab}\ar[r]\ar[d]   &1 \\
            &1                                                                           &1                  &1                                              & 
}
\end{equation*}
By \cite[Theorem 1.14 (iii)]{ol}, the columns are all exact. By the induction hypothesis, the bottom row is exact. In the top row, 
the integral logarithm map $\mathfrak{L}_0$ is injective, by Proposition \ref{oli-6.4}, and the map $\omega_0$ is clearly onto. 
Moreover, $\mathrm{Im}(\mathfrak{L}_0)\subset\mathrm{ker}(\omega_0)$, by \eqref{int_subset_ker_omega}. Lastly, by Proposition \ref{oli-6.4} again, we have,
\begin{equation*}
\begin{split}
 \mid\mathrm{ker}(\alpha^{\mathrm{ab}})\mid&=\begin{cases}
                                             1 & \mbox{ if } z \mbox{ is a commutator }\\
                                             p & \mbox{ otherwise}
                                            \end{cases}\\
                                           &=\mid\mathrm{coker}(\mathfrak{L}_0) \mid.
\end{split}                                           
\end{equation*}
Again, as $\omega_G\circ\mathfrak{L}_G=1$, by the equality \eqref{int_subset_ker_omega}, it follows that the middle row is short exact.
\end{proof}
\begin{defn}\label{sk-def2}
Let $\mathbb{K}=K[[X_1,\cdots,X_r,Y]]$, where $K$  is the quotient field of $\cO,$ and $Y$ is the variable corresponding to $\wt\Gamma^{p^e}$. Recall that 
$\WW=(1+p\Z_p)^{r+1}$. 
For any finite group $G$, we define the following groups (see \cite[Page 173]{ol}):
\begin{equation*}
 \begin{split}
 SK_1(\sR[G])&:=\mathrm{ker}\left[\kone{\sR[G]}\lra\kone{\mathbb{K}[G]}\right],\\
  \konep{\sR[G]}&:=\kone{\sR[G]}/S\kone{\sR[G]},\\
  \mathrm{Wh}(\sR[G])&:=\kone{\sR[G]}/\left(\mu_K\times\WW\times G^{ab}\right), 
 \end{split}
\end{equation*}
where $\mu_K$ is the set of roots of unity in $K$. 
This is a generalization of Definition \ref{sk-def}
\end{defn}

The following proposition is a generalization of \cite[Theorem 7.1]{ol}.
\begin{propn}
 Let $G$ be a finite $p$-group and $z\in Z(G)$ such that the order of $z$ is the prime $p$. Let 
 \begin{equation*}
  \Omega=\{g\in G: [g,h]=z, \mbox{ for some } h\in G\}.
 \end{equation*}
On this set, consider the following relation $\sim$:
\begin{equation*}
 g\sim h \mbox{ if } \begin{cases}
                      g \mbox{ is conjugate to } h, \mbox{ or }\\
                      [g,h]=z^i, \mbox{ for any } i \mbox{ prime to } p.
                     \end{cases}
\end{equation*}
Then 
\begin{equation*}
 \mathrm{ker}\left[\mathrm{tors}(\Wh(\sR[G]))\lra \mathrm{tors}(\Wh(\sR[G/\langle z\rangle]))\right]\cong (\Z/p)^N, \mbox{where } N=\begin{cases}
          0 &\!\!\! \mbox{ if } \Omega=\emptyset  \\
          |\Omega/\sim|-1 &\!\!\! \mbox{ if } \Omega\neq0.\end{cases}
\end{equation*}
\end{propn}
\begin{proof}
The proof of this Proposition also follows the same argument as in the proof of Oliver. The first step is to recall the exact sequence in \eqref{exact-log}, which comes from the 
 homomorphism 
 \begin{equation*}
  \log_G:\kone{\sR[G],(1-z)\sR[G]}\lra \mathrm{H}_0(G;(1-z)\sR[G]),
 \end{equation*}
where $\mathrm{ker}(\log_G)=\langle z\rangle$ and $\mathrm{im}(\log)=\{(1-z)\sum r_ig_i:r_i\in\sR, g_i\in G,\sum r_i\in\mathrm{ker}(\tau)\}$, for the composite map $\tau:\sR\lra\sR/{\m_\sR}\stackrel{Tr}{\lra}\FF_p$.
This map fits into the following commutative diagram:
\begin{equation*}
 \xymatrix{
                                              &0\ar[d]\\
                                              &\mathrm{H}_0(G,(1-z)\sR[\Omega])\ar[d]\\
 \kone{\sR[G],(1-z)\sR[G]}\ar[r]^{\log_G}\ar[d] &\mathrm{H}_0(G,(1-z)\sR[G])\ar[d]\\
 \kone{\sR[G]}\ar[r]\ar[d]^\eta                    &\mathrm{H}_0(G,\sR[G])\\
 \Wh(\sR[G])\ar[ru]_{\log_{\sR[G]}}.
 }
\end{equation*}
In the above diagram, we have used the following equality
\begin{equation*}
\begin{split}
&\mathrm{ker}\left[\mathrm{H}_0(G,(1-z)\sR[G])\lra\mathrm{H}_0(G,\sR[G])\right]\\
&=\langle r(1-z)g\in\mathrm{H}_0(G,(1-z)\sR[G]): g \mbox{ is conjugate to } gz, r\in\sR[G]\rangle\\
                                                                              &=\mathrm{H}_0(G,(1-z)\sR[\Omega]).
\end{split}
\end{equation*}
Now, consider the surjection $\kone{\sR[G]}\stackrel{\eta}{\lra}\Wh(\sR[G])$. Let $x\in(\Wh(\sR[G]))_{\mathrm{tors}}$, with $\eta(y)=x$ for some $y\in\kone{\sR[G]}$. 
Then, as $x^n=1$, for some $n$, it follows that $y^n\in\mathrm{ker}(\eta)$, which is finite. Therefore $y^{nr}=1$, for some $r$, and $y\in\kone{\sR[G]}_{\mathrm{tors}}$.
Hence $\eta$ induces a surjection $\kone{\sR[G]}_{\mathrm{tors}}\stackrel{\eta}{\lra}\Wh(\sR[G])_{\mathrm{tors}}$. On the other hand, by a straight forward generalization
of \cite[Theorem 2.9]{ol}, we have
$\mathrm{ker}(\log_{\sR[G]}\circ\eta)=\kone{\sR[G]}_{\mathrm{tors}}$. Combining this with the surjection 
$\mathrm{ker}(\log_{\sR[G]}\circ\eta)\stackrel{\eta}{\lra}\mathrm{ker}(\log_{\sR[G]})$, we have $\mathrm{ker}(\log_{\sR[G]})=\Wh(\sR[G])_\mathrm{tors}$.


Further, set $I:=(1-z)\sR[G]$ and consider the map $L:1+I\stackrel{\log_I}{\lra} I\stackrel{\mathrm{proj}}{\lra}I/[\sR[G],I]$. 
Then by a straightforward generalization of \cite[Theorem 2.9]{ol}, we have a surjection from $\mathrm{ker}(L)$ to $\mathrm{ker}(\log_I)$.
Therefore, 
for any $u\in1+(1-z)\sR[G]$, if $\bar u\in\Wh(\sR[G])$ denotes the image of $u$, then 
\begin{equation*}
\begin{split}
 \bar u\in\Wh(\sR[G])_{tors}&\iff \bar u\in \mathrm{ker}(\log_{\sR[G]})\\ 
                            &\iff \log_G(u)\in\mathrm{ker}\left[\mathrm{H}_0(G,(1-z)\sR[G])\lra\mathrm{H}_0(G,\sR[G])\right]=\mathrm{H}_0(G,(1-z)\sR[\Omega])
 \end{split}
\end{equation*}

The next step is to consider the sets
\begin{equation*}
\begin{split}
 D&=\{\xi\in\sR(\Omega):(1-z)\xi\in\log_G(1+(1-z)\sR[G])\}\\
 C&=\{\xi\in\sR(\Omega):(1-z)\xi=\log_G(u), \mbox{ for some }u\in\mathrm{ker}(1+(1-z)\sR[G]\lra\mathrm{Wh}(\sR[G]))\}
 \end{split}
\end{equation*}
and show that the required kernel is equal to $D/C$. This is done exactly as in the proof of \cite[Theorem 7.1]{ol}. 
\end{proof}
As a consequence of the proposition, we get the following corollary.
\begin{cor}\label{sk-trivial}
 Let $G$ be a finite $p$-group containing an abelian subgroup $H\unlhd G$ such that $G/H$ is cyclic. Then $S\kone{\sR[G]}=1$.
\end{cor}
\begin{proof} 
 The proof proceeds by induction as in \cite[Cor 7.2]{ol}. If $G=1$, then as $\sR$ is a local ring, $S\kone{\sR[G]}=S\kone{\sR}=1$ (\cite[Cor V.9.2]{bass}).
Then assume that $H\neq1$, and choose $z\in H\cap Z(G)$ of order $p$. We also assume inductively that $\Wh(\sR[G/(z)])$ is torsion free. As above, we consider the 
set $\Omega$ and the relation $\sim$. By the previous result, it is enough to show that this relation is transitive on $\Omega$. We include the short proof for 
convenience. Let $\Omega\neq\emptyset$. We take any $g\in\Omega$,
and any $x\in G-H$, which is a generator of $G/H$. Let $h\in\Omega$ such that $[g,h]=z$. As $G/H$ is cyclic, there exists $i$ such that either $gh^{i}$ or $g^{i}h$ lies in $H$.
By symmetry, we may assume that $gh^i=a\in H$. Let $h=bx^j$ for some $b\in H$, then
\begin{equation*}
 z=[g,h]=[gh^i,h]=[a,bx^i]=[a,x^j]=[ax,x^j]=[ax,x^j(ax)^{-j}]=[x,x^j(ax)^{-j}], 
\end{equation*}
where the last equality happens as $x^j(ax)^{-j}\in H$. Therefore, in $\Omega$, we have,
\begin{equation*}
 g\sim h\sim gh^i=a\sim x^j\sim ax\sim x^j(ax)^{-j}\sim x.
\end{equation*}
Hence, the relation is transitive, and the result follows.
\end{proof}
\begin{theorem}\label{k-tors}
 Let $G$ be any finite $p$-group. Then $(\kone{\sR[G]})_\mathrm{tors}\cong\mu_K\times G^{ab}\times S\kone{\sR[G]}$.
\end{theorem}
\begin{proof}
If $G$ is abelian then the previous Proposition implies the result. Now let $G$ be any $p$-group, then we show that the projection map 
\begin{equation*}
 pr^*:\konep{\sR[G]}_{\mathrm{tors}}\lra\konep{\sR[G^{ab}]}_{\mathrm{tors}}
\end{equation*}
is injective. Now, we fix the group $G$ and assume inductively that the theorem holds for all of its proper subgroups and quotients. If $G$ is cyclic, dihedral, quaternionic
or semi-dihedral, then the Proposition holds by the previous corollary. 

Since the characteristic of $\KK=K[[X_1,\cdots,X_r,Y]]$ is zero, by Maschke's theorem the ring $\KK[G]$ is semisimple. Therefore,
by Wedderburn's Theorem we have the decomposition
$\KK[G]\cong\prod_{i=1}^s A_i$, for some simple $\KK[G]$-modules $A_i$ and some $s\in\N$. Since $\KK[G]$ contains the field $K$ of characteristic 0, we can show as in  
\cite[Section 2]{roquette}, that each of the division algebras that occur in the above decomposition is isomorphic to that of a primitive, faithful representation of some 
subquotient
of $G$. In other words, the endomorphism rings of the simple modules $A_i$ are isomorphic to that of simple modules defined over $\KK[T]$ for subgroups or subquotients $T$ of 
$G$. Therefore, the restriction maps and the quotient maps define the following monomorphism:
\begin{equation}\label{k-roquette}
 \sum\mathrm{Res}_T^G\oplus\sum\mathrm{Proj}_{G/N}^G:\kone{\KK[G]}\lra
 \bigoplus_{\stackrel{T\subset G,}{[G:T]=p}}\kone{\KK[T]}
 \oplus\bigoplus_{\stackrel{N\unlhd G,}{|N|=p}}\kone{\KK[G/N]}.
\end{equation}
It follows that the corresponding homomorphism for $\konep{\II[G]}$ is also injective.
Next, for any subgroup $H$ of $ G$ of index $p$, we have the following commutative diagram, where the maps $t_1, t_2$ are the transfer maps and the maps 
$\mathrm{Proj_1}, \mathrm{Proj_2}$ and $\mathrm{Proj_3}$ are induced by the projection maps:
\begin{equation*}
 \xymatrix{
 \konep{\sR[G]}_\mathrm{tors}\ar[r]^{\mathrm{Proj}_1}\ar[d]_{t_1} & \konep{\sR[G/[H,H]]}_\mathrm{tors}\ar@{^{(}->}[r]^{\mathrm{Proj}_3}\ar[d]_{t_2} 
                                                                                                                                         &\konep{\sR[G^{ab}]}_\mathrm{tors} \\
 \konep{\sR[H]}_\mathrm{tors}\ar@{^{(}->}[r]^{\mathrm{Proj}_2}             & \konep{\sR[H^{ab}]}_\mathrm{tors}.
 }
\end{equation*}
Here the map $\mathrm{Proj_2}$ is injective by the induction assumption. Regarding the map $\mathrm{Proj_3}$, it is also injective by Corollary \ref{sk-trivial} above,
since $G/[H,H]$ contains an abelian subgroup of index $p$. Therefore, for any $u\in\mathrm{ker}(\mathrm{Proj}_3\circ\mathrm{Proj}_1)$, we have $t_1(u)=1\in\konep{\sR[H^{ab}]}$.
Together with the fact that $\mathrm{Proj}_{G/N}^G(u)=1$, for all $N\unlhd G$ of order $p$, by the induction hypothesis, we have $u=1$ by the injective map \eqref{k-roquette}.
Hence the map $pr^\ast$ is injective.
\end{proof}
As a consequence of this along with Proposition \ref{log-exact}, we get the following result.
\begin{cor}
 Let $G$ be a finite $p$-group. Then we have the following exact sequence of groups:
 \begin{equation}\label{k-exact}
  1\lra\mu_K\times\WW\times G^{ab}\lra\konep{\sR[G]}\stackrel{\mathfrak{L}_G}{\lra}\sR[G]\lra \WW\times G^{ab}\lra 1.
 \end{equation}
\end{cor}
\subsection{The Logarithm map over $\wh{\II(Z)}_{(p)}$}
Recall that $Z:=Z(\cG)$, the center of $\cG$. 
Let $\wh{\sR}=\wh{\II(Z)}_{(p)}$ and $J({\wh\sR})$ denote its Jacobson radical. Since $\cG$ is pro-$p$, $\wh\sR[\ocG]^\tau$ is a local ring and 
$J({\wh\sR}[\ocG]^\tau)$ is its maximal ideal. We again consider the power series:
\begin{equation*}
 \Log{(1+x)}=\sum_{n=1}^\infty(-1)^{n+1}\frac{x^n}{n}.
\end{equation*}
\begin{lemma}Let $G=\ocG$.
 The ideal $J({\wh\sR}[G]^\tau)/p\wh\sR[G]^\tau$ is a nilpotent ideal of $\wh\sR[G]^\tau/p\wh\sR[G]^\tau$.
\end{lemma}
\begin{proof}
Let $\kappa=\II/p$. Consider the exact sequences:
\begin{eqnarray*}
 0\lra & J({\wh\sR}[G]^\tau)/p\wh\sR[G]^\tau\lra Q(\kappa[[\cG]])\lra  Q(\kappa[[\Gamma]])\lra0
\end{eqnarray*}
Let $N:=\ker(\kappa[[\cG]])\lra\kappa[[\Gamma]]$, and 
 $I_\cH$ be the augmentation ideal of $\kappa[\cH]$. Then $N=I_\cH\kappa[[\cG]]$. By Lemma \ref{center-ore}, any 
$x\in Q(\kappa[[\cG]])$ equals $a/t$, where $a\in\kappa[[\cG]]$ and $t\in\kappa[[Z]]$. Here 
$x\in J({\wh\sR}[G]^\tau)/p\wh\sR[G]^\tau$ if and only if $a\in N$.
As $\cH$ is a finite $p$-group, say of order $p^r$,  $(\sigma-1)^{p^r}\in p\sR[G]^\tau$, for any $\sigma\in \cH$. Therefore, there exists 
$n$ sufficiently large such that $I_\cH^n=0$, and $N^n=0$. Thus, $J({\wh\sR}[G]^\tau)/p\wh\sR[G]^\tau$ is nilpotent.
\end{proof}
Using this lemma the proof of Oliver in \cite[Lemma 27]{ol} as generalized by Kakde to prove \cite[Lemma 66]{kakde}, can further be generalized easily to show the following Lemma.
\begin{lemma}\label{log-of-hat}
 Let $I\subset J(\wh\sR[\ocG]^\tau)$ be any ideal of $\wh\sR[\ocG]^\tau$. Then
\begin{enumerate}
 \item For any $x,y\in I$, the series $\Log(1+x)$ converges to an element in $\wh\sR[\ocG]^\tau[\frac{1}{p}]$, and
 \begin{equation}
  \Log((1+x)(1+y))\equiv\Log(1+x)+\Log(1+y)\pmod{[\wh\sR[\ocG]^\tau[\frac{1}{p}], I[\frac{1}{p}]]}
 \end{equation}
 \item If $I\subset\xi\wh\sR[\ocG]^\tau$, for some central element $\xi$ such that $\xi^p\in p\xi\wh\sR[\ocG]^\tau$, then for any $x,y\in I$,
 $\Log(1+x)$ and $\Log(1+y)$ converge in $I$, and 
 \begin{equation}
  \Log((1+x)(1+y))\equiv\Log(1+x)+\Log(1+y)\pmod{[\wh\sR[\ocG]^\tau, I]}.
 \end{equation}
 Moreover, if $I^p\subset pI J(\wh\sR[\ocG]^\tau)$, then 
 \begin{enumerate}
  \item for all $x\in I$ the series $\mathrm{Exp}(x)=\sum_{n=0}^\infty\frac{x^n}{n!}$ converges to an element in $1+I$. 
  \item the maps $\Log$ and $\mathrm{Exp}$ are bijections and inverse to each other between $1+I$ and $I$.
 \end{enumerate} 
\end{enumerate}
\end{lemma}
\begin{propn}Let $I$ be any ideal contained in the maximal ideal $J(\wh\sR[\ocG]^\tau)$, then the logarithm map 
\begin{equation*}
\Log(1+x)=\sum_{n=1}^\infty(-1)^{n+1}\frac{x^n}{n} 
\end{equation*}
defined on $I$, induces a unique homomorphism
\begin{equation*}
 \log_I:\kone{\wh\sR[\ocG]^\tau,I}\lra \left(\frac{I}{[\wh\sR[\ocG]^\tau, I]}\right)\left[\frac{1}{p}\right].
\end{equation*}
If, in addition, $I\subset\xi\wh\sR[\ocG]^\tau$, for some central element $\xi$ such that $\xi^p\in p\xi\wh\sR[\ocG]^\tau$, then the logarithm map $\Log$ induces
a homomorphism
\begin{equation*}
 \log_I:\kone{\wh\sR[\ocG]^\tau,I}\lra \frac{I}{[\wh\sR[\ocG]^\tau, I]}.
\end{equation*}
\end{propn}
\begin{proof}
 The proof is the same as the proof of \cite[Prop 67]{kakde}.
\end{proof}
\subsection{The Integral Logarithm over $\wh{\II(Z)}_{(p)}$}
We now define the integral logarithm over the ring $\wh{\II(Z)}_{(p)}$. For this we first consider the kernel 
\begin{equation*}
J:=\mathrm{ker}\left[\wh{\IIG_\sS}\lra\wh{\II(\Gamma)_{(p)}}\right].
\end{equation*}
As the ring $\wh{\IIG_\sS}$ is local, we have the surjective map
\begin{equation*}
 \begin{split}
  \wh{\IIG_\sS}^\times\lra\konep{\wh{\IIG_\sS}} \mbox{ and }  &1+J\lra\konep{\wh{\IIG_\sS},J}.
 \end{split}
\end{equation*}
Now consider the exact sequence of groups which is split by the embedding $\Gamma\inj\cG$,
\begin{equation*}
 1\lra 1+J\lra\wh{\IIG_\sS}^\times\lra\wh{\II(\Gamma)_{(p)}}^\times\lra 1.
\end{equation*}
It is easy to see that any $x\in\wh{\IIG_\sS}^\times$ can be expressed uniquely as $x=uy$, where $u\in 1+J$ and $y\in \wh{\II(\Gamma)_{(p)}}$. Hence, any 
$x\in\konep{\wh{\IIG_\sS}}$ can be written uniquely as a product $x=uy$, where $y\in\konep{\wh{\II(\Gamma)_{(p)}}}$ and $u$ lies in the image of $\kone{\wh{\IIG_\sS},J}$
in $\konep{\wh{\IIG_\sS}}$. We record this fact below.
\begin{lemma}
 Any 
$x\in\konep{\wh{\IIG_\sS}}$ can be written uniquely as a product $x=uy$, where $y\in\konep{\wh{\II(\Gamma)_{(p)}}}$ and $u$ lies in the image of $\kone{\wh{\IIG_\sS},J}$
in $\konep{\wh{\IIG_\sS}}$.
\end{lemma}
Recall the following assumption
\begin{description}
 \item[(Ur)] $\II=\cO[[X_1,\cdots,X_r]]$ with $\cO$ unramified over $\Z_p$.
\end{description}
Then, $\wh{\II(\Gamma)_{(p)}}/p\wh{\II(\Gamma)_{(p)}}\cong\FF_q[[X_1,\cdots,X_r]][[\Gamma]]$, where $\FF_q$ is the finite field of order $q$ and characteristic $p$.
Consider the map 
\begin{equation*}
 \varphi:\FF_q[[X_1,\cdots,X_r]][[\Gamma]]\lra \FF_q[[X_1,\cdots,X_r]][[\Gamma]]
\end{equation*}
such that its restriction to $\FF_q$ is the Frobenius and  maps $X_j^n$ to $X_j^{pn}$ for all $j=1,\cdots,r$ and the 
$p$-th power map on $\Gamma$.
\begin{lemma}
 Let $y\in\kone{\wh{\II(\Gamma)_{(p)}}}$. Then
 \begin{equation*}
  \frac{y^p}{\varphi(y)}\equiv 1\pmod{p\wh{\II(\Gamma)_{(p)}}}.
 \end{equation*}
Therefore, $\Log(\frac{y^p}{\varphi(y)})$ is well-defined.
\end{lemma}
\begin{proof}
 It is enough to show the congruence and this follows by computing $y^p$. For this, let $\bar y\in\FF_q[[X_1,\cdots,X_r]][[\Gamma]]\cong\FF_q[[X_1,\cdots,X_r]][[T]]$. Then 
 $\bar y=\sum_{i=0}^\infty a_iT^i$, with $a_i\in\FF_q[[X_1,\cdots,X_r]]$. Then $\bar y^p=\left(\sum_{i=0}^\infty a_iT^i\right)^p=\sum_{i=0}^\infty a_i^pT^{ip}=\varphi(\bar y)$.
 Therefore $\frac{y^p}{\varphi(y)}\equiv 1\pmod{p\wh{\II(\Gamma)_{(p)}}}.$
\end{proof}
\begin{defn}
 Let $x\in\konep{\wh{\IIG_\sS}}$. Then $x=uy$, with $y\in\konep{\wh{\II(\Gamma)_{(p)}}}$ and $u$ lies in the image of $\kone{\wh{\IIG_\sS},J}$
in $\konep{\wh{\IIG_\sS}}$. We define the \emph{integral} logarithm map on $\konep{\wh{\IIG_\sS}}$ by
\begin{equation*}
 L(x)=L(uy)=L(u)+L(y)=\Log(u)-\frac{1}{p}\varphi(\Log(u))+\frac{1}{p}\Log\left(\frac{y^p}{\varphi(y)}\right).
 \end{equation*}
\end{defn}
For this integral logarithm, we have the following result which is proven exactly as in \cite[Prop 74]{kakde}.
\begin{propn}
 The integral logarithm map defined above induces a homomorphism 
 \begin{equation*}
L:\konep{\wh{\IIG_\sS}}\lra{\wh{\II(Z)}_{(p)}}[\mathrm{Conj}(\ocG)]^\tau,  
 \end{equation*}
which is independent of the choice of the splitting of $\cG\lra\Gamma$.
\end{propn}
\subsection{The Logarithm map under restriction maps}
For any $P\leq\ol\cG$, recall the map
\begin{equation*}
 t_P^{\ol\cG}:\sR[\Conj{\ol\cG}]^\tau\lra \sR[P^{ab}]^\tau
\end{equation*}
defined by 
\begin{equation*}
t_P^{\ol\cG}(\bar g)=\sum_{x\in C(\ol\cG,P)}\{ ({\bar x}^{-1})(\bar g)(\bar x)\mid x^{-1}gx\in P \},
\end{equation*}
where $C(\ol\cG,P)$ is the set of left coset representatives of $P$ in $\cG$.

Let $\KK=K[[X_1,\cdots,X_r]]$. Then the map $t_P^{\ocG}$ can be naturally extended to a map
\begin{equation*}
 t_P^{\ol\cG}:\KK[[Z]][\Conj{\ol\cG}]^\tau\lra \KK[[Z]][\Conj{P}]^\tau.
\end{equation*}
\begin{lemma}\label{k-restr}
 For any $P\leq\ocG$, we have the commutative diagram
 \begin{equation*}
  \xymatrix{
  \konep{\IIG}\ar[r]^-{\log}\ar[d]_{\Theta_P^{\ocG}}  & \KK[[Z]][\Conj{\ol\cG}]^\tau\ar[d]^{t^{\ocG}_P}\\
  \konep{\II(U_P)}\ar[r]^-{\log}    & \KK[[Z]][\Conj{P}]^\tau
  }
 \end{equation*}
Similarly, for $J=\mathrm{ker}\left[\wh{\IIG_\sS}\lra\wh{\II(\Gamma)_{(p)}}\right]$, we also have
 \begin{equation*}
  \xymatrix{
  \kone{\wh{\IIG_\sS},J}\ar[r]^-{\Log}\ar[d]_{\wh{\Theta_P^{\ocG}}}  & \wh{\II(Z)_{(p)}}[\Conj{\ol\cG}]^\tau[\frac{1}{p}]\ar[d]^{t^{\ocG}_P}\\
  \kone{\wh{\II(U_P)},J}\ar[r]^-{\Log}        & \wh{\II(Z)_{(p)}}[\Conj{P}]^\tau[\frac{1}{p}].
  }
 \end{equation*}
\end{lemma}
The proof of this lemma proceeds exactly as in \cite[Theorem 6.8]{ol} and for the second part 
consider any $u\in\kone{\wh{\IIG_\sS},J}$, then the commutativity follows from the following equalities
\begin{equation*}
 \begin{split}
  \Log(u) &=\lim_{n\to\infty}\frac{1}{p^n}(u^{p^n}-1)\\
  \wh{\Theta_P^{\ocG}}(u)&=\lim_{n\to\infty}(1+t^{\ocG}_P(u^{p^n}-1))^{1/{p^n}}.
 \end{split}
\end{equation*}
Next, proceeding as in the proof of \cite[Lemma 77]{ol}, we get the following commutative diagram.
\begin{lemma}\label{alpha-restr}
 Let $P\in C(\ocG)$ be a non-trivial subgroup. Then, the following diagram is commutative:
 \begin{equation*}
  \xymatrix{
  {\II(U_P)}^\times\ar[r]^-{\log}\ar[d]_{\alpha_P}  & \KK[U_P]\ar[d]^{p\eta_P}\\
  {\II(U_P)}^\times\ar[r]^-{\log}    & \KK[U_P]
  }
 \end{equation*}
 Similarly, the following diagram is also commutative
 \begin{equation*}
  \xymatrix{
  \kone{\wh{\II(U_P)_\sS},J}\ar[r]^-{\Log}\ar[d]_{\alpha_P}  & \wh{\II[U_P]_{\sS}}[\frac{1}{p}]\ar[d]^{p\eta_P}\\
  \kone{\wh{\II(U_P)},J}\ar[r]^-{\Log}        & \wh{\II[U_P]_{\sS}}[\frac{1}{p}].
  }
 \end{equation*}
\end{lemma}
To establish a compatibility between the subgroups, we also consider the following maps.
\begin{defn}\label{def-v}
 Define a map $v_P^{\ocG}:\prod_{C\leq\ocG}\KK[[Z]][C^{ab}]\lra\KK[[Z]][P^{ab}]$, as follows:
 
 If $P$ is not cyclic, then
 \begin{equation*}
 v_P^{\ocG}((x_C))=\left(\sum_{P'}\frac{[P:P']}{P:(P')^p}\varphi(x_{P'})\right)
 \end{equation*}
where $P'$ runs over all subgroups contained in $C(\ocG)$ such that $(P')^p\leq P$.

If $P$ is cyclic, then 
\begin{equation*}
 v_P^{\ocG}((x_C))=\sum_{P'}[P:(P')^p]\varphi(x_{P'})=p\sum_{P'}\varphi(x_{P'}),
 \end{equation*}
where $P'$ runs over all the $P'\in C(\ocG)$ with $(P')^p=P$ but $P'\neq P$.

We set $v^{\ocG}=(v^{\ocG}_P)_P$. 

Analogously, we define maps in the case of the $p$-adic completions, and we denote them again by $v^{\ocG}$.
\end{defn}
Then we can show the following lemma as in \cite[Lemma 79]{kakde}.
\begin{lemma}\label{beta-restr}
 Let $P\neq1$. Then the following diagram is commutative
 \begin{equation*}
  \xymatrix{
  \KK[[Z]][\Conj{\ocG}]^\tau\ar[r]^\varphi\ar[d]_-{\beta^{\ocG}} &\KK[[Z]][\Conj{\ocG}]^\tau\ar[d]^{\beta^{\ocG}_P}\\
  \prod_{C\leq\ocG}\KK[[Z]][C^{ab}]\ar[r]_-{v_P^{\ocG}}          &\KK[[Z]][P^{ab}].
  }
 \end{equation*}
Let $P=\{1\}$, then we have the following commutative diagram
 \begin{equation*}
  \xymatrix{
  \KK[[Z]][\Conj{\ocG}]^\tau\ar[r]^\varphi\ar[d]_-{\beta^{\ocG}} &\KK[[Z]][\Conj{\ocG}]^\tau\ar[d]^{\beta^{\ocG}_P}\\
  \prod_{C\leq\ocG}\KK[[Z]][C^{ab}]\ar[r]_-{\varphi+v_1^{\ocG}}          &\KK[[Z]].
  }
 \end{equation*}
 Analogous results hold for the $p$-adic completions.
\end{lemma}
\begin{defn}\label{defn-u}
 Define the map $u_P^{\ocG}:\prod_{C\leq\ocG}\II(U^{ab}_{C})^\times\lra\II(U^{ab}_P)^\times$ as follows:
 
 If $P$ is not a cyclic subgroup of $\ocG$, then 
 \begin{equation*}
  u^{\ocG}_P((x_C))=\prod_{P'}\varphi(x_{P'})^{|P'|},
 \end{equation*}
where $P'$ runs over all subgroups contained in $C(\ocG)$ such that $(P')^p\leq P$.

If $P$ is cyclic, then 
\begin{equation*}
  u^{\ocG}_P((x_C))=\prod_{P'}\varphi(x_{P'}),
 \end{equation*}
where $P'$ runs over all the $P'\in C(\ocG)$ with $(P')^p=P$ but $P'\neq P$.

Then, we define the collection of maps by $u^{\ocG}=(u^{\ocG}_P)_P$. Analogously, we define maps in the case of the $p$-adic completions, and we denote them again by $v^{\ocG}$.
\end{defn}
As the logarithm maps that we have defined respects group homomorphisms of Iwasawa algebras induced by the group homomorphism, we have the following lemma.
\begin{lemma}\label{u-restr}
 Let $P$ be a non-cyclic subgroup of $\ocG$. Then the following diagram is commutative.
 \begin{equation*}
 \xymatrix{
  \prod_{C\leq\ocG}\II(U_C^{ab})^\times\ar[r]^-{\log}\ar[d]_{u_P^{\ocG}}   &\prod_{C\leq\ocG} \KK[[Z]][C^{ab}]\ar[d]^{|P|v_P^{\ocG}}\\
  \II(U_P^{ab})^\times\ar[r]_-\log                                         &\KK[[Z]][P^{ab}].
  }
 \end{equation*}
Let $P$ be a cyclic subgroup of $\ocG$. Then the following diagram is commutative.
\begin{equation*}
 \xymatrix{
  \prod_{C\leq\ocG}\II(U_C^{ab})^\times\ar[r]^-{\log}\ar[d]_{u_P^{\ocG}}   &\prod_{C\leq\ocG} \KK[[Z]][C^{ab}]\ar[d]^{\frac{1}{p}v_P^{\ocG}}\\
  \II(U_P^{ab})^\times\ar[r]_-\log                                         &\KK[[Z]][P^{ab}].
  }
 \end{equation*}
Recall that $\wh\sR=\wh{\II(Z)_{(p)}}$ and let $J_P=\mathrm{ker}\left[\wh\sR[P^{ab}]\lra\wh{\II(\Gamma)_{(p)}}\right]$.

Let $P$ be a non-cyclic subgroup of $\ocG$. Then we have the following commutative diagram.
\begin{equation*}
\xymatrix{
 \prod_{C\leq\ocG}1+J_C\ar[r]^-{\Log}\ar[d]^{U_P^{\ocG}}   & \prod_{C\leq\ocG}\Q_p\otimes J_C\ar[d]^{|P|v_P^{\ocG}}\\
 1+J_P\ar[r]_-{\Log}                                      & \Q_p\otimes J_P.
 }
\end{equation*}
Let $P$ be a cyclic subgroup of $\ocG$. Then the following diagram is commutative.
\begin{equation*}
\xymatrix{
 \prod_{C\leq\ocG}1+J_C\ar[r]^-{\Log}\ar[d]^{U_P^{\ocG}}   & \prod_{C\leq\ocG}\Q_p\otimes J_C\ar[d]^{\frac{1}{p}v_P^{\ocG}}\\
 1+J_P\ar[r]_-{\Log}                                      & \Q_p\otimes J_P.
 }
\end{equation*}
\end{lemma}
\begin{lemma}\label{alpha-formula}
 Let $x\in\konep{\sR[\ocG]^\tau}$ or $\konep{\wh{\sR}[\ocG]^\tau}$. Then for every non-cyclic subgroup $P\leq\ocG$ , we have
 \begin{equation*}
  \alpha_P(\theta_P^{\ocG}(x))^{p|P|}\equiv u_P^{\ocG}(\alpha(\theta^{\ocG}(x)))\pmod{p}.
 \end{equation*}
In particular, the logarithm $\log\left(\frac{\alpha_P(\theta_P^{\ocG}(x))^{p|P|}}{u_P^{\ocG}(\alpha(\theta^{\ocG}(x)))}\right)$ is well-defined.
\end{lemma}
\begin{proof}
 Let $C$ be cyclic. Then $\alpha_C(\theta_C^{\ocG}(x))\equiv 1\pmod{p}$ and $u_P^{\ocG}(\alpha(\theta^{\ocG}(x)))\equiv\varphi(\alpha_{\{1\})}(\Theta_{\{1\}}\ocG))\pmod{p}$.
 Then the result follows from the congruence $\theta_P^{\ocG}(x)^{|P|}\equiv\theta_{\{1\}}^{\ocG}(x)\pmod{p}$, which follows from a straightforward generalization of 
 \cite[Prop 2.3]{sv}.
\end{proof}
We now give the relation between the multiplicative and the additive sides. For completeness and 
to see how the lemmas proved above are used, we also give a proof of one of the formulas on the lines of \cite[Prop 84]{kakde}.
\begin{propn}\label{beta-formula}
 Let $x\in\konep{\IIG}$. Then
 \begin{equation*}
   \beta_P^{\ocG}(L(x))=\begin{cases}
                         \frac{1}{p^2|P|}\log\left(\frac{\alpha_P(\theta_P^{\ocG}(x))^{p|P|}}{u_P^{\ocG}(\alpha(\theta^{\ocG}(x)))}\right), &\mbox{ if } P\notin C(\ocG) \\
                         \frac{1}{p}\log\left(\frac{\alpha_P(\theta_P^{\ocG}(x))^{p|P|}}{u_P^{\ocG}(\alpha(\theta^{\ocG}(x)))}\right),    &\mbox{ if } P\in C(\ocG), P\neq\{1\}\\
                         \frac{1}{p}\log\left(\frac{\alpha_{\{1\}}(\theta_{\{1\}}^{\ocG}(x))^{p}}
                                                                 {\varphi(\theta_{\{1\}}^{\ocG}(x))(u_{\{1\}}^{\ocG}(\alpha(\theta^{\ocG}(x))))}\right), &\mbox{ if } P=\{1\}.
                        \end{cases}
 \end{equation*}
We also have analogous relations over the $p$-adic completions $\konep{\wh{\IIG_\sS}}$.
\end{propn}
\begin{proof}
 Let $x\in\konep{\IIG}$. In the first case, we consider a group $P\in C(\ocG)$. Then we have
\begin{align*}
\beta_P^{\ocG}(L(x))=&\beta_P^{\ocG}(\log(x)-\frac{\varphi}{p}(\log(x))) \\
                        =&\frac{1}{p}\log(\alpha_P(\Theta_P^{\ocG}(x)))-\beta_P^{\ocG}(\frac{\varphi}{p}(\log(x)), &(\mbox{ Lemmas } \ref{k-restr}, \ref{alpha-restr})\\
                        =&\frac{1}{p}\log(\alpha_P(\Theta_P^{\ocG}(x)))-\frac{1}{p}v_P^{\ocG}(\beta^{\ocG}(\log(x), &(\mbox{ Lemma } \ref{beta-restr})\\
                        =&\frac{1}{p}\log(\alpha_P(\Theta_P^{\ocG}(x)))-\frac{1}{p^2}v_P^{\ocG}(\log(\alpha(\Theta^{\ocG(x)}))), &(\mbox{ Lemmas } \ref{k-restr}, \ref{alpha-restr}) \\
                        =&\frac{1}{p}\log(\alpha_P(\Theta_P^{\ocG}(x)))-\frac{1}{p^2|P|}\log(u_P^{\ocG}(\alpha(\Theta^{\ocG(x)}))), &(\mbox{ Lemma } \ref{u-restr})\\
                        =&\frac{1}{p^2|P|}\log\left(\frac{\alpha_P(\theta_P^{\ocG}(x))^{p|P|}}{u_P^{\ocG}(\alpha(\theta^{\ocG}(x)))}\right) 
\end{align*}

If $P\in C(\ocG), P\neq\{1\}$, or $P=\{1\}$, then the formula for $\beta_P^{\ocG}$ can also be shown similarly as in \cite[Prop 84]{kakde}.

We now give a proof if $x\in\konep{\wh{\IIG_\sS}}$. For this, we first note that we can write $x=uy$, for some 
$u\in\mathrm{im}\left[\kone{\wh{\IIG_\sS},J}\lra\konep{\wh{\IIG}_\sS}\right]$ and $y\in\konep{\wh{\II(\Gamma)}_{(p)}}$. 
(Recall that $J=\mathrm{ker}[\wh{\IIG_\sS}\lra\wh{\II(\Gamma)_{(p)}}]$). The proof for $u$ is the same as above, and in fact, we may show that it is true for any $u$ in the 
image of $\kone{\wh{\IIG_\sS},J_\sR}$, where $\wh{J}_\sR$ is the Jacobson radical of $\wh{\IIG_\sS}$.

We now show it for $y\in\wh{\II(Z)_{(p)}}$. Note that $\Theta_P^{\ocG}(y)=y^{[\ocG:P]}$, for all $P\in\ocG$. Further, if $P$ is a non-trivial cyclic subgroup of $\ocG$, then
$\alpha_P(\Theta_P^{\ocG}(y))=1$. On the other hand, since $L(y)\in\wh{\II(Z)_{(p)}}$, we have
\begin{equation*}
 \beta_P^{\ocG}(L(y))=\begin{cases}
                       [\ocG:P]L(y), &\mbox{ if } P \mbox{ is noncyclic or } P=\{1\}\\
                       0,            &\mbox{ if } P \mbox{ is cyclic}.
                      \end{cases}
\end{equation*}
Now, it is easy to see that the formula for $\beta_P^{\ocG}(L(y))$ holds. We now consider the case when $y\in\wh{\II(\Gamma)_{(p)}}$. In this case, 
$\varphi(y)^r\in\wh{\II(Z)_{p}}$, for some $r$. Then $\frac{\varphi(y)^r}{y^{p^r}}\equiv 1\pmod{\wh{\II(\Gamma)_{(p)}}}$ and hence in the image of $\kone{\wh{\IIG_\sS},J_\sR}$.
Therefore the formula holds for $\frac{\varphi(y)^r}{y^{p^r}}$ and also $\varphi^r(y)$. Hence, the formula holds for $y^{p^r}$. Since the image of $\beta_P^{\ocG}$ is torsion 
free abelian group, therefore the formula holds for $\beta_P^{\ocG}$.
\end{proof}

\subsection{Congruences over $\IIG$}
For any subgroup $P$ of $\ol\cG$, we write $\maptheta{\ol\cG,ab}{P}$ for the following natural composite homomorphism
\begin{equation*}
 \konep{\IIG}\stackrel{\maptheta{\ol\cG}{P}}{\lra}\konep{\II(U_P)}\lra\kone{\II(U_P^{ab})}\cong\II(U_P^{ab})^\times,
\end{equation*}
where the isomorphism is induced by taking determinants over $\II(U_P^{ab})$. We now show that the image of the map $\maptheta{\ol\cG}{}=(\maptheta{\ocG}{P})$ lies in $\Phi^\cG$.
\begin{theorem}\label{cong}
 Let $\Xi\in\konep{\IIG}$ and for all subgroups $P$ of $\ol\cG$, put $\Xi_{U_P^{ab}}:=\maptheta{\ol\cG,ab}{P}(\Xi)\in\II(U_P^{ab})^\times$.
\begin{enumerate}
 \item\label{C1} For all subgroups $P, P'$ of $\ol\cG$ with $[P',P']\leq P\leq P'$, we have
\begin{equation*}
 \mathrm{Nr}_P^{P'}(\Xi_{U_{P'}^{ab}})=\Pi_P^{P'}(\Xi_{U_{P'}^{ab}}).
\end{equation*}
 \item\label{C2} For all subgroups $P$ of $\ol\cG$ and all $g$ in $\ol\cG$ we have $\Xi_{gU_{P}^{ab}g^{-1}}=g\Xi_{U_{P}^{ab}}g^{-1}$.
 \item\label{C3}For every $P\in\ocG$ and $P\neq (1)$, we have
 \begin{equation*}
  \mathrm{ver}_P^{P'}(\Xi_{U_{P'}^{ab}})\equiv \Xi_{U_P^{ab}} \pmod{\sT_{P,P'}} (\mbox{ resp. } \sT_{P,P',\sS} \mbox{ and } \widehat\sT_{P,P'}).
 \end{equation*}
 \item\label{C4} For all $P\in C(\ol\cG)$  we have $\alpha_P(\Xi_{U_{P}^{ab}})\equiv\prod_{P'\in C_P(\ol\cG)}\alpha_{P'}(\Xi_{U_{P'}^{ab}})\pmod{p\sT_P}$.
\end{enumerate}
\end{theorem}
To prove this theorem, we recall an explicit description of the map $\maptheta{P',ab}{P}$. We write $n_{P'/P}:=[P':P]=[U_{P'}:U_P]$. Since 
$\II(U_{P'})$ is a local ring, the natural homomorphism 
\begin{equation*}
 q_{P'}:\II(U_{P'})^\times\lra\kone{\II(U_{P'})}
\end{equation*}
is \emph{surjective}.
For any $\Xi\in\kone{\II(U_{P'})}$, let $\wt\Xi\in\II(U_{P'})^\times$ denote a pre-image under $q_{P'}$. 
We denote the set of left coset representatives of $U_p$ in $U_{P'}$ by $C(P',P):=\left\lbrace c_i:1\leq i\leq n_{P'/P}\right\rbrace$. 
Then as an $\II(U_{P'})$-module we have
\begin{equation*}
 \II(U_{P'})\cong\bigoplus_{i=1}^{n_{P'/P}}\II(U_P)c_i.
\end{equation*}
Let $M_{C(P',P)}(\wt\Xi)$ denote the matrix in $M_{n_{P'/P}}(\II(U_P))$ of the automorphism given by multiplication by $\wt\Xi$ on the right,
and $\Pi_{P',P}:M_{n_{P'/P}}(\II(U_P))\lra M_{n_{P'/P}}(\II(U_P^{ab}))$ denote the natural projection. Then
\begin{equation*}
 \maptheta{P',ab}{P}(\Xi)=\det\left(\Pi_{P',P}(M_{C(P',P)}(\wt\Xi))^{}\right)\in\II(U_{P}^{ab})^\times,
\end{equation*}
\begin{proof}[Proof of Theorem \ref{cong}(\ref{C1}):] Consider the following diagram
\begin{equation*}
 \xymatrix{
\konep{\IIG}\ar[r]^{\maptheta{\ol\cG}{P}}\ar[d]^{\maptheta{\ol\cG}{P'}} &               \konep{\II(U_{P})}\ar[d]^{\pi_P} &\\
\konep{\II(U_{P'})}\ar[r]^{\maptheta{P'}{P}}\ar[d]^{\pi_{P'}}               &               \II(U_{P}^{ab})^\times\ar[dd]^{\Pi^{P'}_P} \\
\II(U_{P'})^\times\ar[rd]^{\mathrm{Nr}_P^{P'}} &                                              \\
                    &                                                       \II\left(U_P/[U_{P'},U_{P'}]\right)^\times.
}
\end{equation*}
The upper quadrilateral in the diagram is obviously seen to be commutative. The lower quadrilateral is also commutative since the coset space 
$C(P',P)$ can be regarded as an $\II(U_P/[U_{P'},U_{P'}])$-basis of $\II(U_{P'}^{ab})$. Therefore, we have
\begin{equation*}
 \mathrm{Nr}_P^{P'}(\Xi_{P'})=\mathrm{Nr}_P^{P'}(\pi_{P'}(\wt\Xi))=\Pi^{P'}_P(\det\left(\Pi_{P',P}(M_{C(P',P)}(\wt\Xi))\right))
                                                                  =\Pi^{P'}_P(\maptheta{P',ab}{P}(\Xi)).
\end{equation*}
\end{proof}
\begin{proof}[Proof of Theorem \ref{cong}(\ref{C2}):] Let $C:=C(P,\ol{\cG})$. Then, for any $g\in\ol\cG$, the set $gCg^{-1}:=\{gc_ig^{-1}\mid c_i\in C\}$ 
is a set of left coset representatives of $gU_Pg^{-1}=U_{gPg^{-1}}$ in $\cG$. By definition, we have
\begin{equation*}
 \Xi_{gPg^{-1}}=\maptheta{\ol\cG,ab}{P}(\Xi)=\det\left(\Pi_{\ol\cG,gPg^{-1}}(gM_{C}(\wt\Xi)g^{-1})\right)=g\det\left(\Pi_{\ol\cG,P}(M_{C}(\wt\Xi))\right)g^{-1}=g\Xi_{P}g^{-1}.
\end{equation*}
The equality follows from this.
 \end{proof}
\begin{proof}[Proof of Theorem \ref{cong}(\ref{C3}):] The proof of (C3) is same as the proof of (M3) in \cite[Lemma 85]{kakde}.
\end{proof}
For the proof of Theorem \ref{cong}(\ref{C4}), we need the following lemma.
\begin{lemma}\label{log-eta-res}
For all $x\in\konep{\IIG}$ and all $P\in C(\ol\cG)$, we have,
\begin{equation*}
\log_P\left(\dfrac{\alpha_P(\maptheta{\cG}{P}(x))}{\prod_{P'\in C_P(\ol\cG)}\alpha_{P'}(\maptheta{\cG}{P}(x))}\right)=
p(\eta_P\circ\Res{\ol\cG}{P})(\mathfrak{L}_{\ol\cG}(x)).
\end{equation*}
\end{lemma}
\begin{proof}
 The lemma follows from the commutativity of the following diagram
\begin{equation*}
 \xymatrix{
\konep{\IIG}\ar[r]^{\log_{\ol\cG}}\ar[d]^{\maptheta{\ol\cG}{P}} &\II[\Conj{\ol\cG}]^\tau[\frac{1}{p}]\ar[d]_{\Res{\ol\cG}{P}}\ar@/^1pc/[ddr]^a    \\
\prod_{P\in C(\ol\cG)}\konep{\II(U_P)}\ar[r]^{\log_{P}}\ar@/_1pc/[rrd]_b             
                                                                   &\prod_{P\in C(\ol\cG)}\II[P]^\tau[\frac{1}{p}]\ar[dr]^c                 \\
                                                                   &                     &\prod_{P\in C(\ol\cG)}\II[P]^\tau[\frac{1}{p}],   \\
}
\end{equation*}
where
\begin{eqnarray*}
 a(x)     &:=&\left(\eta_P(\Res{\ol\cG}{P}((1-p^{-1}\varphi_{conj})(x)))\right)_P\\
b((x_P)_P)&:=&\left(
                    p^{-1}\log_Q\left(\dfrac{\alpha_Q(\maptheta{\cG}{Q}(x_Q))}{\prod_{P'\in C_Q(\ol\cG)}\alpha_{P'}(\maptheta{\cG}{P}(x_{P'}))}\right)
              \right)_Q.
\end{eqnarray*}
By Lemma \ref{log-theta-res}, the square in the diagram is commutative. To show that the triangles commute, as in \cite{kakde}, the map $c$
is chosen as 
\begin{equation*}
 c((x_P)_P):=\left(1-\delta_Pp^{-1}\varphi_{conj}(x_P)-\sum_{P'\in C_P(\cG)}\varphi_{conj}(x_{P'})\right)_P,
\end{equation*}
where $\delta_P=1$ if $P$ is non-trivial, and $0$ otherwise. Then the commutativity of the triangles follow analogously as in \cite[Lemma 7.4]{kakde}.
\end{proof}
\begin{lemma}\label{res-TP}
 We have $\Res{\ol\cG}{P}(\sR[\Conj{\ol\cG}]^\tau)\subset \sT_P$.
\end{lemma}
\begin{proof}
 Let $x\in \sR[\Conj{\ol\cG}]^\tau$. Then $x:=\sum_{(i_1,\cdots,i_n)\geq 0}(\sum_{g\in\ol\cG}c_{g,(i_1,\cdots,i_n)}\kappa(g))X_1^{i_1}\cdots X_n^{i_n}$, where $c_{g,(i_1,\cdots,i_n)}\in \sR$, for all $g$ and $(i_1,\cdots,i_n)$. 
 Consider the normalizer $N_{\ol\cG}(P)$ of $P$ in $\ol\cG$. Then
 $\Res{\ol\cG}{P}=\Res{N_{\ol\cG}(P)}{P}\circ\Res{\ol\cG}{N_{\ol\cG}(P)}$. 
 Therefore 
\begin{equation*}
 \Res{\ol\cG}{P}(x)=\sum_{(i_1,\cdots,i_n)\geq 0}\left(\sum_{g\in\ol\cG}c_{g,(i_1,\cdots,i_n)}\Res{N_{\ol\cG}(P)}{P}\left(\Res{\ol\cG}{N_{\ol\cG}(P)}(\kappa(g))\right)\right)X_1^{i_1}\cdots X_n^{i_n}.
\end{equation*}
Since for every $h,h'\in N_{\ol\cG}(P)$, $h^{-1}h'h\in P$ if and only if $h'\in P$, by equation \eqref{kappa}, we have
\begin{equation*}
\Res{N_{\ol\cG}(P)}{P}(\kappa_{N_{\ol\cG}(P)}(h'))= \begin{cases}
                            \sum_{x\in W_{\ol\cG}(P)}\kappa_P((x^\tau)^{-1}h'^\tau x^\tau), & \mbox{ if } h'\in P,\\
                            0,                                                             & \mbox{ otherwise}.
                          \end{cases}
\end{equation*}
It follows that each term $c_{g,(i_1,\cdots,i_n)}\Res{N_{\ol\cG}(P)}{P}\left(\Res{\ol\cG}{N_{\ol\cG}(P)}(\kappa(g))\right)X_1^{i_1}\cdots X_n^{i_n}\in \sT_P$.
\end{proof}
\begin{proof}[Proof of Theorem \ref{cong}(\ref{C4}):] 
Note that $p\sT_P\subset p\sR[\ol\cG]^\tau$. Taking $I=p\sR[\ol\cG]^\tau$ in Lemma \ref{log-sub}, we have an isomorphism $I\stackrel{log_{\ol\cG}^{-1}}{\lra} 1+I$. Then
it follows from Proposition \ref{int-log} and Lemma \ref{log-eta-res}, that the 
congruences follow if $\eta_P\circ\Res{\ol\cG}{P}(\sR[\Conj{\ol\cG}]^\tau)\subset \sT_P$. Since $\eta_P$ preserves $\sT_P$, it follows from the 
Lemma \ref{res-TP} above, that the containment holds and hence the congruence. 
\end{proof}
This finishes the proof of Theorem \ref{cong}, and the Theorem D.
By this theorem, to show that an element is in the image of $\konep{\IIG}$ under the map $\Theta^\cG$ it is sufficient to verify the statements in Theorem \ref{cong}.
\begin{defn}
 We now consider the map $\cL=(\cL_P):\Phi^{\ocG}\lra\Psi^{\ocG}$, defined by
 \begin{equation*}
  \cL_P((x_C))=\begin{cases}
               \frac{1}{p^2|P|}\log\left(\frac{\alpha_P(\theta_P^{\ocG}(x_P))^{p|P|}}{u_P^{\ocG}(\alpha(\theta^{\ocG}(x_C)))}\right), &\mbox{ if } P\notin C(\ocG) \\
               \frac{1}{p}\log\left(\frac{\alpha_P(\theta_P^{\ocG}(x_P))^{p|P|}}{u_P^{\ocG}(\alpha(\theta^{\ocG}(x_C)))}\right),    &\mbox{ if } P\in C(\ocG), P\neq\{1\}\\
               \frac{1}{p}\log\left(\frac{\alpha_{\{1\}}((x_{\{1\}}))^{p}}
                                         {\varphi((x_{\{1\}}))(u_{\{1\}}^{\ocG}(\alpha(\theta^{\ocG}(x_{\{1\}}))))}\right), &\mbox{ if } P=\{1\}. 
               \end{cases}
 \end{equation*}
\end{defn}
\begin{lemma}
 The following sequence is exact
 \begin{equation*}
  1\lra\mu(\cO)\times\WW\times\cG^{ab}\lra\Phi^{\ocG}\stackrel{\cL}{\lra}\Psi^{\ocG}\stackrel{\omega}{\lra}\WW\times\cG^{ab}\lra 1.
 \end{equation*}
More precisely, the map $\mu(\cO)\times\WW\times\cG^{ab}\lra\Phi^{\ocG}\subset\prod_{P\leq\ocG}\II(U_P^{ab})^\times$ is the composition 
\begin{equation*}
 \mu(\cO)\times\WW\times\cG^{ab}\lra\prod_{P\leq\ocG}\mu(\cO)\times\WW\times U_P^{ab}\inj\prod_{P\leq\ocG}\II(U_P^{ab})^\times
\end{equation*}
where the first map is the identity on $\mu(\cO)$ and the transfer homomorphism from $\cG^{ab}$ to $U_P^{ab}$ for each $P\leq\ocG$.
\end{lemma}
\begin{proof}
 Clearly the image of $\cL$ is contained in $\prod_{P\leq\ocG}\Q_p\otimes\II(U_P^{ab})$. To show that the image is contained in $\Psi^{\ocG}\prod_{P\leq\ocG}\II(U_P^{ab})$, we
 show that the conditions defining the set $\Psi^{\ocG}$ are satisfied. 
Below we show how the first condition defining $\Psi^{\ocG}$ can be shown. The rest of the conditions can be demonstrated easily from the conditions defining $\Phi^{\cG}$
(\cite[Lemma 88]{kakde}).

Let $P\leq P'\leq\ocG$ such that $[P',P]\leq P$ with $P$ a non-trivial cyclic group if $[P',P']\neq P$. We then
have three cases to consider: (i) $P$ is not cyclic, (ii) $P$ is cyclic but $P'$ is not cyclic, and (iii) $P'$ is cyclic.

Case (i): Let $P$ be not cyclic. Letting $C'$ run through all cyclic subgroups of $\ocG$ with ${C'}^p\leq P'$ and $C$ run through all cyclic subgroups of $\ocG$ with $C^p\leq P$,
we have
\begin{equation*}
 \begin{split}
 Tr_P^{P'}(\cL_{P'}((x_C)))&=Tr_P^{P'}\left(\frac{1}{p^2|P'|}\log\left(\frac{\alpha_{P'}(x_{P'})^{p|P'|}}{u_{P'}^{\ocG}(\alpha((x_C)))}\right)\right)\\
                           &=Tr_P^{P'}\left(\frac{1}{p^2|P'|}\log\left(\frac{(x_{P'})^{p^2|P'|}}{\prod_{C'}\varphi(\alpha_{C'}(x_{C'}))^{|C'|}}\right)\right)\\
                           &=\frac{1}{p^2|P'|}\log\left(\frac{Nr_P^{P'}(x_{P'})^{p^2|P'|}}{Nr_P^{P'}(\prod_{C'}\varphi(\alpha_{C'}(x_{C'}))^{|C'|})}\right)\\
                           &=\frac{1}{p^2|P'|}\log\left(\frac{\Pi_P^{P'}(x_{P'})^{p^2|P'|}}{\prod_{C}\varphi(\alpha_{C}(x_{C}))^{p|C|}}\right), \mbox{ by first condition of }\Phi^{\cG} \\
                           &=\Pi_P^{P'}\left(\frac{1}{p^2|P'|}\log\left(\frac{\alpha_P(x_{P})^{p|P'|}}{\prod_{C}\varphi(\alpha_{C}(x_{C}))^{|C|}}\right)\right)\\
                           &=\Pi_P^{P'}\left(\cL_P((x_C))\right).
 \end{split}
\end{equation*}
Case (ii): Let $P$ be cyclic but $P'$ be not cyclic. Let $C$ run through all cyclic subgroups of $\ocG$ with $C^p\leq P$. Then
\begin{equation*}
 \begin{split}
  \eta_P(Tr_P^{P'}(\cL_{P'}((x_C))))&=\Pi_P^{P'}\left(\eta_P
                                      \left(\frac{1}{p^2|P|}\log\left(\frac{(x_{P})^{p^2|P|}}{\prod_{C}\varphi(\alpha_{C}(x_{C}))^{|C|}}\right)\right)\right).
 \end{split}
\end{equation*}
 Since $\alpha_P(\varphi(\alpha_C(x_C)))=\alpha_P(\alpha_C(x_C))^p$ (resp. $1$) if $C^p=P$ ( resp. $C^p\neq P$), 
letting $C$ run through all cyclic subgroups of $\ocG$ with $C^p=P$, we have
\begin{equation*}
 \begin{split}
  \eta_P(Tr_P^{P'}(\cL_{P'}((x_C))))&=\Pi_P^{P'}\left(\frac{1}{p}\log\left(\frac{\alpha(x_{P})}{\prod_{C}\varphi(\alpha_{C}(x_{C}))}\right)\right)\\
                                    &=\Pi_P^{P'}(\cL_P((x_C))).
 \end{split}
\end{equation*}
Case (iii): Let $P'$ be cyclic. This case follows from the following lemma:
\begin{lemma}
  Let $P\leq P'\leq\ocG$ such that $[P':P]=p$. Let $C\in C(\ocG)$ be such that $C^p$ is contained in $P'$ but not in $P$. Then $\mathrm{Nr}^{P'}_P(\varphi(\alpha_C(x_C))=1$
  in $\II(U_P/[U_{P'},U_{P'}])$.
 \end{lemma}
 \begin{proof}
 By definition, we have $\alpha_C(x_C)=\frac{x_C^p}{\prod_{k=0}^{p-1}\omega_C^k(x_C)}$. Therefore, 
 $\varphi(\alpha_C(x_C))=\frac{\varphi(x_C^p)}{\prod_{k=0}^{p-1}\varphi(\omega_C^k(x_C))}=\frac{\varphi(x_C^p)}{\prod_{k=0}^{p-1}\omega_{C^p}^k(\varphi(x_C))}$.
 Since $\mathrm{Nr}^{P'}_P(\varphi(\alpha_C(x_C))=\prod_{k=0}^{p-1}\omega_{C^p}^k(\varphi(x_C))$ by a straightforward generalization of \cite[Lemma 50]{kakde}, the lemma 
 follows.
 \end{proof}
The second and third conditions defining $\Psi^{\ocG}$ follow easily from (C2) and (C4) respectively. 
Finally, to show 
that the image of $\cL$ is contained in $\prod_{P\leq\ocG}\II(U_P^{ab})$, it is enough to note that $\cL_P((x_C))\in\sT_P$ for all $P\in C(\ocG)$. Then by Proposition
\ref{id-beta}, $\mathrm{im}(\cL)\subseteq\prod_{P\leq\ocG}\II(U_P^{ab})$.

The exactness of the four term sequence can also shown as in \cite[Lemma 88]{kakde}. However, there is one crucial input which is the fact that the only torsion elements of 
$\II(U_P^{ab})^\times$ are contained in 
$\mu(\cO)\times\WW\times U_P^{ab}$, which is a generalization of a Theorem of Higman \cite{higman}. This input is provided by Proposition \ref{k-tors} and 
Corollary \ref{sk-trivial}. We then have the following commutative diagram:
\begin{equation*}
 \xymatrix{
 \konep{\IIG}\ar[r]^{\mathfrak L}\ar[d]_-{\Theta^{\ocG}}    &\II(Z)[\Conj{\ocG}]^\tau\ar[d]^{\beta^{\ocG}}\\
 \Phi^{\ocG}\ar[r]_-{\cL}                      &\Psi^{\ocG}.
 }
\end{equation*}
In other words, the image of $\Theta^{\ocG}$ is contained in $\Phi^{\ocG}$.
\end{proof}
In the same way, we can prove that the image of $\wh{\Theta_\sS^{\ocG}}$ is contained in $\wh{\Phi_\sS^{\ocG}}$, which we record below.
\begin{theorem}\label{hat-Theta}
 The image of $\wh{\Theta_\sS^{\ocG}}$ under the logarithm map is contained in $\wh{\Phi_\sS^{\ocG}}$.
\end{theorem}
\begin{theorem}\label{Theta-iso}
The map $\Theta^{\ocG}:\konep{\IIG}\lra\Phi^{\ocG}$ is an isomorphism.
\end{theorem}
\begin{proof}
 The lemmas regarding the restrictions under integral logarithms give us the following commutative diagram:
 \begin{equation*}
  \xymatrix{
  1\ar[r] &\mu(\cO)\times\WW\times\cG^{ab}\ar[r]\ar[d]_{=} &\konep{\IIG}\ar[r]^-{\mathfrak L}\ar[d]^-{\Theta^{\ocG}}
                                                                      &\II(Z)[\Conj{\ocG}]^\tau\ar[r]\ar[d]^{\beta^{\ocG}}_\cong &\WW\times\cG^{ab}\ar[r]\ar[d]_{=} &1\\
  1\ar[r] &\mu(\cO)\times\WW\times\cG^{ab}\ar[r]       &\Phi^{\ocG}\ar[r]_{\mathcal L}        &\Psi^{\ocG}\ar[r]                    &\WW\times\cG^{ab}\ar[r]       &1.
  }
 \end{equation*}
The Five Lemma then gives the result.
\end{proof}
\begin{theorem}
 The map $\Theta^{\ocG}_\sS$ maps $\konep{\IIG_\sS}$ into $\Phi^{\ocG}_\sS$. Further
 \begin{equation*}
  {\Phi_\sS^{\ocG}}\cap\prod_{P\leq{\ocG}}\II(U_P^{ab})^\times=\mathrm{im}(\Theta^\cG).
 \end{equation*}
\end{theorem}
\begin{proof}
 Note that $\wh{\Phi_\sS^{\ocG}}\cap\prod_{P\leq{\ocG}}\II(U_P)_\sS^\times=\Phi_\sS^{\ocG}$. 
 By Theorem \ref{cong} and \ref{hat-Theta}, it follows that 
 \begin{equation*}
 \mathrm{im}(\Theta_\sS^{\ocG})\subset\Phi_\sS^{\ocG} 
 \end{equation*}
Further, since ${\Phi_\sS^{\ocG}}\cap\prod_{P\leq{\ocG}}\II(U_P^{ab})^\times=\Phi^{\ocG}$, we get from Theorem \ref{cong}, that 
\begin{equation*}
 {\Phi_\sS^{\ocG}}\cap\prod_{P\leq{\ocG}}\II(U_P^{ab})^\times=\mathrm{im}(\Theta^\cG).
\end{equation*}
\end{proof}

\section{Relations between the congruences over $\II[[\cG]]$ and $\Z_p[[\cG]]$}
\subsection{Congruences over $\Z_p[[\cG]]$}
We first recall the main result of Kakde \cite{kakde}. 
As in the previous section, we fix a lift $\wt\Gamma$ of $\Gamma$ in $\cG$. Then we can identify $\cG$ with $H\rtimes\Gamma$. Fix $e\in\N$ such that 
$\wt\Gamma^{p^e}\subset Z(\cG)$, and put $\ol\cG:=\cG/{\Gamma^{p^e}}$ and $R:=\Lambda_\cO(\wt\Gamma^{p^e})$. Then $\Lambda_\cO(G)\cong R[\ol\cG]^\tau$, 
the twisted group ring with multiplication 
\begin{equation*}(h\wt\gamma^a)^\tau(h\wt\gamma^b)^\tau=\wt\gamma^{p^e[\frac{a+b}{b}]}(h\wt\gamma^a.h'\wt\gamma^b)^\tau,
\end{equation*}
where $g^\tau$ is the image of $g\in \cG$ in $R[\ol\cG]^\tau$.

Let $P$ be a subgroup of $\ol\cG$ and $U_P$ be the inverse image of $P$ in $\cG$. Recall that,  $N_{\ol\cG}P:=$ the normalizer of $P$ in $\cG$, 
$W_{\ol\cG}(P):=N_{\ol\cG}P/P$, and $C(\ol\cG):=$ set of cyclic subgroups of $\ol\cG$.
If $P\in C(\ol\cG)$, then $U_P$ is a rank one abelian subquotient of $\cG$, and for every $P\in C(\ol\cG)$, set 
\begin{equation*}
T_P:=\{\sum_{g\in W_{\ol\cG}(P)}g^\tau x(g^\tau)^{-1}\mid x\in R[P]^\tau\}.
\end{equation*}
Let $P\leq P'\leq\ocG$. Then consider the homomorphism $\Z_p[[\cG]]\lra\Z_p[[\cG]]$ given by $x\mapsto \sum_{g\in P'/P}\tilde{g}x\tilde{g}^{-1}$, where $\tilde{g}$ is a lift of $g$.
We define $T_{P,P'}$ to be the image of this homomorphism.
For two subgroups $P, P'$ of $\ol\cG$ with $[P',P']\leq P\leq P'$ consider 
\begin{equation*}
 \begin{split}
  &\mbox{nr}_P^{P'}:\Lambda_\cO(U_{P'}^{ab})^\times\lra\Lambda_\cO(U_P/[U_{P'},U_{P'}])^\times,\quad\mbox{(the norm map)},\\
  &\pi_P^{P'}:\Lambda_\cO(U_{P'}^{ab})\lra\Lambda_\cO(U_P/[U_{P'},U_{P'}]),\quad\mbox{(the projection map)}.
 \end{split}
\end{equation*}
For $P\in C(\ol\cG)$ with $P\neq (1)$, fix a homomorphism $\omega_P:P\lra\bar\Q_p^\times$ of order $p$, and also a homomorphism 
$\omega_1:=\omega_{\{1\}}:\wt\Gamma^{p^e}\lra\bar\Q_p^\times$ of order $p$. The homomorphism $\omega_P$ induce the following homomorphism which we 
again denote by the same symbol:
\begin{equation*}
 \omega_P:\Lambda_\cO(U_P)^\times\lra\Lambda_\cO(U_P)^\times, g\mapsto \omega_P(g)g.
\end{equation*}
For $P\leq \ol\cG$, consider the homomorphism $\alpha_P:\Lambda_\cO(U_P)_{S}^\times\lra\Lambda_\cO(U_P)_{S}^\times$ defined by
\begin{equation*}
 \alpha_P(x):=\begin{cases}
               x^p\varphi(x)^{-1} &\mbox{ if } P=\{1\}\\
               x^p(\prod_{k=0}^{p-1}\omega_P^k(x))^{-1} &\mbox{ if } P\neq\{1\}\mbox{ and cyclic}\\
               x^p             &\mbox{ if } P \mbox{ is not cyclic}.
              \end{cases}
\end{equation*}
Note that, for all $P\leq\ol\cG$, there is an action of $\cG$ and $\ol\cG$ on $U_P^{ab}$ by conjugation since $\wt\Gamma^{p^e}$ is central. 
Now consider the following map
\begin{equation*}
 \konep{\Z_p[[\cG]]_S}\lra \konep{\Z_p[[U]]_S}\lra \konep{\Z_p[[U^{ab}]]_S}\lra\Z_p[[U^{ab}]]_S^\times\subset Q(\Z_p[[U^{ab}]])^\times.
\end{equation*}
Taking all the $U^{ab}$ in $\Sigma(\cG)$ we get the following homomorphism
\begin{equation*}
 \theta_{\Sigma(\cG)}:\konep{\Z_p[[\cG]]_S}\lra\prod_{U^{ab}\in\Sigma(\cG)}Q(\Z_p[[U^{ab}]])^\times.
\end{equation*}
For any subgroup $P$ of $\ol\cG$, we write $\theta_{\ol\cG,ab}^{P}$ for the following natural composite homomorphism
\begin{equation*}
 \konep{\Z_p[[\cG]]}\stackrel{\theta_{\ol\cG}^{P}}{\lra}\kone{\Z_p[[U_P]]}\lra\kone{\Z_p[[U_P^{ab}]]}\cong\Z_p[[U_P^{ab}]]^\times,
\end{equation*}
where the isomorphism is induced by taking determinants over $\Z_p[[U_P^{ab}]]$.
\begin{defn}\label{phi-G}
As in \cite{kakde}, we  denote the subgroup of $\prod_{P\leq\ol\cG}\Lambda_\cO(U_P^{ab})^\times$ consisting of tuples $(x_P)$ satisfying the conditions of the result below by $\phi^{\ol\cG}$.
\end{defn}
\begin{propn}\emph{\cite{kakde}}
 Let $\cG$ be a rank one pro-$p$ group. Then the set $\Sigma(\cG):=\{U_P^{ab}:P\leq\ol\cG\}$ satisfies the condition $(\ast)$. Further, an element 
$(\xi_\cA)_{\cA}\in\prod_{\cA\in\Sigma(G)}\Lambda_\cO(\cA)^\times$ belongs to $im(\theta_{\Sigma(G)})$ if and only if it satisfies all of the following
three conditions.
\begin{enumerate}
 \item For all subgroups $P, P'$ of $\ol\cG$ with $[P',P']\leq P\leq P'$, one has
\begin{equation*}
 \mathrm{nr}_P^{P'}(\xi_{U_{P'}^{ab}})=\pi_P^{P'}(\xi_{U_{P'}^{ab}}).
\end{equation*}
 \item For all subgroups $P$ of $\ol\cG$ and all $g$ in $\ol\cG$ one has $\xi_{gU_{P}^{ab}g^{-1}}=g\xi_{U_{P'}^{ab}}g^{-1}$.
 \item For every $P\in\ocG$ and $P\neq (1)$, we have
 \begin{equation*}
  \mathrm{ver}_P^{P'}(\xi_{U_{P'}^{ab}})\equiv \xi_{U_P^{ab}} \pmod{T_{P,P'}} 
 \end{equation*} 
 \item For all $P\in C(\ol\cG)$ one has $\alpha_P(\xi_{U_{P}^{ab}})\equiv\prod_{P'\in C_P(\ol\cG)}\alpha_{P'}(\xi_{U_{P'}^{ab}})\pmod{pT_P}$.
\end{enumerate}
\end{propn}
\subsection{Relation between the congruences}
\begin{propn}
 Let $\Xi\in\konep{\II[[\cG]]}$ be a $p$-adic L-function over $\II[[\cG]]$, then under the specialization maps $\phi_k$, $\phi_k(\Xi)\in\konep{\Z_p[[\cG]]}$ is a $p$-adic L-function over $\Z_p[[\cG]]$.
\end{propn}
\begin{proof}
 Consider the following commutative diagram, which is induced by a specialization map:
\begin{equation*}
 \xymatrix{
 \konep{\II[[\cG]]}\ar[r]\ar[d] & \Phi^{\ol\cG}\ar[d] \\
 \konep{\Z_p[[\cG]]}\ar[r]      & \phi^{\ol\cG}.
 }
\end{equation*}
From this commutative diagram, the congruences over $\IIG$ implies the congruences over $\Z_p[[\cG]]$ easily.
The interpolation formula of the $p$-adic L-function over $\II[[\cG]]$ also implies those over $\Z_p[[\cG]]$.
\end{proof}

\section{Application to $p$-adic L-function}\label{applications}
In this section, we generalize the \emph{torsion congruences} of Ritter-Weiss that has been used to prove the congruences and hence the main conjecture in some important cases.
\subsection{Torsion Congruences and $p$-adic L-function}
Let $F^{ab,p}$ be the maximal pro-$p$ abelian extension of $F$ that is unramified outside $p$ and $\infty$. 
We set $\mathscr G_F=\mathrm{Gal}(F^{ab,p}/F)$. 
Let $\{f_\kappa\}$ be a family of Hilbert modular forms over $F$ which is parameterized by an irreducible component $\II$  of  the universal cyclotomic deformation ring 
$\cR_F$. 
In fact, $\II$ is a finite flat algebra over $\Z_p[[\mathbb{W}]]$, where $\mathbb{W}_{}$ is the torsion free part of $\cO_\p$ as in Section \ref{adelic-hmf}.
As in Definition \ref{arithmetic}, let $\phi_\kappa:\II\lra\Z_p$ denote an arithmetic point of weight $\kappa$. Set $\mathcal{P}_\kappa=\mathrm{ker}(\phi_\kappa)$.
This induces an algebra homomorphism $\II[[\mathscr G_F]]\lra\Z_p[[\mathscr G_F]]$, which we again denote by $\phi_\kappa$,
by setting $\phi_\kappa(g)=1$ for all $g\in \mathscr G_F$. 
For the maximal ideal $\m=\langle p,\mathcal{P}_\kappa\rangle$ of $\II$, and we have a natural homomorphism 
$\II/\m[[\mathscr G_F]]\lra\FF_p[[\mathscr G_F]]$. 

Consider the  character $\mathcal N_F:\mathscr G_F\lra\Z_p$ to be the cyclotomic character.
For any integer $m\geq 0$, let $\psi_{m}:\mathscr G_F\lra\Z_p^\times$ be a character of the form $\psi\mathcal{N}_F^{m}$, where
 $\psi$ is a character of finite order. We also extend the character $\psi_m$ to the group $\mathscr G_F\times\mathbb{W}$ by setting 
$\psi_m(g)=1$ for all $g\in\mathbb W$. 

Let $\mu_F$ be a measure in $\II[[G_F]]$ that interpolates the  
critical values of each of  the representations $\ad{\rho_{f_\kappa}}\otimes\psi$ for characters $\psi$ of $\mathscr G_F$, i.e., 
 \begin{equation*}
 \int_{\mathscr G_F\times\mathbb{W}}\chi(g)\phi_\kappa(g)d\mu_F(g)=L^\ast(\ad{\rho_{f_\kappa}}\otimes\chi,0),
\end{equation*}
where $L^\ast(\ad{\rho_{f_\kappa}}\otimes\chi,0)$ involves the critical value $L(\ad{\rho_{f_\kappa}}\otimes\chi,0)$ twisted by the finite order character $\chi$, 
some archimedean periods related to $\ad{\rho_{f_\kappa}}$, and some Euler factors removed, as in the interpolation formula in \eqref{interpolate}. 
Now, we consider the measure $\mu_\kappa\in\Z_p[[\mathscr G_F]]$, defined by
\begin{equation*}
 \mu_\kappa(g)=\int_{\gamma'\in \{g\}\times\mathbb{W}}\phi_\kappa(\gamma')d\mu_F(\gamma'), \mbox{ for all } g\in \mathscr G_F.
\end{equation*}
Then
\begin{equation*}
 \int_{\mathscr G_F}\psi_m(g)d\mu_\kappa(g)=\int_{\mathscr G_F}\psi_m(g)\int_{\gamma'\in \{g\}\times\mathbb{W}}\phi_\kappa(\gamma')d\mu_F(\gamma')
 =\int_{\mathscr G_F\times\mathbb{W}}\psi_m(x)\phi_\kappa(x)d\mu_F(x).
\end{equation*}
Therefore, $\int_{\mathscr G_F}\chi(g)d\mu_\kappa(g)=L^\ast(\ad{\rho_{f_\kappa}}\otimes\chi,1)$ for all finite order characters $\chi$. 
It follows that $\mu_\kappa\in\Z_p[[\mathscr G_F]]$ is a $p$-adic L-function interpolating the special values of $\ad{f_\kappa}$ twisted by all finite order 
characters of $\mathscr G_F$.




Let $\delta^{(x)}$ be the characteristic function of a coset of an open subgroup $U$. Then $\delta^{(x)}(g)=\sum_j c_j\chi_j(g)$, for some $c_j\in\Z_p$. 
Let 
\begin{equation*}
  L^\ast(\ad{\wt f_\kappa}/F',\chi)=
                     e_p(\ad{\wt f_\kappa}/F',\chi)\mathcal{L}_p(\ad{\wt f_\kappa}/F',\chi)\dfrac{L_{(S,p)}(\ad{\wt f_\kappa}/F',\chi,0)}{\Omega_\infty(\ad{\wt f_\kappa}/F')},
 \end{equation*}
where $\mathcal{L}_p(\ad{\wt f_\kappa}/F',\chi):=\prod_{\p\mid p}\mathcal{L}_\p(\ad{\wt f_\kappa}/F',\chi)$ comes from the Euler factors at primes lying above $p$, and 
$e_p(\ad{\wt f_\kappa}/F',\chi)$ is the product of the local epsilon factors above $p$.
We define,
\begin{equation*}
 L^\ast(\ad{f_\kappa},\delta^{(x)})=\sum_jL^\ast(\ad{f_\kappa},\chi_j,0).
\end{equation*}
Then an open subgroup $U$ of $\mathscr G_F$ is said to be admissible if $\mathcal{N}_F(U)\subset 1+p\Z_p$,
and define $m_F(U)\geq1$, by $\mathcal N_F(U)=1+p^{m_F(U)}\Z_p$. Using the following lemma, which is a generalization of a result in \cite{ritter-weiss}, we relate this measure to the trace ideal.
\begin{lemma}
 $\II[[\mathscr G_F]]$ is the inverse limit of the system $\II[\mathscr G_F/U]/p^{m_F(U)}\II[\mathscr G_F/U]$, with $U$ running over the cofinal system of admissible open subgroups of $\mathscr G_F$.
\end{lemma}
\begin{proof}
 If $V$ is an admissible open subgroup of $\mathscr G_{F'}$ and $U$ is an admissible open subgroup of $\mathscr G_F$ in $ver^{-1}(V)$, then $m_F(U)\geq m_{F'}(V)-1$. 
 Consider the natural map 
 \begin{equation*}\II[[\mathscr G_F]]\lra\varprojlim_U\II[\mathscr G_F/U]/p^{m_F(U)}\II[\mathscr G_F/U].\end{equation*}
 We first show that this map is surjective. Let $(x_U)_U\in\varprojlim_U\II[\mathscr G_F/U]/p^{m_F(U)}\II[\mathscr G_F/U]$. Then for any $V\subseteq U$, consider the map 
 \begin{equation*}
  \II[\mathscr G_F/V]/p^{m_F(V)}\II[\mathscr G_F/V]\lra \II[\mathscr G_F/U]/p^{m_F(V)}\II[\mathscr G_F/U]
 \end{equation*}
and let the image of $x_V$ be denoted by $\ol{x_V}$. Note that fixing $U$ and taking the projective limit over $m_F(V)$, we have 
\begin{equation*}
 \II[\mathscr G_F/U]\cong\varprojlim_V\II[\mathscr G_F/U]/p^{m_F(V)}\II[\mathscr G_F/U].
\end{equation*}
Indeed, we have $\II[\mathscr G_F/U]/p^k\II[\mathscr G_F/U]\cong(\II/p^k)[\mathscr G_F/U]$, for every non negative integer $k$. Taking projective limit with respect to $k$, we have 
$\II[\mathscr G_F/U]=\varprojlim_k\II[\mathscr G_F/U]/p^k\II[\mathscr G_F/U]$. The lemma follows.
\end{proof}
\begin{lemma}
 The image of the measure $\mu_F\in\II[[\mathscr G_F]]$ in $\II[\mathscr G_F/U]$ is given by $\sum_{\bar g\in \mathscr G_F/U}\mu_F(.,\delta^{\bar g})\bar g$.
\end{lemma}
\begin{proof}
 Consider the measure $\mu_F\in\II[[\mathscr G_F]]$. Then $\mu_F$ is a measure on $\mathbb{W}\times \mathscr G_F$. Then, it is a standard fact that the image of $\mu_F$ in $\II[\mathscr G_F/U]$ 
 is given by $\sum_{\bar g\in \mathscr G_F/U}\mu_F(.,\delta^{(\bar g)})\bar g$, where $\mu(.,\delta^{(\bar g)})\in\II$. 
\end{proof}
Now let $F'$  be a totally real extension of $F$ contained in $F^{ab,p}$. Consider the base change Hilbert modular form over $F'$. Let $\wt f_\kappa$ denote the base-change of
$f_\kappa$ to $F'$, and let $\mathbb{J}$ be the irreducible component
to which $\wt f_\kappa$ belongs. Let $\mu_{F'}$ be the measure in $\mathbb{J}[[\mathscr G_{F'}]]$ interpolating the special values of $\ad{\wt f_\kappa}$. 
\begin{lemma}\label{invariance-measure}
Let $y$ be a coset of a $\Delta$-stable admissible open subgroup of $\mathscr G_{F'}$, where $\Delta=G(F'/F)$. Then 
\begin{equation*}
 L^\ast(\ad{\wt f_\kappa}/F',\delta^{(y)}_{F'})=L^\ast(\ad{\wt f_\kappa}/F',\delta^{(y^\gamma)}_{F'}),
\end{equation*}
for all $\gamma\in\Delta$. Further, let $\wt\mu_{F'}$ be the image of $\mu_{F'}$ under the map $\JJ[[\mathscr G_{F'}]]\lra\II[[\mathscr G_{F'}]]$. Then $\wt\mu_{F'}\in\II[[\mathscr G_{F'}]]^\Delta$.
\end{lemma}
\begin{proof}
Here $\gamma\in\Delta$ acts on $\mathscr G_{F'}$ by conjugation and trivially on $\II$. 
 It is enough to prove for finite order characters $\chi$ of $\mathscr G_{F'}$.
 Recall that 
 \begin{equation*}
  L^\ast(\ad{\wt f_\kappa}/F',\chi)=
                     e_p(\ad{\wt f_\kappa}/F',\chi)\mathcal{L}_p(\ad{\wt f_\kappa}/F',\chi)\dfrac{L_{(S,p)}(\ad{\wt f_\kappa}/F',\chi,0)}{\Omega_\infty(\ad{\wt f_\kappa}/F')},
 \end{equation*}
where $L_{(S,p)}(\ad{\wt f_\kappa}/F',\chi,0)$ is the critical value at $s=0$ of the $L$-function $L(\ad{\wt f_\kappa}/F',\chi,s)$ with the Euler factors at $S$ and those 
above $p$ removed.

Note that by induction of $L$-functions, we have,
\begin{equation*}
\begin{split}
L_{(S,p)}(\ad{\wt f_\kappa}/F',\chi,s)=& L_{(S,p)}(\ad{\wt f_\kappa}/F,\mathrm{ind}_F^{F'}\chi,s)\\
                             =& L_{(S,p)}(\ad{\wt f_\kappa}/F,\mathrm{ind}_F^{F'}\chi^\gamma,s)\\
                             =&L_{(S,p)}(\ad{\wt f_\kappa}/F',\chi^\gamma,s).
\end{split}
\end{equation*}
We also have $\mathcal{L}_p(\ad{\wt f_\kappa}/F,\chi)=\prod_{\p\mid p}\mathcal{L}_\p(\ad{\wt f_\kappa}/F',\chi)$ and therefore the equality in the lemma holds.

Let $\kappa'$ be the weight of $\wt f_\kappa$ and $\phi_{\kappa'}$ be any arithmetic specialization of weight ${\kappa'}$. Now 
\begin{equation*}
 \begin{split}
  (\mu_{F'})^\gamma(\phi_{\kappa'},\delta_{F'}^{(y)})=&\mu_{F'}(\phi_{\kappa'},\delta_{F'}^{(y^\gamma)})\\
                                                  =& L^\ast(\ad{\wt f_{\kappa'}}/F',\delta_{F'}^{(y^\gamma)})\\
                                                  =& L^\ast(\ad{\wt f_{\kappa'}}/F',\delta_{F'}^{(y)})\\
                                                  =&\mu_{F'}(\phi_{\kappa'},\delta_{F'}^{(y)}).
 \end{split}
\end{equation*}
In fact, we have $(\mu_{F'})^\gamma(\phi_{\kappa'},\chi)=\mu_{F'}(\phi_{\kappa'},\chi),$ for any finite order $\chi$ of $\mathscr G_{F'}$.
Since this holds for all the arithmetic specializations $\phi_{\kappa'}$, the measures $(\mu_{F'})^\gamma(.,\chi)=\mu_{F'}(.,\chi)$. Indeed,
the measures $(\mu_{F'})^\gamma(.,\chi)$ and $\mu_{F'}(.,\chi)$ on $\mathbb{W}$ are equal at infinitely many characters, 
they are equal. This further implies that $(\mu_{F'})^\gamma=\mu_{F'}$, for all $\gamma\in\Delta$. Since the morphism $\JJ\lra\II$ is equivariant with respect to $\Delta$,
$(\wt\mu_{F'})^\gamma=\wt\mu_{F'}$, for all $\gamma\in\Delta$. Therefore $\wt\mu_{F'}\in\II[[\mathscr G_{F'}]]^\Delta$.
\end{proof}
\begin{theorem}
Let $\mu_F\in\II[[\mathscr G_F]]$ be a measure interpolating all the critical values of each arithmetic specialization twisted by finite order characters of $\mathscr G_F$.
Similarly, let $\mu_{F'}\in\JJ[[\mathscr G_{F'}]]$ interpolating the critical values of the base change of each arithmetic specialization. Recall the trace ideal
$\sT\in\II[[\mathscr G_{F'}]]^\Delta$ generated by the elements $\Sigma_{\gamma\in\Delta}\alpha^\gamma$, with $\alpha\in\II[[\mathscr G_{F'}]]$.
Then the congruence
\begin{equation}\label{torsion-congruence}
 ver(\mu_F)\equiv\wt\mu_{F'}\mod\sT
\end{equation}
hold if and only if for every locally constant $\Z_p$-valued function $\epsilon$ of $\mathscr G_{F'}$ satisfying $\epsilon^\gamma=\epsilon$ for all $\gamma\in\Delta$
 we have the following congruences
 \begin{equation*}
  \int_{\mathscr G_F}\epsilon\circ ver(x)d\mu_F(x)\equiv \int_{\mathscr G_{F'}}\epsilon(x)d\wt\mu_{F'}(x) \mod{p\II}.
 \end{equation*}
\end{theorem}
\begin{proof}
The necessary part is clear and we need only prove the sufficient part.
Consider the components of the images of $\wt\mu_{F'}$ and $ver(\mu_{F})$ in $\II[\mathscr G_{F'}/V]/p^{m_{F'}(V)-1}$ for a $\Delta$-stable admissible open
subgroup $V$ of $\mathscr G_{F'}$. We denote the component obtained by evaluating $\wt\mu_{F'}$ at $\delta^{(y)}_{F'}$ by $\wt\mu_{F'}(.,\delta_{F'}^{(y)})$ and the component obtained by
evaluating $\mu_F$ at $\delta^{(x)}_{F'}$ by $\wt\mu_F(.,\delta_{F'}^{(y)})$. Let $U:=ver^{-1}(V)\subseteq \mathscr G_F$, then $ver(\mu_{F})$ is the image under the transfer map of the 
$U$-component of $\mu_F$. These components are the images of 
\begin{enumerate}[(i)]
 \item $\sum_{y\in \mathscr G_{F'}/V}\wt\mu_{F'}(.,\delta_{F'}^{(y)})y$,
 \item $\sum_{x\in \mathscr G_{F}/U}\mu_{F}(.,\delta_{F'}^{(x)})ver(x)$
\end{enumerate}
in $(\II[\mathscr G_{F'}/V]/p^{m_{F'}(V)-1})^\Delta$. Let $\sT(V)$ be the image of the trace ideal $\sT$ in $(\II[\mathscr G_{F'}/V]/p^{m_{F'}(V)-1})^\Delta$. We consider
the following two cases:

\paragraph{Case(i): $y$ is fixed by $\Delta$.}
In this case, $\delta_{F'}^{(y)}$ is a locally constant function as in the statement of the theorem. 

Now if $y\notin \mathrm{im}(ver)$, then $\delta_{F'}\circ ver=0$, then again by the congruence condition we have $\wt\mu_{F'}(.,\delta_{F'}^{(y)})\equiv 0\pmod{p}$.
Therefore the corresponding summands in (i) and (ii) above vanishes modulo $\sT(V)$.

\paragraph{Case(ii): $y$ is not fixed by $\Delta$:} By Lemma \ref{invariance-measure}, we have
\begin{equation*}
 \wt\mu_{F'}(.,\delta_{F'}^{(y)})=\wt\mu_{F'}(.,\delta_{F'}^{(y^\gamma)}), \forall \gamma\in\Delta.
\end{equation*}
Therefore the $\Delta$-orbit of $y$ in the sum is given by $\wt\mu_{F'}(.,\delta_{F'}^{(y)})\sum_{\gamma\in\Delta} y^\gamma$, which belongs to $\sT(V)$.
\end{proof}

Viewing the elements $\int_{\mathscr G_F}\epsilon\circ ver(x)d\mu_F(x)$ and $ \int_{\mathscr G_{F'}}\epsilon(x)d\wt\mu_{F'}(x)$ in $\II$, as measures on the weight space $\mathbb W$, they are 
determined by their values on characters $\homs{\mathbb{W}}{\ol\Q_p}$. Let $\nu:=\int_{\mathscr G_F}\epsilon\circ ver(x)d\mu_F(x)-\int_{\mathscr G_{F'}}\epsilon(x)d\wt\mu_{F'}(x)$. For
any character $\chi:\mathbb{W}\lra\ol\Q_p^\times$ let $\int_\mathbb{W}\chi(\gamma)d\nu(\gamma)\equiv 0 \pmod{p\Z_p}$. Then, it is easy to see that 
$\eta=\frac{1}{p}\nu$ defines a measure on $\mathbb{W}$, with $\eta(\chi)=\nu(\chi)/p$, and $\nu\equiv 0\pmod{p\II}$.
We therefore have the following result.
\begin{theorem}
 The congruence 
 \begin{equation*}
  ver(\mu_F)\equiv\wt\mu_{F'}\mod\sT
 \end{equation*}
hold if and only if
\begin{equation}\label{torsion-two-var}
\int_\mathbb{W}\chi(y)\int_{\mathscr G_F}\epsilon\circ ver(x)d\mu_F\equiv\int_\mathbb{W}\chi(y)\int_{\mathscr G_{F'}}\epsilon(x)d\wt\mu_{F'}\pmod{p\Z_p} 
\end{equation}
for all locally constant functions 
 $\chi$ of $\mathbb{W}$, and for every locally constant $\Z_p$-valued function $\epsilon$ of $G_{F'}$ satisfying $\epsilon^\gamma=\epsilon$ for all $\gamma\in\Delta$.
\end{theorem}
The congruences in equation \eqref{torsion-congruence}  are generalizations over $\II$ for those of the torsion congruences over 
$\Z_p$. The torsion congruences will be an important step towards proving the congruences in Theorem \ref{cong}. 


\subsection{Remarks on Torsion Congruence in Families}\label{example}
With the above theorem in place, the torsion congruences along with the validity of the Main conjecture over $\mathscr G_F$, for any finite extension $F\subset\inft{F}$ will be the subject of another exploration.
Briefly, for the adjoint representations,
as before, $f$ be a Hilbert modular form of weight $\kappa=(0,I)$ defined over $F$, and $f'$ be the base-change of $f$
to $F'$. We assume that both these modular forms are ordinary at all the primes above $p$.
Then $\int_{\mathscr G_{F}}d\mu_F(\sigma)$ is the $p$-adic L-function of $f$ in $\II$, which we denote by $L_{p,F}$. These $p$-adic L-functions have been constructed in 
\cite[\S 5.3.6]{hida-mfg} when $F=\Q$. A similar construction works over the totally real fields under some conditions \cite{rosso}. Note that there is no cyclotomic
variable in these $p$-adic L-functions. Further, $\int_{\mathscr G_{F'}}d\wt\mu_{F'}(\sigma)$ is the image of the $p$-adic 
L-function of $f'$ in $\II$. We denote this image by $\wt L_{p,F'}$. These $p$-adic L-functions generate the characteristic ideals of the dual Selmer groups 
$\sg{F}{\ad{\bosym\rho_f}\otimes\II}$ and $\sg{F'}{\ad{\bosym\rho_{f'}}\otimes\II}$. 
The torsion congruence then takes following form
\begin{equation*}
 \int_{\mathscr G_{F}}d\mu_F(\sigma)\equiv\int_{\mathscr G_{F'}}d\wt\mu_{F'}(\sigma)\pmod{p\II},
\end{equation*}
which can be written as:
\begin{equation*}\label{cong-trivial}
 L_{p,F}\equiv \wt L_{p,F'}\pmod{p\II}.
\end{equation*}
Such kinds of congruences for certain $p$-adic families have been shown by using an appropriate Eisenstein series and applying the $q$-expansion principle on them (see \cite{bouganis}). An understanding of integrality of these $p$-adic L-functions as well as the congruence ideals  of Hida seems to be required.
The $p$-adic L-functions of $\II$-adic families twisted by powers of $p$-adic Hecke characters constructed by Wan using ideas of Hida, (see \cite{wan}), also seem to satisfy the congruences. More on this and other related analytic aspects needs to be
explored.


\vspace{1cm}

\noindent
\author{Chandrakant Aribam\\
Indian Institute of Science Education and Research Mohali\\
Punjab 140306\\
aribam@iisermohali.ac.in}

\end{document}